\documentclass[11pt]{amsart}
\usepackage{marginnote}

\usepackage[hang,flushmargin]{footmisc}

\usepackage{a4wide}
\usepackage{graphicx, rotating}
\usepackage{subfigure}

\usepackage[utf8]{inputenc}
\usepackage[toc]{appendix}
\usepackage{twoopt}

\usepackage{hyperref}
\usepackage{nameref}
\usepackage{framed}
\usepackage{doi}

\usepackage{amssymb}
\usepackage{amsthm}
\usepackage{amsmath}
\numberwithin{equation}{section}
\usepackage{mathtools}
\usepackage{amsfonts}
\usepackage[foot]{amsaddr}

\usepackage[]{caption}
\usepackage{marvosym}
\usepackage{verbatim}
\usepackage{enumerate}
\usepackage{pdfsync}

\mathtoolsset{showonlyrefs}

\theoremstyle{plain}
\newtheorem{theorem}{Theorem}[section]
\newtheorem*{theorem*}{Theorem}

\theoremstyle{plain}
\newtheorem{corollary}[theorem]{Corollary}

\theoremstyle{plain}
\newtheorem{lemma}[theorem]{Lemma}
\newtheorem*{lemma*}{Lemma}

\theoremstyle{plain}
\newtheorem*{conjecture*}{Conjecture}

\theoremstyle{plain}
\newtheorem{proposition}[theorem]{Proposition}

\theoremstyle{definition}

\theoremstyle{remark}
\newtheorem*{remark}{Remark}

\theoremstyle{remark}

\theoremstyle{plain}

\theoremstyle{definition}

\providecommand{\norm}[1]{\left\lVert #1 \right\rVert}

\newcommand{\R}{\mathbb{R}}
\newcommand{\C}{\mathbb{C}}
\newcommand{\T}{\mathbb{T}}
\newcommand{\Rd}{\mathbb{R}^d}
\newcommand{\Z}{\mathbb{Z}}
\newcommand{\N}{\mathbb{N}}
\newcommand{\Lt}[1][d]{L^2(\R^{#1})}

\newcommand{\F}{\mathcal{F}}

\renewcommand{\l}{\lambda}
\renewcommand{\L}{\Lambda}

\newcommand{\D}{\mathbb{D}}

\newcommand{\vol}{\textnormal{vol}}
\renewcommand{\H}{\mathbb{H}}

\DeclareMathOperator*{\argmin}{arg \, min}

\begin{document}

\title[]{A variational principle for Gaussian lattice sums}

\author[]{Laurent Bétermin}
\address{Institut Camille Jordan - Université Claude Bernard Lyon 1, France}
\email{betermin@math.univ-lyon1.fr}

\author[]{Markus Faulhuber}
\address{Faculty of Mathematics, University of Vienna, Austria}
\email{markus.faulhuber@univie.ac.at}

\author[]{Stefan Steinerberger}
\address{Department of Mathematics, University of Washington, Seattle, WA 98195, USA}
\email{steinerb@uw.edu}

\begin{abstract}
	\vspace*{-.5\baselineskip}
	We consider a two-dimensional analog of Jacobi theta functions and prove that, among all lattices $\Lambda \subset \mathbb{R}^2$ with fixed density, the minimal value is maximized by the hexagonal lattice. This result can be seen as the inhomogeneous counterpart to Montgomery's 1988 result, which proved that the hexagonal lattice minimizes the maximum.
\end{abstract}


\vspace*{-4\baselineskip}
\maketitle

\vspace*{-.5\baselineskip}
\section{Introduction}\label{sec_intro}
\subsection{Main Result}

The purpose of this paper is to characterize optimizers for a variational problem and to discuss their applications across various fields. Let $\Lambda \subset \R^2$ be a lattice. Consider the function
\begin{equation}\label{eq:translatedtheta}
	E_\L (z;\alpha) = \sum_{\l \in \L} e^{- \pi \alpha |\lambda + z |^2} \quad \quad z \in \mathbb{R}^2, \ \alpha>0.
\end{equation}
The function $E_\L (z;\alpha)$ is simply the sum of (scaled) Gaussians centered at the points given by a (shifted) lattice: it may thus be understood as the two-dimensional, real analog of the Jacobi theta functions. Given the fundamental nature of this object, the function $E_\L (z;\alpha)$ naturally arises in many different areas of mathematics.

Throughout this paper, we will use the notation $\L_2$ to denote the hexagonal lattice, as done in the book by Conway and Sloane \cite{ConSlo98}. Indeed, the problems considered in this paper are all invariant under rotation, and we consider equivalence classes of lattices by rotation and choose a representative (see Appendix~\ref{sec_lattice}). For example, the hexagonal lattice is (up to rotation)
\begin{equation}
    \L_2 = \sqrt{2/\sqrt{3}}
    \begin{pmatrix}
        1 & 1/2\\
        0 & \sqrt{3}/2
    \end{pmatrix} \Z^2.
\end{equation}
The prefactor normalizes the lattice to have density 1. The density of a lattice is $1/\vol(\R^2/\L)$. To consider all possible scalings of a lattice, we introduce the parameter $\alpha > 0$, which will be fixed for comparison. The following result is due to Montgomery~\cite{Mon88}.

\begin{theorem*}[Montgomery, 1988]
	Among all lattices $\Lambda \subset \mathbb{R}^2$ of density 1,
	\begin{equation}
		\max_{z \in \mathbb{R}^2} E_{\L_2} (z;\alpha) \leq \max_{z \in \mathbb{R}^2} E_\L (z;\alpha) \qquad \forall \alpha > 0,
	\end{equation}
	with equality if and only if $\Lambda$ is the hexagonal lattice $\L_2$, up to equivalence by rotation.
\end{theorem*}

Thus, the hexagonal lattice minimizes the maximum. Our main result resolves the inhomogeneous dual problem of maximizing the minimum, also known as the polarization problem~\cite{HarPetSaf22}.
\begin{theorem}[Main Result]\label{thm_main}
	Among all lattices $\Lambda \subset \mathbb{R}^2$ of density 1,
	\begin{equation}
		\min_{z \in \mathbb{R}^2} E_\L (z;\alpha) \leq \min_{z \in \mathbb{R}^2} E_{\L_2} (z;\alpha) \qquad \forall \alpha > 0,
	\end{equation}
	with equality if and only if $\Lambda$ is the hexagonal lattice $\L_2$, up to equivalence by rotation.
\end{theorem}

While the problem under consideration is natural and easy to formulate, its proof is not simple. A nice aspect of Montgomery's result is that the maximum is assumed in a lattice point; in contrast, we have relatively little control over the point $z$ in which the minimum is assumed, making the proof significantly harder (cf.\ \cite{BacMouValZim26}). One important consequence of Theorem~\ref{thm_main} is that the hexagonal lattice maximizes the minimum while \textit{simultaneously} minimizing the maximum of $E_\L$. We expect this to be a rare property, reaffirming the special role that the hexagonal lattice $\Lambda_2$ plays for variational problems in $\mathbb{R}^2$.

As is common for lattice problems in dimension 2, we use the parametrization induced by the complex upper half-plane $\mathbb{H}$ (see, e.g., \cite{Mon88}) and then optimize with respect to the lattice parameters. The optimization process depends on the parameter $\alpha >0$, as does the minimum of $E_\Lambda(z;\alpha)$. Tracking the minimum is difficult, though it can be solved for exceptionally symmetric lattices, as shown by Bachoc, Moustrou, Vallentin, and Zimmermann \cite{BacMouValZim26}.

We bypass this difficulty using geometric and algebraic simplifications: in the limit $\alpha \to \infty$, the question of maximizing the minimum essentially reduces to the sphere covering problem. Thus, instead of comparing the minima of Gaussian lattice sums, we compare their values at a special point, which is close to the minimizer, but independent of $\alpha$. The trick lies in a sophisticated choice of the special point, which still depends on the lattice geometry. We design the point such that it becomes the deep hole for the hexagonal lattice (which is isomorphic to the $\mathsf{A}_2$ root lattice) when the lattice parameters vary. Due to the high symmetry of the hexagonal lattice, the deep hole is the minimizer in this case \cite{Baernstein_HeatKernel_1997} (see also \cite{BacMouValZim26}). Recall that a deep hole is the point that is furthest away from the nearest lattice points \cite{ConSlo98}.

We now give a list of corollaries to Theorem~\ref{thm_main} (which may need the combination with Montgomery's result \cite{Mon88}). We do not believe that this list is complete. The statements are deliberately kept informal at this point, with the necessary context provided in Appendix~\ref{sec_results}.

\begin{corollary}
    The conjecture of Strohmer and Beaver from \cite{StrBea03}, about the optimality of the hexagonal lattice for Gabor frame sampling, is correct.
\end{corollary}

\begin{corollary}
    Among all flat tori $\mathbb{R}^2/\Lambda$ and for all times $t > 0$, the heat kernel on the hexagonal torus $\mathbb{R}^2/\Lambda_2$ has the warmest cold spot.
\end{corollary}

\begin{corollary}
    The conjecture of Hardin, Petrache, and Saff \cite{HarPetSaf22}, about the hexagonal lattice solving the polarization problem among all lattices in the plane, is correct.
\end{corollary}

\begin{corollary}
    Among all minimal charged ionic crystals in $\mathbb{R}^2$, the hexagonal crystal structure yields the highest energy.
\end{corollary}

\subsection{Discussion of related results.}
One way of interpreting the results in this article is that we study two families of two-dimensional lattice theta functions. They are particular restrictions of Riemann theta functions and can be seen as a canonical extension of the restricted Jacobi theta $\theta_2$ and $\theta_4$, providing an alternative to the lattice theta-functions studied in \cite{BetFau20}. The functions we study complement the functions studied by Montgomery \cite{Mon88}, just as $\theta_2$ and $\theta_4$ accompany the Jacobi $\theta_3$ function. Studying these functions is mathematically interesting in its own right, and, as evidenced by applications, also eminently useful.

In recent years, lattice theta functions have attracted considerable attention, mainly motivated by the breakthrough due to Viazovska, who solved the sphere packing problem in dimension 8 \cite{Viazovska8_2017} and then, in collaboration with Cohn, Kumar, Miller, and Radchenko, solved the sphere packing problem in dimension 24 \cite{Viazovska24_2017}. The methods from these two articles have recently been further developed to solve the problem of universal optimality in dimensions 8 and 24 \cite{Coh-Via19}. The only other dimension where this problem is solved is dimension 1 \cite{CohKum07}. Despite the knowledge that the hexagonal lattice gives the unique densest sphere packing in dimension 2, the problem of its universal optimality is still open to date. Results on this problem include the seminal result by Montgomery \cite{Mon88}, which implies the universal optimality of the hexagonal lattice among lattices. Further consequences of \cite{Mon88} include optimality results for the Epstein zeta function \cite{Cas59, Dia64, Ran53} or heights \cite{Sar90} and determinants of Laplacians \cite{OsgPhiSar88} of flat tori. Results for specific classes of periodic configuration were derived in \cite{FauShaZlo24} or \cite{HarTen25}. Most recently, Lebl\'e proved that the hexagonal lattice is locally universally optimal \cite{Leb25}.

It has to be noted that the class of potentials for which the hexagonal lattice is optimal at all densities is expected to be larger than the class of completely monotone functions, see \cite{OptinonCM} and compare also \cite{FauSte19}. Recently, Luo and Wei also proved interesting crystallization results for non-Gaussian potential functions \cite{LuoWei25}. Furthermore, for a large class of potentials, the results in \cite{Mon88} imply that the hexagonal lattice is a so-called ground state among all possible lattices for the problem of energy minimization of pairwise interacting particles. See, for instance, \cite{BetTheta15,OptinonCM} for general considerations concerning this optimality problem, \cite{LBLJComput2021} for the optimality of a hexagonal lattice for Lennard-Jones type energies, and \cite{PetracheSerfatyCryst} concerning general Coulombian and Riesz interactions. Our results also generalize similarly to a large class of potentials for optimally charged lattices. Note that our work also relates to the Ho-Mueller Conjecture \cite{HoMueller} for two-component Bose-Einstein condensates. Luo and Wei recently made interesting progress regarding this conjecture in \cite{LuoWei}, where they minimize weighted sums of shifted and non-shifted lattice theta functions.

One could wonder whether our result has analogs in higher dimensions and, if so, which distinguished lattices they would be. The fundamental difference to the above works is that our results relate to the theory of sphere coverings (not packings). This potentially excludes the  $\mathsf{E}_8$ lattice from being optimal for our problem as, e.g., the $\mathsf{A}_8^*$ Voronoi lattice (the dual of the $\mathsf{A}_8$ root lattice) yields a thinner covering than the $\mathsf{E}_8$ lattice \cite[Chap.~2.1.3]{ConSlo98}. The currently best lattice covering in dimension 8 was found by Schuermann and Vallentin \cite{SchVal05}. To the best of our knowledge (see again \cite[Chap.~2.1.3]{ConSlo98}), the Leech lattice $\L_{24}$ is still nominated for being the optimal candidate, as it yields the best known covering in dimension 24.


To a large extent, our paper has been influenced by the aforementioned works. Another driving force has been a long-standing conjecture in functional analysis with applications in wireless communication and digital signal processing. A seemingly harmless question raised in \cite{StrBea03} asks for an optimal lattice sampling pattern for Gaussian Gabor systems. This question has been connected to various other fields \cite{Faulhuber_SampTA19, Faulhuber_Curious_2021, Faulhuber_Determinants_2021}. Inspired by the theory of sphere packings, Strohmer and Beaver conjectured that the hexagonal lattice (of course) should give the optimal solution \cite{StrBea03}. Strohmer and Beaver \cite{StrBea03} disprove a conjecture raised in \cite{FloAlaBer95}, which suggested that the square lattice could be optimal. This, however, has been proved to hold for rectangular lattices by two of the authors \cite{FaulhuberSteinerberger_Theta_2017}. The conjecture has been partially solved \cite{Faulhuber_Hexagonal_2018, FaulhuberSteinerberger_Theta_2017}, and our results complete the picture. In particular, our result demonstrates that not only is the theory of sphere packings important, but that a good covering must also be achievable. This is relevant to analogous problems in higher dimensions, as the lower and upper bounds may then have different optimizers.

\section{Preliminaries and notation}\label{sec_pre_notation}
\subsection{Variational problems}
The problem we consider is of a variational type. We consider a combination of minimization and maximization. We will pose the problem for the Euclidean space $\Rd$ and lattices $\L$ of density 1, i.e., $1/\vol(\Rd/\L) = 1$. It is probably in the nature of the problem that solutions for some dimensions will be easier to establish than for others, but the solutions will probably be hard to attain in any dimension.  We consider the function
\begin{equation}\label{eq_energy}
	E_\L (z;\alpha) = \sum_{\l \in \L} e^{-\pi \alpha |\l+z|^2} = \alpha^{-d/2} \sum_{\l^\perp \in \L^\perp} e^{-\frac{\pi}{\alpha} |\l^\perp|^2} e^{2 \pi i \l^\perp \cdot z}, \quad z \in \Rd, \, \alpha \in \R_+
\end{equation}
and $\L$ a lattice with dual lattice $\L^\perp$. This is characterized by (see Section \ref{sec_lattice} for more details)
\begin{equation}
	\l^\perp \in \L^\perp
	\qquad	\Longleftrightarrow \qquad
	\l^\perp \cdot \l \in \Z, \, \forall \l \in \L.
\end{equation}
The equality of the two series in \eqref{eq_energy} follows from the Poisson summation formula. Note that both formulas are $\L$-periodic.
Now, for any fixed $\alpha >0$, we seek to maximize the minimal value of $E_\L$. This is, for any positive fixed value of $\alpha$ we seek to find the universal constant
\begin{equation}\label{eq_A}
	\mathcal{A}(\alpha) = \max_{\L} \min_{z \in \Rd} E_\L(z; \alpha).
\end{equation}
This problem is somewhat opposite to the popular problem of finding universally optimal structures in $\Rd$ (see \cite{CohKum07}, \cite{Coh-Via19}), which in the lattice case can be formulated as follows. Find, for any fixed $\alpha > 0$, the value of the universal constant
\begin{equation}\label{eq_B}
	\mathcal{B}(\alpha) = \min_{\L} \max_{z \in \Rd} E_\L(z;\alpha).
\end{equation}
Of course, the combined minimization and maximization procedures in \eqref{eq_A} and \eqref{eq_B} can be carried out over arbitrary point sets $\Gamma$ of density 1. The problem considered in \eqref{eq_B} was recently solved for periodic configurations in dimension 1 \cite{FauSte24}. For general $\Gamma$, the problem of universal optimality involves a limiting procedure to define the energy of a class of radial completely monotone potentials of squared distance (see \cite{CohKum07, Coh-Via19} for the details). By using \eqref{eq_Bernstein_Widder} from the Bernstein-Widder theorem, we can pass to Gaussian interaction potentials. For a periodic point set $\Gamma \subset\Rd$ of density 1, the energy minimization problem is to find, for any $\alpha > 0$, the universal constant ($E_\Gamma$ may be defined as in \cite{CohKum07, Coh-Via19})
\begin{equation}
	\mathbf{B}(\alpha) = \min_{\Gamma} E_\Gamma(0;\alpha).
\end{equation}
If we assume a lattice structure, then
\begin{equation}\label{eq_z_max}
	E_\L(0;\alpha) = \max_{z \in \Rd} E_\L(z;\alpha).
\end{equation}
This is an easy consequence of the Poisson summation formula and the triangle inequality;
\begin{align}
	E_\L(0,\alpha) & = \alpha^{-d/2} \sum_{\l^\perp \in \L^\perp} e^{-\frac{\pi}{\alpha} |\l^\perp|^2}
	= \alpha^{-d/2} \sum_{\l^\perp \in \L^\perp} \left| e^{-\frac{\pi}{\alpha} |\l^\perp|^2} \right| \left| e^{2 \pi i \l^\perp \cdot z} \right|\\
	& \geq \alpha^{-d/2} \sum_{\l^\perp \in \L^\perp} e^{-\frac{\pi}{\alpha} |\l^\perp|^2} e^{2 \pi i \l^\perp \cdot z} = E_\L(z;\alpha), \quad \forall \alpha > 0.
\end{align}
Note that for a non-lattice structure $\Gamma$, the maximum may not be assumed on $\Gamma$ (Figure~\ref{fig_non_lattice}).
\begin{figure}[ht]
    \begin{framed}
    	\vspace{-0.25cm}
    	\includegraphics[width=.65\textwidth]{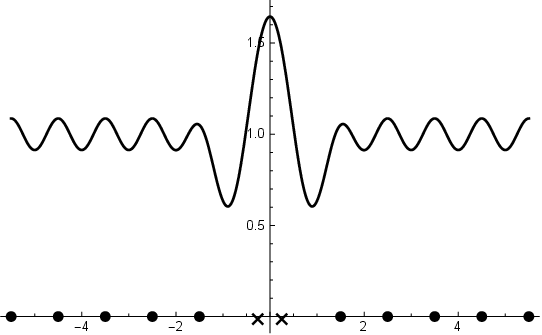}
    	\vspace{-0.15cm}
    	\caption{The function $E_\Gamma(z;1)$ where $\Gamma$ is the integer lattice shifted by $\frac{1}{2}$ and the points $\{\pm \frac{1}{2}\}$ have been moved to $\{\pm \frac{1}{4}\}$. The maximum is attained at the origin $0 \notin \Gamma$.}\label{fig_non_lattice}
    \end{framed}
\end{figure}
Energy minimization connects to the theory of optimal sphere packings. Indeed, for the limit $s \to \infty$ of the potential $r \mapsto r^{-s}$, the sphere packing problem can be derived from the universal optimality problem. For finding the value $\mathcal{A}$ in \eqref{eq_A}, the theory of optimal sphere coverings enters the scene, by using the limit of the potential $r \mapsto r^{-s}$ again. We refer to \cite{ConSlo98} and \cite{SchVal04} for the theory of sphere packing and covering. Noting that we deal with a sphere covering problem is a central point in our proofs. Already for lattices, finding the minimizer
\begin{equation}
	z_\L^-(\alpha) = \argmin_{z \in \Rd} E_{\L}(z;\alpha),
\end{equation}
depends on the parameter $\alpha$ and the free parameters of the $d$ vectors spanning the lattice, i.e., the $(d+1)d/2-1$ parameters defining the lattice geometry: by a $QR$ decomposition, we may always assume that the spanning vectors form an upper triangular matrix, which has $(d+1)d/2$ free parameters. We lose one free lattice parameter because we assume density~1, but we have the additional parameter $\alpha$ instead. Already for $d=2$, this means that the minimizing point depends on 3 parameters. This is in great contrast to Montgomery's work \cite{Mon88} and to the problem of finding the maximizing point $z$ in \eqref{eq_z_max} if we restrict to lattices.

\smallskip

We rely on the following result by Baernstein \cite{Baernstein_HeatKernel_1997}, who showed that the minimum of $E_{\L_2}$ is always attained in the circumcenter of a fundamental triangle, i.e., a deep hole of $\L_2$.
\begin{theorem*}[Baernstein]\label{thm:Baernstein}
	Let $\L_2$ denote the hexagonal lattice. Then, for all $\alpha > 0$,
	\begin{equation}
	E_{\L_2} (z;\alpha) = \sum_{\l \in \L_2} e^{- \pi \alpha |\lambda + z |^2}
	\end{equation}
assumes its minimum at the circumcenter of the fundamental triangle of $\L_2$.
\end{theorem*}

\subsection{Notation}\label{sec_notation}
We will write the classical theta functions as
\begin{equation}
	\vartheta(\beta;t) = \sum_{k \in \Z} e^{-\pi t (k+\beta)^2}
	\quad \text{ and } \quad
	\widehat{\vartheta}(\beta;t) = \sum_{k \in \Z} e^{-\pi t k^2} e^{2 \pi i k \beta},
\end{equation}
where $\beta \in \R$ and $t \in \R_+$. They fulfill the functional equation
\begin{equation}\label{eq:theta_functional}
	\sqrt{\tfrac{1}{t}} \vartheta(\beta; \tfrac{1}{t}) = \widehat{\vartheta}(\beta;t),
\end{equation}
which is easily seen by using the Poisson summation formula.

The product representation for $\widehat{\vartheta}$ \cite{SteSha_Complex_03, WhiWat69}, which can be used for $\vartheta$ by applying \eqref{eq:theta_functional}, is
\begin{align}
	\widehat{\vartheta}(\beta;t) & = \prod_{k \geq 1} (1 - e^{-2 \pi k t})(1 + e^{-(2k-1)\pi t} e^{2 \pi i \beta}) (1 + e^{-(2k-1)\pi t} e^{-2 \pi i \beta})\\
	& = \prod_{k \geq 1} (1 - e^{-2 \pi k t})(1 +  \cos(2 \pi \beta) 2 e^{-(2k-1)\pi t} + e^{-(4k-2)\pi t}).
\end{align}

Passing to dimension 2, for an upper triangular matrix
\begin{equation}
    M = y^{-1/2}
    \begin{pmatrix}
        1 & x\\
        0 & y
    \end{pmatrix} \in SL(2,\R),
\end{equation}
let $\L = \L(x,y) = M \Z^2$ be the generated lattice with parameters $(x,y) \in \R \times \R_+$ (cf.~\eqref{eq_lattice_M}). So, $\tau = x + i y \in \H$ is an element of the upper half plane. Furthermore, we can restrict to
\begin{equation}
	D_+ = \{ \tau \in \H \mid 0 \leq x \leq \tfrac{1}{2}, \, x^2+y^2 \geq 1 \}.
\end{equation}
Details on lattice parametrization by elements $\tau \in \H$ and the restriction to $D_+$ are given in Section \ref{sec_lattice}. For the moment, we mention that the hexagonal lattice $\L_2$ is parametrized by $e^{2 \pi i/3} = (1+i\sqrt{3})/2$ and that we call lattices with $x = 0$ rectangular. For $\L = M \Z^2$ with parameters from $D_+$, $b_\L = M b$, we study the following, family of lattice theta functions;
\begin{equation}\label{eq_theta_b}
	\theta_\L(b_\L;\alpha) = \sum_{\l \in \L} e^{-\pi \alpha |\l+Mb|^2} = \sum_{k,l \in \Z} e^{-\tfrac{\pi \alpha}{y} ((k+b_1)^2 + 2 x (k + b_1) (l + b_2) + (x^2+y^2) (l + b_2)^2)},
\end{equation}
\begin{equation}\label{eq_theta_hat_b}
	\widehat{\theta}_\L(b_\L;\alpha) = \sum_{\l \in \L} e^{-\pi \alpha |\l|^2} e^{2 \pi i \sigma(\l,Mb)} = \sum_{k,l \in \Z} e^{-\frac{\pi \alpha}{y} (k^2 + 2 x k l +(x^2+y^2)l^2)} e^{2 \pi i (k b_2 - l b_1)},
\end{equation}
which are lattice theta functions shifted by $b_\L=Mb=M(b_1,b_2)$ or charged lattice theta functions with charge $M(b_1,b_2)$. We remark that $\sigma(. \, , .)$ denotes the standard (skew-symmetric) symplectic form and that we use the symplectic Fourier transform and the symplectic Poisson summation formula to obtain the functional equation
\begin{equation}\label{eq_functional}
	\theta_\L(b_\L;\alpha) = \tfrac{1}{\alpha} \, \widehat{\theta}_\L (b_\L;\tfrac{1}{\alpha}), \quad b_\L \in \R^2, \; \alpha > 0.
\end{equation}
Details on these methods are provided in Section \ref{sec_tfa}, but for the moment, we note that we primarily use an additional 90-degree rotation relative to the standard formulas. The coordinates $b_1$ and $b_2$ implicitly depend on the lattice parameters $x$ and $y$;
\begin{equation}\label{eq_b}
	b_1 = b_1(x,y) = \frac{x+(1-x)4y^2}{8 y^2}
	\qquad \textnormal{ and } \qquad
	b_2 = b_2(x,y) = \frac{4y^2-1}{8 y^2}.
\end{equation}
In general, we do not have a proper interpretation of the geometric meaning of $b_\L$. Actually, we first considered the point $a_\L=Ma$ with $a = (a_1,a_2)$ and
\begin{equation}\label{eq_a}
	a_1 = a_1(x,y) = \frac{(1-x)(x^2+y^2)}{2 y^2}
	\qquad \textnormal{ and } \qquad
	a_2 = a_2(x,y) = \frac{-x+x^2+y^2}{2y^2}.
\end{equation}
The point $a_\L = M a$ has a nice geometric interpretation, namely that we evaluate $\theta_\L$ at a deep hole, making it dependent on the lattice geometry, defined $(x,y) \in D_+$, but not on $\alpha > 0$. Also, it is crucial to note that the coordinate vectors $b$ (as well as $a$), obtained from \eqref{eq_b} and \eqref{eq_a} respectively, have been constructed using the lattice base $M$. Thus, we can plug \eqref{eq_b} into \eqref{eq_theta_b} and \eqref{eq_theta_hat_b}, as long as $\tau \in D_+$. Adapting the formulas when $\tau \notin D_+$ is straightforward, requiring only a linear coordinate change. The necessary information on the modular group is given in Appendix~\ref{sec_lattice}, and the precise statement is given in Proposition \ref{pro:change_B}.

There is a nice algebraic relation between $a_1$ and $a_2$, which we will exploit in order to simplify the problem algebraically, namely, we have
\begin{equation}
	a_1 + x \, a_2 = \tfrac{1}{2}.
\end{equation}
Now, the coordinates of the point $b_\L$ have been designed in such a way that
\begin{equation}\label{eq_b_prop}
	b_1 + x \, b_2 = \tfrac{1}{2}, \; b_2(x,y) = b_2(y), \; \forall x \in [0, \tfrac{1}{2}]
	\quad \text{ and } \quad
	b(\tfrac{1}{2},y) = a(\tfrac{1}{2},y).
\end{equation}
\begin{figure}[ht]
    \begin{framed}
	\hspace*{\fill}
	\includegraphics[width=.425\textwidth]{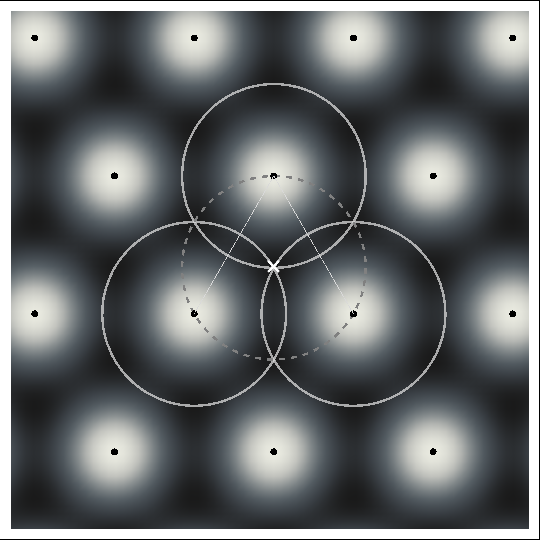}
	\hspace*{\fill}
    \hspace*{\fill}
	\includegraphics[width=.425\textwidth]{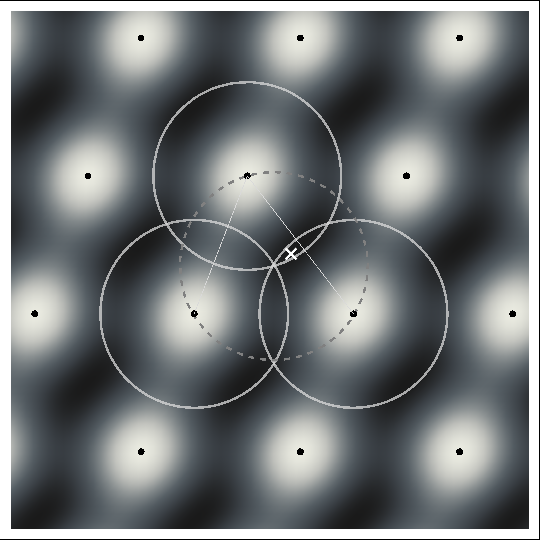}
	\hspace*{\fill}
	\caption{The hexagonal lattice $\Lambda_2$ (left) and a non-hexagonal, non-rectangular lattice (right), each with 3 covering circles centered at lattice points. For $E_{\L_2}(z;\alpha)$, the minimum among all $z$ is attained exactly at the circumcenter for all $\alpha > 0$, and the closest lattice points are a covering radius away. For general lattices, the minimum of $E_{\L}(z;\alpha)$, marked by $\times$, is ``close" to the circumcenter and varies with $\alpha$.}\label{fig_hex_covering}
    \end{framed}
\end{figure}

The first two properties are important as they allow us to lose the implicit $x$-dependence of the point $b_\L$ (however, there is still an $x$-dependence in the lattice geometry). The last property of \eqref{eq_b_prop} is important for two reasons. The hexagonal lattice lies on the boundary line $x = 1/2$ of $D_+$, and the circumcenter $a_\L$ is heuristically close to the point at which the function assumes its minimum. For rectangular lattices (compare Janssen's work \cite[Sec.~6]{Jan96}) and the hexagonal lattice, the minimizing point is the circumcenter of the fundamental triangle, and for general lattices it approaches the circumcenter in the limit $\alpha \to \infty$ (compare Figure \ref{fig_hex_covering}). The second, and maybe more important, point has already been mentioned: for the hexagonal lattice, the circumcenter yields the minimizer. We prove the following result.
\begin{theorem}\label{thm_theta}
	Let $M \in SL(2,\R)$ have parameters $(x,y) \in \R \times \R_+$ such that $\tau = x+iy \in D_+$ and $\alpha > 0$. Let $b_\L = M b$, $b(x,y) = b = (b_1, b_2)$ be defined by \eqref{eq_b} and $\theta_\L(b_\L;\alpha)$ and $\widehat{\theta}_\L(b_\L;\alpha)$ be defined by \eqref{eq_theta_b} and \eqref{eq_theta_hat_b}, respectively. Then
	\begin{equation}
		\theta_{\L_2}(b_{\L_2};\alpha) \geq \theta_\L(b_\L;\alpha)
		\quad \text{ and } \quad
		\widehat{\theta}_{\L_2}(b_{\L_2};\alpha) \geq \widehat{\theta}_\L(b_\L;\alpha), \quad \forall \alpha > 0,
	\end{equation}
	with equality if and only if $\L$ is the hexagonal lattice $\L_2$, up to equivalence by rotation.
\end{theorem}
The \nameref{thm_main} readily follows as $b_{\L_2}$ gives the minimum of $E_{\L_2}$ for the hexagonal lattice $\L_2$ and any other lattice contains a point $b_\L$ where the value of $E_\L$ is already smaller.

\subsection{Outline of the proof}

The fact that $b_\L = b_\L(x,y)$ is not static makes the proof somewhat delicate and demands good bookkeeping of estimates at several points. The general idea of the proof of Theorem \ref{thm_theta} is as follows. We simplify the problem by using the functional equation \eqref{eq_functional}. This allows us to consider $\theta_\L(b_\L,\alpha)$ and $\widehat{\theta_\L}(b_\L;\alpha)$ only for $\alpha \geq 1$ and not the whole range of $\alpha > 0$. Therefore, we find dominant terms giving the variations of the expressions as well as (negligible) exponentially small tails. The rest of this work is concerned with proving Theorem \ref{thm_theta}, as outlined below.

\medskip

\textit{The hexagonal lattice is a critical point}: We start by showing that the hexagonal lattice is a critical point in the space of lattices. This is achieved by using geometric simplifications due to the existing symmetry. We also establish general properties of a point $c_\L \in \R^2$ such that $\L_2$ is still a critical point of $E_\L(c_\L;\alpha)$, i.e.,
\begin{equation}
	\begin{pmatrix}
		\partial_x \theta_\L(c_\L;\alpha)\\
		\partial_y \theta_\L(c_\L;\alpha)
	\end{pmatrix} \Big|_{(x,y)=(1/2, \sqrt{3}/2)} =
	\begin{pmatrix}
		0 \\ 0
	\end{pmatrix}
	\quad \text{ and } \quad
	\begin{pmatrix}
		\partial_x \widehat{\theta}_\L(c_\L;\alpha)\\
		\partial_y \widehat{\theta}_\L(c_\L;\alpha)
	\end{pmatrix} \Big|_{(x,y)=(1/2, \sqrt{3}/2)} =
	\begin{pmatrix}
		0 \\ 0
	\end{pmatrix}.
\end{equation}
The generality of $c_\L$ suggests that we have some room for our estimates.

\medskip

\textit{Analysis of the $x$-direction}: Then, for $\alpha \geq 1$, we show that the $x$-derivative of both functions is positive in a region containing $D_+$. More precisely, for $y \geq 1/ \sqrt{2}$ fixed, we show that
\begin{equation}
	\partial_x \theta_\L(b_\L;\alpha) > 0
	\quad \text{ and } \quad
	\partial_x \widehat{\theta}_\L(b_\L;\alpha) > 0, \qquad x \in (0,\tfrac{1}{2}).
\end{equation}
This yields that the maximizing lattice must be found on the boundary line $x = \frac{1}{2}$ for $y \geq \frac{\sqrt{3}}{2}$ (the line $x = 0$ contains minimizing lattices for fixed $y$). The proofs rely on an intuition coming from a result of Montgomery \cite{Mon88}, which roughly says that, for any fixed $t >0$, the derivative of the classical theta function behaves like a sine-function (up to some bounds).
\begin{equation}
	\vartheta' (\beta;t) \asymp - \sin(2 \pi \beta).
\end{equation}
The lower and upper bounds of the above estimate depend uniformly on $t$.

\begin{figure}[ht]
    \begin{framed}
    	\includegraphics[width=.75\textwidth]{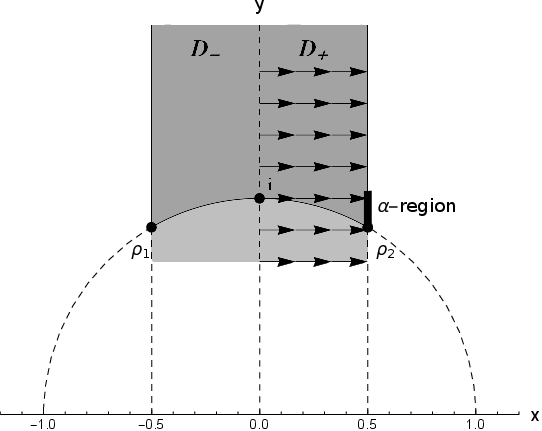}
    	\caption{The upper half-plane $\H$ and the fundamental domain $D$, split into $D_+$ and $D_-$. The arrows indicate the direction in which $E_\L(b_\L;\alpha)$ is growing, for $y \geq 1/\sqrt{2}$ and all $\alpha > 0$. The point $i$ corresponds to the square lattice, whereas $\rho_1$ and $\rho_2$ yield the hexagonal lattice. As a function of $y$, $E_\L(b_\L;\alpha)$ is concave in an $\alpha$-neighborhood of $\rho_2$ and smaller than $E_{\L_2}(b_\L;\alpha)$ outside. The hexagonal lattice is a critical point in $D_+$ and hence yields the unique maximum.}
    \end{framed}
\end{figure}

\textit{Analysis on the line $x = \frac{1}{2}$}: The final step is to show that, for $x = \frac{1}{2}$, both functions are concave in the $y$-direction in a neighborhood of $\L_2$ (where the size of the neighborhood depends on $\alpha$), followed by the observation that there is no maximum outside this neighborhood. The rough strategy can be summed up as follows: we have a function of the form
\begin{equation}
	h_\alpha(y) = \sum_{k,l \in \Z} e^{- \pi \alpha \, \phi_{k,l}(y)} \psi_{k,l}(y),
\end{equation}
where $\alpha \geq 1$, $\phi_{k,l}(y)$ is a positive definite quadratic form in $k$ and $l$ and $\psi_{k,l}$ may be highly oscillating, but bounded. By bounding $\phi_{k,l}(y)$, $\phi_{k,l}'(y)$ and $\phi_{k,l}''(y)$ for $y \geq \sqrt{3}/2$ we establish bounds on $h_\alpha(y)$, $h_\alpha'(y)$ and $h_\alpha''(y)$. These estimates are then combined to show that there is an $\alpha$-region near $y = \sqrt{3}/2$ where $h_\alpha(y)$ is concave in $y$ and that there is no global maximum outside this region. The subtleties lie in a careful asymptotic analysis as $\alpha$ becomes large.

\medskip

Even though we follow more or less the same strategy in both cases, the proofs for $\theta_\L$ and $\widehat{\theta}_\L$ require somewhat different strategies to deal with different types of complications. As noted, our Main Result follows immediately since $\theta_{\L_2}(z_{\L_2}^-;\alpha) = \theta_{\L_2}(b_\L;\alpha) \geq \theta_\L(b_\L;\alpha) \geq \theta_\L(z_\L^-;\alpha)$.
\begin{flushright}
	$\diamond$
\end{flushright}

\section{The hexagonal lattice is a critical point}
\subsection{Critical lines and critical points}
We will show that the boundaries of the right half of the fundamental domain are critical for our theta functions in terms of directional derivatives. For the vertical boundaries, this can be derived from the algebraic structure of $\widehat{\theta}_\L(b_\L;\alpha)$ and then using the Poisson summation formula to derive the result for $\theta_\L(b_\L;\alpha)$. In the sequel, we will be a bit sloppy in the formulation of our results and simply write $(x,y) \in D_+$ when we actually mean that the lattice parameters $(x,y)$ are coming from $\tau \in D_+$. The following result shows that the hexagonal lattice is a critical point in the space of lattices.
\begin{lemma}\label{lem_hexagon_critical}
	Let $(x,y) \in D_+$ and $b(x,y)$ as in \eqref{eq_b}. Then, in the space of lattices, for any fixed $\alpha > 0$, the functions
	\begin{equation}
		\theta_\L (b_\L;\alpha)
		\quad \text{ and } \quad
		\widehat{\theta}_\L (b_\L;\alpha)
	\end{equation}
	have a critical point in the hexagonal lattice $\L_2$, i.e., for
	\begin{equation}
		(x,y) = \left( \tfrac{1}{2} , \tfrac{\sqrt{3}}{2} \right).
	\end{equation}
	In other words, the gradient of both $\theta_\L(b_\L;\alpha)$ and $\widehat{\theta}_\L(b_\L;\alpha)$ vanishes in $(x,y) = \left( \frac{1}{2} , \frac{\sqrt{3}}{2} \right)$.
\end{lemma}
\begin{proof}
	We will only show the result for $\widehat{\theta}_\L(b_\L,\alpha)$ for any $\alpha > 0$. The result for $\theta_\L(b_\L,\alpha)$ follows from the functional equation \eqref{eq_functional}. As we know that $\theta_\L(b_\L,\alpha)$ is real-valued, this has to be true for $\widehat{\theta}_\L(b_\L,\alpha)$ as well. Hence, the complex exponential can only contain the real (cosine) part. This also follows from symmetry considerations. We have
	\begin{equation}
		\widehat{\theta}_\L(b_\L;\alpha) = \sum_{k,l \in \Z} e^{- \frac{\pi\alpha}{y} \left(k^2 + 2 x k l + (x^2 + y^2) l^2 \right)} \cos \left( 2\pi  (k b_2(x,y) - l b_1(x,y) ) \right)
	\end{equation}
	Now, we compute the derivatives and evaluate at $(x,y) = (1/2, \sqrt{3}/2)$, which gives
	\begin{align}
		\partial_x \widehat{\theta}_\L(b_\Lambda;\alpha) \Big|_{(x,y)=(\frac{1}{2},\frac{\sqrt{3}}{2})} = -\frac{2\pi}{3}\sum_{k,l \in \Z} & e^{-\frac{2 \pi \alpha}{\sqrt{3}} \left(k^2+k l+l^2\right)} \times \\
		& \; l \left(\sqrt{3} \, \alpha (2 k+l) \cos \left(\tfrac{2}{3} \pi  (k-l)\right)+\sin \left(\tfrac{2}{3} \pi  (k-l)\right)\right).
	\end{align}
	Clearly, for $l = 0$, the whole expression vanishes. The cosine and sine expressions only take the following values
	\begin{equation}
		\cos(\tfrac{2}{3} \pi (k-l)) \in \{\pm \tfrac{1}{2}, 1 \}
		\quad \text{ and } \quad
		\sin(\tfrac{2}{3} \pi (k-l)) \in \{\pm \tfrac{\sqrt{3}}{2}, 0 \}, \quad k,l \in \mathbb{Z}.
	\end{equation}
	So, there will be cancellations in the cosine and sine parts, independent of the value $\alpha$.
	We note that the quadratic form $k^2+k l+l^2$ is invariant under the transformation $k \mapsto -k-l$. So, for $l$ odd, we pair $k$ with $-k-l$ to see that the expression vanishes and for $l$ even, we also pair $k$ with $-k-l$, but we have to treat the case $k = - \frac{l}{2}$ separately. However, for $l$ even and $k = - \frac{l}{2}$, the respective term vanishes. Hence
	\begin{equation}
		\partial_x \widehat{\theta}_\L(b_\L;\alpha) \Big|_{(x,y)=(\frac{1}{2},\frac{\sqrt{3}}{2})} = 0.
	\end{equation}
	With the same arguments, we see that
	\begin{align}
		\partial_y \widehat{\theta}_\L(b_\L;\alpha) \Big|_{(x,y)=(\frac{1}{2},\frac{\sqrt{3}}{2})} = \frac{2 \pi}{3} \sum_{k,l \in \Z} & e^{-\frac{2 \pi  \alpha}{\sqrt{3}} \left(k^2+k l+l^2\right)} \times \\
		& \hspace{-0.15cm}\left(\alpha \left(2 k^2+2 kl-l^2\right) \cos \left(\tfrac{2}{3} \pi (k-l)\right)- \tfrac{(2 k+l)}{\sqrt{3}} \sin \left(\tfrac{2}{3} \pi (k-l)\right)\right)
	\end{align}
	indeed evaluates to 0.
\end{proof}
We note that the above proof also holds verbatim if the point $b_\L$ defined by \eqref{eq_b} is replaced by the circumcenter of the fundamental triangle, i.e., by $a_\L$ defined by \eqref{eq_a}. This suggests that there is some kind of perturbation result for our problem in the background (see Appendix~\ref{sec_stability}), and the above lemma actually easily extends to the following result.
\begin{lemma}\label{lem_critical_line}
	Let $(x,y) \in D_+$ and let $c_\L = M c$ with $c = c(x,y) = \left(c_1(x,y), c_2(x,y)\right)$ have the following properties;
	\begin{align}
		c(\tfrac{1}{2},y) = a(\tfrac{1}{2},y), \quad c_1 + x \, c_2 = \tfrac{1}{2}, \quad \text{ and } \quad \partial_x c_2 |_{x=\frac{1}{2}} = 0.
	\end{align}
	Then, in the space of lattices, for any fixed $\alpha > 0$, the functions
	\begin{equation}
		\theta_\L (c_\L;\alpha)
		\quad \text{ and } \quad
		\widehat{\theta}_\L (c_\L;\alpha)
	\end{equation}
	have a critical point in the hexagonal lattice $\L_2$, i.e., for
	\begin{equation}
		(x,y) = \left( \tfrac{1}{2} , \tfrac{\sqrt{3}}{2} \right).
	\end{equation}
\end{lemma}
\begin{proof}
	Since $c(1/2,y) = a(1/2, y)$ it follows from Lemma \ref{lem_hexagon_critical} that the gradient at $(1/2, \sqrt{3}/2)$ vanishes in the $y$-direction. Hence, we only need to prove the result for the $x$-direction. This time, we prove the result only for $\theta_\L(c;\alpha)$. After some algebraic simplification, we get
	\begin{align}
		\theta_\L (c_\L;\alpha) = \sum_{l \in \Z} e^{-\pi \alpha y (l+c_2)^2} \underbrace{\sum_{k_ \in \Z} e^{-\frac{\pi \alpha}{y} (k + \frac{1}{2} + xl)^2}}_{\vartheta(\frac{1}{2} + x l; \frac{\alpha}{y})}.
	\end{align}
	As $c_2$ implicitly depends on $x$, the partial derivative with respect to $x$ is
	\begin{align}
		\partial_x \theta_\L (c_\L;\alpha) & = - \sum_{l \in \Z} 2 \pi \alpha y (l+c_2) \, c_2' \, e^{-\pi \alpha y (l+c_2)^2} \vartheta(\tfrac{1}{2} + x l; \tfrac{\alpha}{y})\\
		& \quad + \sum_{l \in \Z} e^{-\pi \alpha y (l+c_2)^2} \, l \, \vartheta'(\tfrac{1}{2} + x l; \tfrac{\alpha}{y}).
	\end{align}
	Here, $c_2'$ and $\vartheta'$ denote differentiation with respect to the first argument. Now, by assumption $c_2'$ vanishes for $x = 1/2$. For $x = 1/2$ the function $\vartheta'$ only takes integer or half integer values in the first argument. By using the product representation (see also \cite{Mon88}) it is easy to see that $\vartheta'$ vanishes at exactly the integers and half-integers (independently from the parameter in the second argument). Therefore
	\begin{equation}
		\partial_x \theta_\L(c_\L;\alpha) \Big|_{x=\frac{1}{2}} = 0, \quad \forall y \geq \frac{\sqrt{3}}{2}, \, \forall \alpha > 0.
	\end{equation}
\end{proof}

\begin{remark}
	It is immediate from the above proof that actually the right boundary of $D_+$ is critical with respect to the $x$-derivative. If one additionally assumes $\partial_x c_2 |_{x=0}$, then, with the same proof as above, one also finds that the line $x = 0$ is critical in that sense. In particular, the assumptions of Lemma \ref{lem_critical_line} are met for the circumcenter $a_\L$ defined by \eqref{eq_a}.
\end{remark}

\section{Analysis of the \texorpdfstring{$x$}{x}-derivative}
We will now prove a result for the directional derivative. This needs some preparation and auxiliary results. Note that we can write $\theta_\L(b_\L;\alpha)$ in the following way;
\begin{equation}
	\theta_\L(b_\L;\alpha) = \sum_{l \in \Z} e^{-\tfrac{\pi \alpha}{y} y^2 (l+b_2)^2} \underbrace{\sum_{k \in \Z} e^{-\tfrac{ \pi \alpha}{y} \left( k + (b_1 + x (l + b_2)) \right)^2}}_{\vartheta \left( b_1 + x(l+b_2); \tfrac{\alpha}{y} \right)}.
\end{equation}

We will assume that $y > 1/\sqrt{2}$ is fixed, and hence write $b_1(x)$ and $b_2(x)$ or simply $b_1$ and $b_2$. Also, for any $y \in \R_+$, we note the following simplification;
\begin{equation}
	b_1(x) + x b_2(x) = \frac{1}{2}.
\end{equation}
Also, we write
\begin{equation}
	b_2 = \frac{1}{2} - \frac{1}{8y^2} = \frac{1}{2} - r
	\quad \text{ with } \quad
	r = r(y) = \frac{1}{8y^2}.
\end{equation}
Therefore, we may write
\begin{align}
	\theta_\L(b_\L;\alpha) & = \sum_{l \in \Z} e^{-\tfrac{\pi \alpha}{y} y^2 (l+ \frac{1}{2} - r)^2} \sum_{k \in \Z} e^{-\tfrac{ \pi \alpha}{y} \left( k + \frac{1}{2} +x l \right)^2}\\
	& = \sum_{l \in \Z} e^{- \pi \alpha y (l+\frac{1}{2} - r)^2} \vartheta(\tfrac{1}{2} + x l; \tfrac{\alpha}{y}).\label{eq_theta_b_simple}
\end{align}

\subsection{Auxiliary results}
The key idea in the proof will be that
\begin{equation}\label{eq_thetaprime_sine}
	\vartheta'(\beta;t) \asymp - \sin(2 \pi \beta),
\end{equation}
where the prime denotes (as is usual) differentiation of $\vartheta$ with respect to the first argument. The equation above should be understood in the sense that $\vartheta'(\beta;\alpha)$ can be bounded below and above by the sine function, up to some (negative) constant depending on $\alpha$. More precisely, we use the following auxiliary results, which can also be found in Montgomery's article \cite{Mon88}.
\begin{lemma}[Montgomery]\label{lem_aux_theta}
	For $t > 0$, the functions $\vartheta(\beta;t)$ and $\widehat{\vartheta}(\beta;t)$ are even and periodic with period 1 as functions of $\beta$. Furthermore, they take their global maxima for $\beta \in \Z$ and their global minima for $\beta \in \Z + \frac{1}{2}$. Furthermore, the functions are strictly decreasing in $\beta$ on the interval $(0,1/2)$ and strictly increasing on the interval $(1/2,1)$.
\end{lemma}
\begin{proof}
	The result quickly follows from the product representation of $\widehat{\vartheta}(\beta;t)$ and the functional equation $\vartheta(\beta;t) = \widehat{\vartheta}(\beta;t)/\sqrt{t}$
	\begin{equation}
		\widehat{\vartheta}(\beta;t) = \prod_{m \geq 1} (1 - e^{-2 m \pi t})(1 + 2 e^{-(2m-1)\pi t} \cos(2 \pi \beta) + e^{-(4m-2) \pi t}).
	\end{equation}
	Now, observe that the only $\beta$-dependence of the function comes from the cosine part (which does not even involve the index $m$), so the result readily follows.
\end{proof}
The proof of the above lemma shows that the $\beta$-dependence of $\vartheta$ is only given by a cosine-function and, also, $\vartheta$ itself oscillates around 1 like the cosine-function. Moreover, the lemma tells us that $\vartheta(\frac{1}{2} + x l; \alpha/y)$ has critical points for $x \in \{0, 1/2 \}$. For $x=0$ we always have a minimum, but for $x = 1/2$ we have a competition of minima and maxima. The difficulty in \eqref{eq_theta_b_simple} is to control the rapid oscillations for $|l|$ large, but then the decay given by the Gaussian factor is helpful. The next two results connect $\vartheta'$ and the sine-function as suggested by \eqref{eq_thetaprime_sine}.

\begin{lemma}[Montgomery]\label{lem_aux_Q}
	Let $\beta \in \R$ and $t > 0$ (fixed). Then the function
	\begin{equation}
		Q(\beta;t) = - \frac{\partial_\beta \, \vartheta(\beta;t)}{\sin(2 \pi \beta)}
	\end{equation}
	is a positive, even, periodic function of $\beta$ with period 1. Also, as a function of $\beta$ it is strictly decreasing on $(0, \frac{1}{2})$ and increasing on $(\frac{1}{2}, 1)$.
\end{lemma}
\begin{proof}
The result quickly follows from the product representation of $\widehat{\vartheta}(\beta;t)$ and the functional equation $\vartheta(\beta;t) = \tfrac{1}{\sqrt{t}} \widehat{\vartheta}(\beta;t)$.
	\begin{align}
		\partial_\beta \widehat{\vartheta}(\beta;t) & = -4 \pi \sin(2 \pi \beta) \sum_{l \geq 1} (1 - e^{-2l \pi t}) e^{-(2l-1)\pi t}\\
		& \quad \times
		\prod_{\substack{m \geq 1\\m \neq l}} (1 - e^{-2 m \pi t})(1 + 2 e^{-(2m-1)\pi t} \cos(2 \pi \beta) + e^{-(4m-2) \pi t}).
	\end{align}
	Hence, the function
	\begin{align}
		\frac{1}{\sqrt{t}} Q\left(\beta;\frac{1}{t}\right) & = 4 \pi \sum_{l \geq 1} (1 - e^{-2l \pi t}) e^{-(2l-1)\pi t}\\
		& \quad \times
		\prod_{\substack{m \geq 1\\m \neq l}} (1 - e^{-2 m \pi t})(1 + 2 e^{-(2m-1)\pi t} \cos(2 \pi \beta) + e^{-(4m-2) \pi t})
	\end{align}
	is a well-defined function for $\beta \in \R$. It is easily seen to be even and periodic in $\beta$ with period 1, as the only part involving $\beta$ is the cosine function. Furthermore, all terms of the above expression are positive as $(1+2e^{-(2m-1)\pi t} \cos(2 \pi \beta) + e^{-(4m-2) \pi t}) \geq (1 - e^{-(2m-1)\pi t})^2$.
\end{proof}

\begin{lemma}[Montgomery]\label{lem_aux_Q_AB}
	For any $\beta \in \R$, the function $Q(\beta;t)$ can be bounded by
	\begin{align}
		A(t) & =
		\begin{cases}
			t^{-3/2} e^{-\frac{\pi}{4 t}}, & 0 \leq t \leq 1\\
			\left(1 - \tfrac{1}{3000}\right) 4 \pi e^{-\pi t},	& 1 \leq t
		\end{cases}
	\end{align}
	and
	\begin{align}
		B(t) & =
		\begin{cases}
			t^{-3/2}, & 0 \leq t \leq 1\\
			\left(1 + \tfrac{1}{3000}\right) 4 \pi e^{-\pi t},	& 1 \leq t
		\end{cases},
	\end{align}
	respectively, i.e.,
	\begin{equation}
		A(t) \leq Q(\beta;t) \leq B(t).
	\end{equation}
\end{lemma}
\begin{proof}
	The result follows by using the functional equation to distinguish the cases $t \in (0,1)$ and $t \geq 1$, and by the dominance of the leading term together with tail-bound estimates. The details are given in \cite{Mon88}.
\end{proof}

\subsection{First Main Lemma}
Our First Main Lemma consists of two parts, which need to be shown separately. However, the range of the parameter $\alpha > 0$ can be reduced to $\alpha \geq 1$ by the functional equation \eqref{eq_functional}.
\begin{lemma}[First Main Lemma]\label{lem_first_main}
	Let $\alpha > 0$ be fixed, $x \in (0, \frac{1}{2})$ and $y \geq \frac{1}{\sqrt{2}}$, then
	\begin{equation}
		\partial_x \theta_\L(b_\L;\alpha) > 0
		\quad \text{ and } \quad
		\partial_x \widehat{\theta}_\L(b_\L;\alpha) > 0.
	\end{equation}
\end{lemma}

\subsubsection{First Main Lemma -- Part 1}
\begin{proposition}\label{pro_1_1}
	Let $\alpha \geq 1$ be fixed, $x \in (0, \tfrac{1}{2})$ and $y \geq \frac{1}{\sqrt{2}}$, then
	\begin{equation}
		\partial_x \theta_\L(b_\L;\alpha) > 0.
	\end{equation}
\end{proposition}
\begin{proof}
	We start with writing the derivative of $\theta_\L(b_\L;\alpha)$ with respect to $x$ as
	\begin{align}
		\partial_x \theta_\L(b_\L;\alpha) & = \sum_{l \in \Z} e^{-\pi \alpha y (l+\frac{1}{2}-r)^2} l \vartheta'(\tfrac{1}{2}+xl;\tfrac{\alpha}{y}),
	\end{align}
	where $\vartheta'$ denotes differentiation of $\vartheta$ with respect to the first argument. Due to the asymmetry caused by $r = r(y)$, we estimate the above expression for $l \geq 1$ and $l \leq - 1$ separately (the case $l = 0$ is trivial and yields 0).

	\medskip	
	
	\textit{Case A} ($l\geq 1$). For $l \geq 1$ and $x \in (0, 1/2)$, we note that $\vartheta'(1/2+x l)$ is positive for $l=1$ and alternates in sign for $l \geq 2$. Therefore, we estimate the term for $l=1$ from below by combining Lemma \ref{lem_aux_Q} and Lemma \ref{lem_aux_Q_AB}. In a similar way, we estimate the absolute values of all other terms from above and assign them a negative sign. In total, we get the following estimate from below
	\begin{align}
		& \, \sum_{l \geq 1} e^{-\pi \alpha y (l+\frac{1}{2}-r)^2} l \vartheta'(\tfrac{1}{2}+xl;\tfrac{\alpha}{y})\\
		\geq & \, e^{- \pi \alpha y (\frac{3}{2} - r)^2} \underbrace{(-\sin(2 \pi (\tfrac{1}{2}+x))}_{= \sin(2 \pi x)} \sqrt{\tfrac{y}{\alpha}} \, A(\tfrac{y}{\alpha})
		- \sum_{l \geq 2} l e^{- \pi \alpha y (l+\frac{1}{2} - r)^2} \sqrt{\tfrac{y}{\alpha}} \, B(\tfrac{y}{\alpha}) \left|\sin(2 \pi (x l+\tfrac{1}{2}))\right|.
	\end{align}
	
	We wish to show that
	\begin{equation}
		 e^{- \pi \alpha y (\frac{3}{2} - r)^2} \sqrt{\tfrac{y}{\alpha}} \, A(\tfrac{y}{\alpha}) \sin(2 \pi x)
		> \sum_{l \geq 2} l e^{- \pi \alpha y (l+\frac{1}{2} - r)^2} \sqrt{\tfrac{y}{\alpha}} \, B(\tfrac{y}{\alpha}) \left|\sin(2 \pi (x l+\tfrac{1}{2}))\right|,
	\end{equation}
	which is equivalent to showing that
	\begin{equation}
		\frac{A(\tfrac{y}{\alpha})}{B(\tfrac{y}{\alpha})} > \sum_{l \geq 2} l \frac{e^{- \pi \alpha y(l+\frac{1}{2}-r)^2}}{e^{-\pi \alpha y (\frac{3}{2}-r)^2}} \frac{\left|\sin(2 \pi x l)\right|}{\sin(2 \pi x)}.
	\end{equation}
	Now, by induction and using addition theorems for trigonometric functions, one easily obtains
	\begin{equation}
		\left| \frac{\sin(2 \pi x l)}{\sin(2 \pi x)} \right| \leq l, \quad l \in \N.
	\end{equation}
	Furthermore,
	\begin{equation}
		\frac{e^{- \pi \alpha y(l+\frac{1}{2}-r)^2}}{e^{-\pi \alpha y (\frac{3}{2}-r)^2}} = e^{- \pi \alpha y(l^2+l-2-2r(l-1))} \leq e^{- \pi \alpha y (l^2 + \frac{l-3}{2})},
	\end{equation}
	where we used the fact that $r(y) \in [0,\frac{1}{4}]$ for $y \geq 1/\sqrt{2}$ and that the whole expression is maximal for $r = 1/4$.
	Hence, we succeed if we can show that
	\begin{equation}\label{eq_aux1}
		\frac{A(\tfrac{y}{\alpha})}{B(\tfrac{y}{\alpha})} > \sum_{l \geq 2} l^2 e^{- \pi \alpha y (l^2 + \frac{l-3}{2})}.
	\end{equation}
	We need to distinguish the cases $y \leq \alpha$ and $y > \alpha$.

	\textit{Case A.1} ($\alpha \geq y$). As we assume $y \geq 1/\sqrt{2}$, we see that
	\begin{equation}
		\alpha y = \frac{\alpha}{y} y^2 \geq \frac{\alpha}{2 y}.
	\end{equation}
	We note that, in the case $y \leq \alpha$, we have
	\begin{equation}
		\frac{A(\tfrac{y}{\alpha})}{B(\tfrac{y}{\alpha})} = e^{- \pi \frac{\alpha}{4 y}},
	\end{equation}
	by the definitions of the functions in Lemma \ref{lem_aux_Q_AB}. The right-hand side of \eqref{eq_aux1} can be estimated by
	\begin{align}
		\sum_{l \geq 2} l^2 e^{- \pi \alpha y (l^2 + \frac{l-3}{2})} & \leq \sum_{l \geq 2} l^2 e^{- \pi \frac{\alpha}{2 y} (l^2 + \frac{l-3}{2})} = e^{- \pi \frac{\alpha}{4 y}} \sum_{l \geq 2} l^2 e^{- \pi \frac{\alpha}{2 y} (l^2 + \frac{l-4}{2})}\\
		& \leq e^{- \pi \frac{\alpha}{4 y}}\sum_{l \geq 2} l^2 e^{- \frac{\pi}{2} (l^2 + \frac{l-4}{2})} \leq 0.0359475 \ldots \times e^{-\pi \frac{\alpha}{4 y}} .
	\end{align}
	Therefore, the strict inequality \eqref{eq_aux1} follows in this case.
	
	\textit{Case A.2} ($\alpha \leq y$). Now, by the definitions in Lemma \ref{lem_aux_Q_AB}, the left-hand side of \eqref{eq_aux1} becomes
	\begin{equation}
		\frac{A(\tfrac{y}{\alpha})}{B(\tfrac{y}{\alpha})} = \frac{1-\frac{1}{3000}}{1+\frac{1}{3000}} = \frac{2999}{3001}.
	\end{equation}
	By assumption $\alpha y \geq \alpha^2 \geq 1$. So, the right-hand side of \eqref{eq_aux1} can be estimated by
	\begin{equation}
		\sum_{l \geq 2} l^2 e^{- \pi \alpha y (l^2 + \frac{l-3}{2})} \leq \sum_{l \geq 2} l^2 e^{- \pi (l^2 + \frac{l-3}{2})} \leq 0.0000671031 \ldots \leq \frac{2999}{3001}.
	\end{equation}
	This settles \textit{Case A}.
	
	\medskip
	
	\textit{Case B} ($l \leq -1$). For $l \leq -1$ and $x \in (0, \tfrac{1}{2})$, we estimate
	\begin{align}
		\sum_{l \leq -1} e^{-\pi \alpha y (l+\frac{1}{2}-r)^2} l \vartheta'(\tfrac{1}{2}+xl;\tfrac{\alpha}{y}) = \sum_{l \geq 1} e^{- \pi \alpha y (l-\tfrac{1}{2}+r)^2} l \vartheta'(\tfrac{1}{2} + x l; \tfrac{\alpha}{y}),
	\end{align}
	where we used the fact that $\vartheta'(1/2 + x l;\alpha/y) = - \vartheta'(1/2 - x l;\alpha/y)$, which can be deduced from Lemma \ref{lem_aux_theta}. We see that the expression above is certainly no less than
	\begin{equation}
		e^{- \pi \alpha y (\frac{1}{2} + r)^2} \underbrace{(-\sin(2 \pi (\tfrac{1}{2}+x))}_{= \sin(2 \pi x)} \sqrt{\tfrac{y}{\alpha}} \, A(\tfrac{y}{\alpha})
		- \sum_{l \geq 2} l e^{- \pi \alpha y (l-\frac{1}{2} + r)^2} \sqrt{\tfrac{y}{\alpha}} \, B(\tfrac{y}{\alpha}) \left|\sin(2 \pi (x l+\tfrac{1}{2}))\right|.
	\end{equation}
	We wish to show that
	\begin{equation}
		e^{- \pi \alpha y (\frac{1}{2} + r)^2} \sin(2 \pi x) \sqrt{\tfrac{y}{\alpha}} \, A(\tfrac{y}{\alpha})
		> \sum_{l \geq 2} l e^{- \pi \alpha y (l-\frac{1}{2} + r)^2} \sqrt{\tfrac{y}{\alpha}} \, B(\tfrac{y}{\alpha}) \left|\sin(2 \pi x l)\right|,
	\end{equation}
	or, equivalently,
	\begin{align}\label{eq_aux2}
		\frac{A(\tfrac{y}{\alpha})}{B(\tfrac{y}{\alpha})} > \sum_{l \geq 2} l e^{-\pi \alpha y (l^2-l+2r(l-1))} \frac{\left| \sin(2 \pi l x)\right|}{\sin(2 \pi x)}.
	\end{align}
	We argue again in two cases.
	
	\textit{Case B.1} ($\alpha \geq y$). The left-hand side of \eqref{eq_aux2} is simply given by
	\begin{equation}
		\frac{A(\tfrac{y}{\alpha})}{B(\tfrac{y}{\alpha})} = e^{- \pi \frac{\alpha}{4 y}}.
	\end{equation}
	For the right-hand side of \eqref{eq_aux2}, we observe that it is maximal if $r=0$, so
	\begin{align}
		\sum_{l \geq 2} l e^{-\pi \alpha y (l^2-l+2r(l-1))} \underbrace{\frac{\left| \sin(2 \pi l x)\right|}{\sin(2 \pi x)}}_{\leq l} & \leq \sum_{l \geq 2} l^2 e^{-\pi \frac{\alpha}{2 y} (l^2-l)} = e^{-\pi \frac{\alpha}{4 y}} \sum_{l \geq 2} l^2 e^{-\pi \frac{\alpha}{2 y} (l^2-l-\frac{1}{2})}\\
		& \leq e^{-\pi \frac{\alpha}{4 y}} \sum_{l \geq 2} l^2 e^{-\frac{\pi}{2} (l^2-l-\frac{1}{2})} \leq 0.380714 \ldots \times e^{-\pi \frac{\alpha}{4 y}}.
	\end{align}
	This shows \eqref{eq_aux2} in this case.
	
	\textit{Case B.2} ($\alpha \leq y$). The left-hand side of \eqref{eq_aux2} is now
	\begin{equation}
		\frac{A(\tfrac{y}{\alpha})}{B(\tfrac{y}{\alpha})} = \frac{2999}{3001}.
	\end{equation}
	As we now have $\alpha y \geq \alpha^2 \geq 1$, the right-hand side can be estimated by
	\begin{align}
		\sum_{l \geq 2} l e^{-\pi \alpha y (l^2-l+2r(l-1))} \frac{\left| \sin(2 \pi l x)\right|}{\sin(2 \pi x)} \leq \sum_{l \geq 2} l^2 e^{-\pi (l^2-l)} \leq 0.00746983 \ldots \, .
	\end{align}
	Thus \eqref{eq_aux2} is true in this case as well.
	
	Putting together the results from \textit{Case A} and \textit{Case B} finishes the proof.
\end{proof}

We remark that, with respect to $\alpha$, $\theta_\L(z_\L;\alpha)$ (for any fixed $z_\L \in \R^2$) behaves much like the function $\theta_\L(0;\alpha)$, which was considered by Montgomery \cite{Mon88}. Therefore, we could simply use some of the auxiliary results from \cite{Mon88}. This is very much in parallel to the similar behavior of the classical Jacobi theta nulls $\vartheta_3(t) = \vartheta(0;t)$ and $\vartheta_2(t) = \vartheta(1/2;t)$ as functions of $t$. However, the behavior of $\widehat{\theta}_\L (z_\L;\alpha)$ is quite opposite to the behavior of $\theta_\L(z_\L;\alpha)$ (or $\theta_\L(0;\alpha)$), just as $\vartheta_4(t) = \widehat{\vartheta}(1/2;t)$ behaves quite in opposition to $\vartheta_2(t)$ or $\vartheta_3(t)$ with respect to $t$.

\subsubsection{First Main Lemma --  Part 2}
Recall that we can write $b_1 = 1/2 - x \, b_2$. Hence,
\begin{equation}
	\widehat{\theta}_\L(b_\L;\alpha) = \sum_{k,l \in \Z} e^{-\pi \alpha (k^2 + 2 x k l +(x^2+y^2)l^2)} e^{2 \pi i (k b_2 - l (\frac{1}{2} - x b_2))}.
\end{equation}
Re-arranging the terms gives
\begin{equation}
	\widehat{\theta}_\L(b_\L;\alpha) = \sum_{l \in \Z} (-1)^l e^{-\pi \alpha y l^2} \sum_{k \in \Z} e^{2 \pi i l x b_2} e^{-\pi \frac{\alpha}{y} (k+xl)^2} e^{2 \pi i k b_2}.
\end{equation}

We see that the sum in $k$ sums time-frequency shifted Gaussians $\pi(-xl,b_2) \varphi(k)$ with a phase factor $e^{2 \pi i l x b_2}$ resulting from the non-commutativity of the time-shift and frequency-shift operators (see Section \ref{sec_tfa}). Since $\F (e^{-2 \pi i \omega x} M_\omega T_x \varphi) = M_{-x} T_\omega \varphi$, this suggests performing a Poisson summation in $k$ in order to get rid of the phase factor. We have
\begin{align}
	\widehat{\theta}_\L(b_\L;\alpha) & = \sum_{l \in \Z} (-1)^l e^{-\pi \alpha y l^2} \sum_{k \in \Z} \sqrt{\tfrac{y}{\alpha}} \, e^{-\pi \frac{y}{\alpha} (k+b_2)^2} e^{2 \pi i k l x}\\
	& = \sqrt{\tfrac{y}{\alpha}} \, \sum_{l \in \Z} (-1)^l e^{-\pi \alpha y l^2} \sum_{k \in \Z} e^{-\pi \frac{y}{\alpha} (k+b_2)^2} \cos(2 \pi k l x).
\end{align}
\begin{proposition}\label{pro_1_2}
	For $\alpha \geq 1$, $y \geq \frac{1}{\sqrt{2}}$ and $x \in (0, \frac{1}{2})$, we have
	\begin{equation}
		\partial_x \widehat{\theta}_\L(b_\L;\alpha) > 0.
	\end{equation}
\end{proposition}
\begin{proof}
	We start with computing the derivative with respect to $x$. We get
	\begin{align}
		\partial_x \widehat{\theta}_\L(b_\L;\alpha) & = \sqrt{\tfrac{y}{\alpha}} \, \sum_{l \in \Z} (-1)^l e^{-\pi \alpha y l^2} \sum_{k \in \Z} e^{-\pi \frac{y}{\alpha} (k+b_2)^2} (-2 \pi k l) \sin(2 \pi k l x)\\
		& = -4 \pi \sqrt{\tfrac{y}{\alpha}} \, \sum_{l \geq 1} (-1)^l l e^{-\pi \alpha y l^2} \sum_{k \in \Z} k e^{-\pi \frac{y}{\alpha} (k+b_2)^2} \sin(2 \pi k l x).
	\end{align}
	We see immediately that the expression vanishes for $x \in \frac{1}{2} \Z$. Also, the first terms in $k$ and $l$ should dominate the rest and the zeros should only come from the sine function. This needs to be proved of course.
	This time, we have a symmetry in $l$ and only need to distinguish the cases $y \geq \alpha$ and $y < \alpha$.

	\textit{First case} ($\alpha \leq y$). The term $l = 1$ (ignoring a factor of $+4 \pi \sqrt{y/\alpha}$ in front of the series) contributes
	\begin{align}
		& e^{-\pi \alpha y} \sum_{k \geq 1} k \left( e^{-\pi \frac{y}{\alpha} (k+b_2)^2} + e^{-\pi \frac{y}{\alpha} (k-b_2)^2} \right) \sin(2 \pi k x)\\
		\geq & 2 e^{-\pi \alpha y} \Bigg( e^{-\pi \frac{y}{\alpha}(1+b_2^2)} \cosh(2 \pi \tfrac{y}{\alpha} b_2) \sin(2 \pi x) - \sum_{k \geq 2} k e^{-\pi \frac{y}{\alpha} (k^2+b_2^2)} \cosh(2 \pi k \tfrac{y}{\alpha} b_2) |\sin(2 \pi k x)|\Bigg).
	\end{align}
	Dropping the factor 2, we wish to show that, for $x \in (0, \frac{1}{2})$,
	\begin{align}
		e^{-\pi \alpha y} e^{-\pi \frac{y}{\alpha} (1+b_2^2)} \cosh(2 \pi \tfrac{y}{\alpha} b_2) \sin(2 \pi x) \geq e^{-\pi \alpha y} \sum_{k \geq 2} k e^{-\pi \frac{y}{\alpha} (k^2+b_2^2)} \cosh(2 \pi k \tfrac{y}{\alpha} b_2) |\sin(2 \pi k x)|\\
		+ \sum_{l \geq 2} l e^{-\pi \alpha y l^2} \sum_{k \geq 1} k e^{-\pi \frac{y}{\alpha} (k^2+b_2^2)} \cosh(2 \pi k \tfrac{y}{\alpha} b_2) |\sin(2 \pi k l x)|.
	\end{align}
	Again, we face the problem that we do not have room for estimates near $x=0$ and $x=\frac{1}{2}$, as these values yield critical points and so the derivative vanishes. In order to overcome this problem, we may divide by the sine function, which is justified by l'Hospital's rule. This gives us the chance to actually estimate the second derivative of $\widehat{\vartheta}_\L(b_\L;\alpha)$ close to the critical points (up to the sign at $x=1/2$) and $\partial_x \widehat{\vartheta}_\L(b;\alpha)$ away from the critical points at the same time. Also, since the leading term should dominate everything, we see that, as a function of $x$, $\partial_x \widehat{\vartheta}_\L(b_\L;\alpha)$ should actually behave like a sine function with respect to $x$.
	So, dividing by $\sin(2 \pi x)$ leaves us to prove that
	\begin{align}
		e^{-\pi \alpha y} e^{-\pi \frac{y}{\alpha} (1+b_2^2)} \cosh(2 \pi \tfrac{y}{\alpha} b_2) \geq & e^{-\pi \alpha y} \sum_{k \geq 2} k e^{-\pi \frac{y}{\alpha} (k^2+b_2^2)} \cosh(2 \pi k \tfrac{y}{\alpha} b_2) \frac{|\sin(2 \pi k x)|}{\sin(2 \pi x)}\\
		+ & \sum_{l \geq 2} l e^{-\pi \alpha y l^2} \sum_{k \geq 1} k e^{-\pi \frac{y}{\alpha} (k^2+b_2^2)} \cosh(2\pi k \tfrac{y}{\alpha} b_2) \frac{|\sin(2 \pi k l x)|}{\sin(2 \pi x)}.
	\end{align}
	Again, we use the formula
	\begin{equation}
		\left| \frac{\sin(2 \pi N x)}{\sin(2 \pi x)}\right| \leq N, \quad N \in \N.
	\end{equation}
	Hence, as $\sin(2 \pi x) = |\sin(2\pi x)|$ for $x \in [0, 1/2]$, we succeed if we can show the stronger inequality
	\begin{align}\label{eq_estimate_reenter}
		e^{-\pi \alpha y} e^{-\pi \frac{y}{\alpha} (1+b_2^2)} \cosh(2 \pi \tfrac{y}{\alpha} b_2) \geq & e^{-\pi \alpha y} \sum_{k \geq 2} k^2 e^{-\pi \frac{y}{\alpha} (k^2+b_2^2)} \cosh(2 \pi k \tfrac{y}{\alpha} b_2)\\
		+ & \sum_{l \geq 2} l^2 e^{-\pi \alpha y l^2} \sum_{k \geq 1} k^2 e^{-\pi \frac{y}{\alpha} (k^2+b_2^2)} \cosh(2\pi k \tfrac{y}{\alpha} b_2),
	\end{align}
	or, equivalently,
	\begin{align}
		\cosh(2 \pi \tfrac{y}{\alpha} b_2) \geq & \sum_{k \geq 2} k^2 e^{-\pi \frac{y}{\alpha} (k^2-1)} \cosh(2 \pi k \tfrac{y}{\alpha} b_2)\\
		& + \sum_{l \geq 2} l^2 e^{-\pi \alpha y (l^2-1)} \sum_{k \geq 1} k^2 e^{-\pi \frac{y}{\alpha} (k^2-1)} \cosh(2\pi k \tfrac{y}{\alpha} b_2).
	\end{align}
	
	By assumption $\alpha \geq 1$ and $\alpha \leq y$, therefore
	\begin{equation}
		\alpha y \geq \alpha^2 \geq 1.
	\end{equation}
	Also, we have $b_2 \in [1/4, 1/2]$ (as $y \geq 1/\sqrt{2}$). By under-estimating the left-hand side and over-estimating the right-hand side, we obtain the stronger inequality
	\begin{align}
		\cosh(\tfrac{\pi}{2}) \geq & \sum_{k \geq 2} k^2 e^{-\pi \frac{y}{\alpha} (k^2-1)} e^{\pi k \tfrac{y}{\alpha}} + \sum_{l \geq 2} l^2 e^{-\pi \alpha y (l^2-1)} \sum_{k \geq 1} k^2 e^{-\pi \frac{y}{\alpha} (k^2-1)} e^{\pi k \tfrac{y}{\alpha}}\\
		& = \sum_{k \geq 2} k^2 e^{-\pi \frac{y}{\alpha} (k^2-k-1)} + \sum_{l \geq 2} l^2 e^{-\pi \alpha y (l^2-1)} \sum_{k \geq 1} k^2 e^{-\pi \frac{y}{\alpha} (k^2-k-1)}.
	\end{align}
	The right-hand side is at most
	\begin{equation}
		 \sum_{k \geq 2} k^2 e^{-\pi (k^2-k-1)} + \sum_{l \geq 2} l^2 e^{-\pi (l^2-1)} \sum_{k \geq 1} k^2 e^{-\pi (k^2-k-1)} = 0.180383 \ldots < 1 < \cosh(\tfrac{\pi}{2})
	\end{equation}
	and so this case is settled.
	
	\medskip
		
	\textit{Second case} ($\alpha \geq y$). Now, as we assume $\alpha \geq 1$ and $y \geq 1/\sqrt{2}$, we have
	\begin{equation}
		\alpha y = \frac{\alpha}{y}	y^2 \geq \frac{\alpha}{2 y}.
	\end{equation}
	Interestingly, the proof of this case now differs somewhat from the previous cases (at least the proof that we could find), but the overall idea stays the same.
	We re-enter the first case at the inequality \eqref{eq_estimate_reenter};
	\begin{align}
		e^{-\pi \alpha y} e^{-\pi \frac{y}{\alpha} (1+b_2^2)} \cosh(2 \pi \tfrac{y}{\alpha} b_2) \geq & e^{-\pi \alpha y} \sum_{k \geq 2} k^2 e^{-\pi \frac{y}{\alpha} (k^2+b_2^2)} \cosh(2 \pi k \tfrac{y}{\alpha} b_2)\\
		+ & \sum_{l \geq 2} l^2 e^{-\pi \alpha y l^2} \sum_{k \geq 1} k^2 e^{-\pi \frac{y}{\alpha} (k^2+b_2^2)} \cosh(2\pi k \tfrac{y}{\alpha} b_2).
	\end{align}
	The problem is now that $\frac{y}{\alpha}$ is potentially small and so a lot of terms may contribute to the series in $k$. The trick is, of course to use the Poisson summation formula again, to invert the parameters in the exponentials.
	In order to be able to use Poisson summation for the right-hand side we wish to sum over all integers again. So, we start with adding the left-hand side, which corresponds to the case $(k,l)=(1,1)$, on both sides of the inequality;
	\begin{align}
		2 e^{-\pi \alpha y} e^{-\pi \frac{y}{\alpha} (1+b_2^2)} \cosh(2 \pi \tfrac{y}{\alpha} b_2) \geq & \sum_{l \geq 1} l^2 e^{-\pi \alpha y l^2} \sum_{k \geq 1} k^2 e^{-\pi \frac{y}{\alpha} (k^2+b_2^2)} \cosh(2\pi k \tfrac{y}{\alpha} b_2).
	\end{align}
	Showing that this is true is equivalent to showing that
	\begin{equation}
		2 e^{-\pi \frac{y}{\alpha}} \geq \sum_{l \geq 1} l^2 e^{-\pi \alpha y (l^2-1)} \sum_{k \geq 1} k^2 e^{-\pi \frac{y}{\alpha} k^2} \tfrac{\cosh(2 \pi \frac{y}{\alpha} k b_2)}{\cosh(2 \pi \frac{y}{\alpha} b_2)}.
	\end{equation}
	Now, we make the simple observation that $\frac{\cosh(R x)}{\cosh(x)}$ is growing for $x > 0$ and any fixed $R > 1$. This is easily shown by finding all $x > 0$ such that
	\begin{equation}
		\dfrac{d}{dx} \frac{\cosh(R x)}{\cosh(x)} = \frac{R \ \sinh(R x)\cosh(x)-\cosh(R x) \sinh(x)}{\cosh(x)^2} >0.
	\end{equation}
	Re-arranging this inequality leads to
	\begin{equation}
		R \ \tanh(R x) > \tanh(x),
	\end{equation}
	which is true for all $x > 0$ and $R > 1$. So, by replacing $b_2$ by $\frac{1}{2}$, we succeed if we can show the resulting stronger inequality
	\begin{equation}
		2 e^{-\pi \frac{y}{\alpha}} \geq \sum_{l \geq 1} l^2 e^{-\pi \alpha y (l^2-1)} \sum_{k \geq 1} k^2 e^{-\pi \frac{y}{\alpha} k^2} \tfrac{\cosh(2 \pi \frac{y}{\alpha} k \frac{1}{2})}{\cosh(2 \pi \frac{y}{\alpha} \frac{1}{2})},
	\end{equation}
	as $b_2 \in [1/4, 1/2]$. Another simple computation shows that, for any $k \geq 1$,
	\begin{equation}
		\frac{\cosh(\pi \frac{y}{\alpha} k)}{\cosh(\pi \frac{y}{\alpha})} \leq e^{\pi \frac{y}{\alpha} (k-1)},
	\end{equation}
	which leads us to the new and stronger inequality
	\begin{align}\label{eq_1}
		2 e^{-\pi \frac{y}{\alpha}} & \geq \sum_{l \geq 1} l^2 e^{-\pi \alpha y (l^2-1)} \sum_{k \geq 1} k^2 e^{-\pi \frac{y}{\alpha} k^2} e^{\pi \frac{y}{\alpha}(k-1)}\\
		& = \sum_{l \geq 1} l^2 e^{-\pi \alpha y (l^2-1)} \sum_{k \geq 1} k^2 e^{-\pi \frac{y}{\alpha} (k-\frac{1}{2})^2} e^{-\pi \frac{3y}{4\alpha}}.
	\end{align}
	Next, we use the simple estimate (note the range of the index)
	\begin{equation}
		\sum_{k \geq 1} k^2 e^{-\pi \frac{y}{\alpha} (k-\frac{1}{2})^2} e^{-\pi \frac{3y}{4\alpha}} \leq \sum_{k \in \Z} k^2 e^{-\pi \frac{y}{\alpha} (k+\frac{1}{2})^2} e^{-\pi \frac{3y}{4\alpha}}.
	\end{equation}
	As $k$ is now running through the integers again, we may apply the Poisson summation formula. We start by rewriting the sum and get
	\begin{align}
		\sum_{k \in \Z} k^2 e^{-\pi \frac{y}{\alpha} (k+\frac{1}{2})^2}
		& =\sum_{k \in \Z} \left[\left(k+\tfrac{1}{2}\right)^2 -\left(k+\tfrac{1}{2}\right)+\tfrac{1}{4}\right] e^{-\pi \frac{y}{\alpha} (k+\frac{1}{2})^2}\\
		& = \sum_{k \in \Z} \left(k+\tfrac{1}{2} \right)^2 e^{-\pi \frac{y}{\alpha} (k+\frac{1}{2})^2} -\underbrace{\sum_{k \in \Z} \left(k+\tfrac{1}{2} \right) e^{-\pi \frac{y}{\alpha} (k+\frac{1}{2})^2}}_{=0} +\tfrac{1}{4}\sum_{k \in \Z}  e^{-\pi \frac{y}{\alpha} (k+\frac{1}{2})^2}\\
		& = \sum_{k \in \Z} \left(k+\tfrac{1}{2} \right)^2 e^{-\pi \frac{y}{\alpha} (k+\frac{1}{2})^2} + \tfrac{1}{4}\left(\tfrac{\alpha}{y} \right)^{\frac{1}{2}}\sum_{k \in \Z} (-1)^k e^{-\frac{\pi \alpha}{y}k^2}.
	\end{align}	
	Note that the negative part canceled due to symmetry (we thank the referee for pointing this out, thus simplifying our argument). Similar results may be found in \cite{Bochner}. We next compute
	\begin{align}
		\sum_{k \in \Z} \left(k+\tfrac{1}{2} \right)^2 e^{-\pi \frac{y}{\alpha} (k+\frac{1}{2})^2}
		& = -\tfrac{1}{\pi}\partial_t\left[\sum_{k \in \Z} e^{-\pi t (k+\frac{1}{2})^2} \right]_{t=\frac{y}{\alpha}}
		 = -\tfrac{1}{\pi}\partial_t\left[\tfrac{1}{\sqrt{t}}\sum_{k \in \Z} (-1)^k e^{-\frac{\pi}{t}k^2}\right]_{t=\frac{y}{\alpha}}\\
		& = -\tfrac{1}{\pi}\sum_{k \in \Z} (-1)^k e^{-\frac{\pi}{t}k^2}\left[-\tfrac{1}{2t^{\frac{3}{2}}}+\tfrac{\pi k^2}{t^{\frac{5}{2}}}  \right]_{t=\frac{y}{\alpha}}\\
		& = \sum_{k \in \Z} (-1)^k e^{-\frac{\pi \alpha}{y}k^2} \left( \tfrac{1}{2\pi}\left( \tfrac{\alpha}{y} \right)^{\frac{3}{2}} - \left( \tfrac{\alpha}{y} \right)^{\frac{5}{2}}k^2 \right).
	\end{align}
	The interchange of summation and differentiation is justified due to the rapid decay of the series. It follows that
	\begin{equation}
		\sum_{k \in \Z} k^2 e^{-\pi \frac{y}{\alpha} (k+\frac{1}{2})^2}=\left(\tfrac{\alpha}{y}\right)^{3/2} \sum_{k \in \Z} (-1)^k e^{-\pi \frac{\alpha}{y} k^2} \left(\tfrac{1}{2 \pi} +\tfrac{y}{4 \alpha} - \tfrac{\alpha}{y} k^2\right).
	\end{equation}
	
	As $\alpha/y$ becomes large, the exponential terms in $k$ are small rather quickly. However, the factor in front of the series becomes huge, but it cannot beat the exponential decay. In fact, we have the following inequality, which can be checked by very elementary means: For any $t > 1$, we have 
	\begin{equation}\label{eq_aux_exp_t}
		t^{3/2} e^{-\pi \frac{t}{4}} < e^{-\pi \frac{1}{4 t}}.
	\end{equation}	
	The aim of inequality \eqref{eq_aux_exp_t} is to compensate for the inversion of the exponent and the appearance of the monomial, due to the Poisson summation formula, in \eqref{eq_1}. By using the substitution $t \mapsto \frac{1}{t}$, we see that we have the equivalent inequality
	\begin{equation}
		e^{-\pi \frac{t}{4}} > \left(\tfrac{1}{t}\right)^{3/2} e^{-\frac{\pi}{4 t}} \quad \text{ for } t \in (0,1). 
	\end{equation}
	As $\alpha > y$, we also have
	\begin{equation}
		2 e^{-\pi \frac{y}{4\alpha}} > 2 \left(\tfrac{\alpha}{y}\right)^{3/2} e^{-\pi \frac{\alpha}{4y}},
	\end{equation}
	which now leads us to showing the stronger inequality
	\begin{equation}
		2 \left(\tfrac{\alpha}{y}\right)^{3/2} e^{-\pi \frac{\alpha}{4y}} \geq \left(\tfrac{\alpha}{y}\right)^{3/2} \sum_{l \geq 1} l^2 e^{-\pi \alpha y (l^2-1)} \sum_{k \in \Z} (-1)^k e^{-\pi \frac{\alpha}{y} k^2} (\tfrac{1}{2 \pi} + \tfrac{y}{4\alpha} - \tfrac{\alpha}{y} k^2).
	\end{equation}
	Now, we succeed if we can show the even stronger inequality
	\begin{equation}
		2 \geq \sum_{l \geq 1} l^2 e^{-\pi \alpha y (l^2-1)} \left( (\tfrac{1}{2\pi} + \tfrac{y}{4 \alpha})e^{\pi\frac{\alpha}{4y}} + 2 \sum_{k \geq 1} e^{-\pi \frac{\alpha}{y} (k^2-\frac{1}{4})} ( \tfrac{\alpha}{y} k^2 + \tfrac{1}{2 \pi} + \tfrac{y}{4 \alpha}) \right).
	\end{equation}
	
	As preparation for estimating the right-hand side, we observe that, for $t > 0$,
	\begin{equation}
		t \mapsto t k^2 e^{-\pi t (k^2 - \frac{1}{4})}
	\end{equation}
	has a maximum in $t = \frac{4}{(4k^2-1)\pi}$ and is decreasing afterwards. In particular, this implies that $t \mapsto t k^2 e^{-\pi t (k^2 - \frac{1}{4})}$ is maximal for $t = 1$ on $[1,\infty)$ and any fixed $1 \leq k \in \N$. Furthermore, as $y^2 \geq \frac{1}{2}$, we have
	\begin{equation}
		e^{-\pi \alpha y(l^2-1)} e^{\pi \frac{\alpha}{4 y}} = e^{- \pi \alpha y (l^2-1-\frac{1}{4y^2})} \leq e^{-\pi \alpha y (l^2 - 1 - \frac{1}{2})}.
	\end{equation}
	
	Therefore, all that is left to observe, is the fact that
	\begin{equation}
		2 >  \left( \tfrac{1}{2\pi} + \tfrac{1}{4} \right) \sum_{l \geq 1} l^2 e^{-\frac{\pi}{2} (l^2-\frac{3}{2})} + 2 \sum_{l \geq 1} l^2 e^{-\frac{\pi}{2} (l^2-1)} \sum_{k \geq 1} e^{-\pi (k^2-\frac{1}{4})} (k^2 + \tfrac{1}{2 \pi} + \tfrac{1}{4}) = 1.20646 \ldots \; .
	\end{equation}
\end{proof}
In combination, Proposition \ref{pro_1_1} and Proposition \ref{pro_1_2} prove Lemma \ref{lem_first_main} and our analysis of the $x$-derivative is finished.

\section{Analysis on the line \texorpdfstring{$x=\tfrac{1}{2}$}{y=1/2}}
\subsection{Second Main Lemma.}
This final section is dedicated to proving the two remaining estimates, which we state as our Second Main Lemma. Recall that for $x=1/2$ we have $b(1/2,y) = a(1/2,y)$ for all $y \geq \sqrt{3}/2$.
\begin{lemma}[Second Main Lemma]\label{lem_second_main}
	Let $\alpha > 0$ be fixed, $b_\L$ given by \eqref{eq_b} with $x = \frac{1}{2}$ and $y \geq \frac{\sqrt{3}}{2}$, then
	\begin{equation}
		\theta_\L(b_\L;\alpha)
		\quad \text{ and } \quad
		\widehat{\theta_\L}(b_\L;\alpha)
		\qquad \text{ are maximized}
	\end{equation}
	if and only if $y = \frac{\sqrt{3}}{2}$.
\end{lemma}
We first state the two estimates and then discuss the ideas behind their proof before giving the concrete arguments. Both estimates aim to show that an infinite sum of functions attains its maximum at a fixed point. We abbreviate $\theta_\L(b_\L;\alpha)$ and $\widehat{\theta}_\L(b_\L;\alpha)$ as
\begin{equation}
	f_{\alpha}(y) = \sum_{k, l \in \mathbb{Z}}  e^{-\frac{\pi \alpha}{y} \left( (k+b_1)^2 + (k+b_1)(l+b_2) + \left( \frac{1}{4} +y^2\right) (l + b_2)^2\right)}
\end{equation}
and
\begin{equation}
	g_{\alpha}(y) = \sum_{k, l \in \mathbb{Z}} e^{- \frac{\pi\alpha}{y} \left(k^2 +  k l + \left(\frac14 + y^2\right) l^2 \right)} e^{2\pi i (k b_2 - l b_1 ) },
\end{equation}
respectively. We will also refer to these functions as \textit{the first sum} and \textit{the second sum}. Since $x = 1/2$, the coordinates for the point $b_\L$ are explicitly given by
\begin{equation}
	b_1(y) =  \frac{1}{4} + \frac{1}{16 y^2} \qquad \mbox{and} \qquad b_2(y) = \frac{1}{2} - \frac{1}{8y^2}.
\end{equation}
The goal is to prove that $f_{\alpha}(y)$ and $g_\alpha(y)$ attain their maximum in $y = \sqrt{3}/2$ for all $\alpha > 0$.
\begin{proposition}\label{prop1}
	For all $\alpha \geq 1$, we have
	\begin{equation}
		\max_{y \geq \sqrt{3}/2} f_{\alpha}(y) = f_{\alpha}(\sqrt{3}/2).
	\end{equation}
\end{proposition}
The other half of the proof is then showing the analogous statement for $g_\alpha(y) = \frac{1}{\alpha} f_{\frac{1}{\alpha}}(y)$.
\begin{proposition}\label{prop2}
	For all $\alpha \geq 1$, we have
	\begin{equation}
		\max_{y \geq \sqrt{3}/2} g_{\alpha}(y) = g_{\alpha}(\sqrt{3}/2).
	\end{equation}
\end{proposition}

\subsection{The architecture of the proofs.} We quickly provide a bird's-eye view of the actual argument and emphasize the nontrivial components of each proof (which also helps in clarifying which parts of the proofs will be laborious but ultimately not difficult estimates). Both proofs can be roughly summarized in the following way: we are given a function of the form
\begin{equation}
	h_{\alpha}(y) = \sum_{k,l \in \mathbb{Z}} e^{ -\pi \alpha \, \phi_{k,l}(y)} \, \psi_{k,l}(y).
\end{equation}
The overarching idea will be as follows: we will establish bounds on
\begin{equation}
	\min_{y \geq \sqrt{3}/2} \phi_{k,l}(y) \quad \text{ from below}
\end{equation}
which grow with $k$ and $l$ (see Lemma \ref{growth1} and Lemma \ref{growth2}). This then suggests that in the sum defining $h_{\alpha}(y)$ only relatively few terms will actually contribute substantially. This is indeed what happens when $\alpha$ is uniformly bounded away from 0. This is also the reason why we use the Poisson summation formula: it allows us to deal with two sums for which $\alpha \geq 1$ as opposed to the much more difficult problem of analyzing $\alpha > 0$ for any single sum.

\medskip

The next idea, crucial for book-keeping, is the following: we expect that most of the argument will actually be concerned with $y \sim \sqrt{3}/2$ close to the point where the maximum occurs. Indeed, both sums, the first sum and the second sum, have a critical point in $y =\sqrt{3}/2$ (see Lemma \ref{lem_hexagon_critical}). We thus expect that what is most important about a given function $\phi_{k,l}(y)$ is actually its value in the conjectured maximum $\phi_{k,l}(\sqrt{3}/2)$. These values result in nice quadratic forms;
\begin{align}
	Q_1(k,l) &= \frac{\sqrt{3}}{2} \left(l + \frac{1}{3} \right)^2 + \frac{2}{\sqrt{3}}\left(k + \frac{l+1}{2} \right)^2 \qquad \mbox{for the first sum} \\
	Q_2(k,l) &= k^2 + kl +l^2 \qquad \qquad \qquad\qquad \qquad\qquad \mbox{for the second sum.}
\end{align}
We remark that we dropped a factor $2/\sqrt{3}$ which should appear in $Q_2(k,l)$. We will use these quadratic forms as a way of estimating $ \exp \left( -\pi \alpha \cdot \phi_{k,l}(y)\right) \psi_{k,l}(y)$ by giving estimates for its maximum, the maximum of its derivative and the maximum of its second derivative all in terms of $Q_i(k,l)$. This is essentially a way of book-keeping by allowing us to discard the variable $y$. The proofs are then ultimately simple:
\begin{enumerate}
	\item Show that $h_{\alpha}(y)$ has a critical point in $y = \sqrt{3}/2$
	\item Show that any global maximum has to happen close to $\sqrt{3}/2$ and
	\item Show that the second derivative of $h_{\alpha}$ is negative in that region of interest.
\end{enumerate}

For both sums, the `region of interest' will depend on $\alpha$ as it gets larger. We will consider 
\begin{align*}
	\frac{\sqrt{3}}{2} &\leq y \leq \frac{\sqrt{3}}{2}+ \frac{1}{3\sqrt{\alpha}} \qquad \mbox{for the first sum and}\\
	\frac{\sqrt{3}}{2} &\leq y \leq \frac{\sqrt{3}}{2}+ \frac{1}{4\sqrt{\alpha}} \qquad \mbox{for the second sum.}
\end{align*}
We note at this point that one could work with a uniform region of interest, say $y \in [\sqrt{3}/2,1]$, for the second sum (Lemma \ref{lem:almostdone} does not technically require a smallness condition on the region of interest); this could simplify the step (2) when carried out for the second sum but we have sufficient control (Lemma \ref{lem:heat}) and it is not required. There are several other differences between the two sums. The second sum is slightly more delicate: we only have favorable universal estimates for summands indexed by $(k,l) \in \mathbb{Z}^2$ when $l \neq 0$. However, the sum corresponding with respect to $k \in \mathbb{Z}$, $l = 0$ can be written as a Jacobi theta function $\theta_3$. We will use some intuition coming from that to get the desired argument.

\subsection{How to read the argument.}
Ultimately, the argument is quite simple: both $f_{\alpha}$ and $g_{\alpha}$ have a critical point in $y = \sqrt{3}/2$, they have a negative second derivative in a neighborhood of that point and are much smaller than the value in $y = \sqrt{3}/2$ outside that neighborhood. We also emphasize that for each $\alpha \geq 1$, both functions $f_{\alpha}$ and $g_{\alpha}$ are effectively dominated by relatively few terms: their summands are of the form  $ \exp \left( -\pi \alpha \, \phi_{k,l}(y)\right) \psi_{k,l}(y)$ and the smallest possible value of $\phi_{k,l}(y)$ grows like the square root of a positive-definite quadratic form in $k,l$ (so at least linearly). In particular, as $\alpha \rightarrow \infty$, the function $f_{\alpha}$ is dominated by three terms (the three terms for which the quadratic form $Q_1(k,l)$, introduced below, is minimal) and the function $g_{\alpha}$ is dominated by six terms (the six terms for which the quadratic form $Q_2(k,l)$, introduced below, is minimal).

\begin{figure}[ht]
    \begin{framed}
    	\includegraphics[width=.45\textwidth]{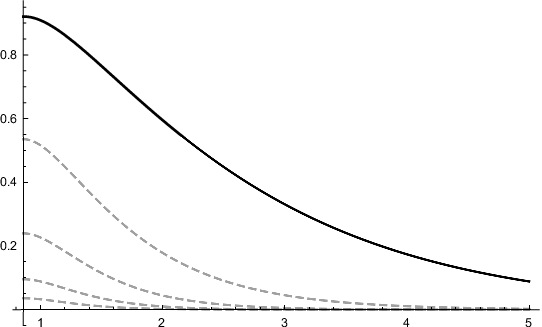}
    	\hfill
    	\includegraphics[width=.45\textwidth]{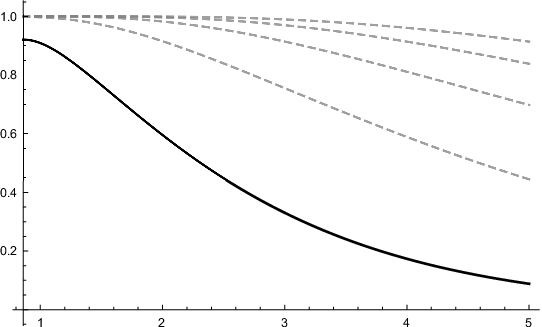}
    	\caption{The functions $\alpha f_\alpha (y)$ (left) and $g_\alpha (y)$ (right) for values $\alpha = 1$ (solid graph) and $\alpha \in \{2,3,4,5\}$ (dashed) in $\sqrt{3}/2 \leq y \leq 5$. Note that $f_1(y) = g_1(y)$. While $\alpha f_\alpha$ gets numerically small, the problem for $g_\alpha$ is that it becomes flat as $\alpha \to \infty$.}\label{fig:galpha}
    \end{framed}
\end{figure}

The difficulty is thus two-fold: to make the arguments effective for values $1 \leq \alpha \ll \infty$, and to ensure that things stay well-behaved as $\alpha$ gets large. A hint of the difficulty is shown in Figure \ref{fig:galpha}. As $\alpha$ gets large, the function $g_{\alpha}$ becomes very flat close to the conjectured maximum: the second derivatives are negative but they are exponentially decaying in $\alpha$. This requires a careful asymptotic analysis. We simplify the exposition of the argument in several ways by
\begin{enumerate}
	\item  using a finite amount of computation to deal with the cases $1 \leq \alpha \leq 6$ for $f_{\alpha}$ and $1 \leq \alpha \leq 5$ for $g_{\alpha}$ which are well-behaved: there are finitely many relevant terms which are nicely behaved and not yet exponentially small
	\item  not giving all the details whenever the missing details are either standard estimates or standard computations (or both).
\end{enumerate}
The cases of small $\alpha$ are actually fairly easy to deal with: (1) the second derivatives are not only negative, they are negative and bounded away from 0, (2) the functions $f_{\alpha}$ and $g_{\alpha}$ decay quickly away from the conjectured maximum. We note that checking these cases computationally would not be strictly required: our framework is robust and could handle $\alpha \geq 1$, however, it would require the asymptotic analysis of $\sim 20$ quantities as opposed to $\sim 6$ quantities which would make the argument much longer without adding anything of substance. Moreover, the behavior for relatively small values of $\alpha$ is not actually in any way subtle: the subtlety lies in controlling the asymptotic behavior.

\medskip

As for the details, there is a little bit of redundancy in both arguments since they are somewhat analogous: for example, Lemma \ref{lem:4} shows how to handle the three main terms of the sum algebraically, Lemma \ref{lem:44} does the same for 9 terms but we skip the details. We generally give a little bit more detail for $f_{\alpha}$ and less details for the analogous parts in $g_{\alpha}$. However, we note that there are difficulties in $g_{\alpha}$ that do not appear in $f_{\alpha}$ and we give details for these. Many of our arguments will end up in inequalities of a specific type, an example of which is the following:
\begin{equation}
	\sum_{\substack{k,l \in \mathbb{Z}\\k^2 + l^2 \geq 2}} \exp\left(-\pi \alpha (k^2 + l^2)\right) \leq e^{-\pi \alpha}, \quad \alpha \geq \alpha_0.
\end{equation}
We note that (which will be true for all inequalities of this type that we encounter) this inequality will be false for small values of $\alpha$  and will be true for sufficiently large values of $\alpha$ (the left-hand side has asymptotic behavior $\sim 4 e^{-2 \pi \alpha}$). None of this is difficult to prove (due to the rapid decay of the terms, one could for example use standard geometric series). We will skip the verification of these sums and instead merely list the approximate threshold value $\alpha_0$ such that the inequality will be true  for any $\alpha \geq \alpha_0$ (in the example shown above, we have $\alpha_0 \sim 0.47$). None of our arguments will be particularly tight in terms of these constants $\alpha_0$ and they could all be made precise using the typical elementary arguments (whenever the arguments are not typical, they are given). Since none of our arguments are particularly tight, there is some flexibility: by improving the estimates in one Lemma, one might be able to reduce the value of $\alpha_0$ in another and so on.
 
\subsection{Proof of Proposition \ref{prop1}}
\subsubsection{Proof of Proposition \ref{prop1}: critical point.}
\begin{lemma} \label{lem:crit1}
For all $\alpha \geq 1$, the function $f_{\alpha}(y)$ has a critical point in $y = \sqrt{3}/2$.
\end{lemma}
\begin{proof}  This is a simple consequence of Lemma \ref{lem_hexagon_critical}. \end{proof}

\subsubsection{Proof of Proposition \ref{prop1}: introducing $Q_1$.}
The first sum, which will be denoted by $f_{\alpha}(y)$ can be written as
\begin{equation}
	f_{\alpha}(y) = \sum_{k,l \in \mathbb{Z}^2} e^{- \pi \alpha \, \phi_{k,l}(y)} = \sum_{k,l \in \Z^2} f_{k,l}(y),
\end{equation}
where we use the abbreviations
\begin{equation}
	f_{k,l}(y) = e^{- \pi \alpha \, \phi_{k,l}(y)}
\end{equation}
and
\begin{align}
	\phi_{k,l}(y) & = \tfrac{1}{y} \left( (k+a_1)^2+(k+a_1)(l+a_2)+(\tfrac{1}{4}+y^2)(l+a_2)^2 \right)\\
	& = \tfrac{1}{y} \left( (k+\tfrac{l+1}{2})^2 + y^2(l+a_2)^2 \right)
\end{align}
Indeed, in the point $y=\sqrt{3}/2$, the expression simplifies to
\begin{equation}
	f_{\alpha}\left(\tfrac{\sqrt{3}}{2} \right) = \sum_{k,l \in \mathbb{Z}^2} e^{- \pi \alpha \, \left(\frac{2}{\sqrt{3}}(k+\frac{l+1}{2})^2 + \frac{\sqrt{3}}{2}(l+\frac{1}{3})^2 \right)}.
\end{equation}
This suggests introducing the quadratic form
\begin{equation}
	Q_1(k,l) =  \frac{2}{\sqrt{3}} \left(k+\frac{l+1}{2} \right)^2 + \frac{\sqrt{3}}{2}\left(l+\frac{1}{3}\right)^2
\end{equation}
allowing us to write
\begin{equation}
	f_{\alpha}\left(\tfrac{\sqrt{3}}{2} \right) = \sum_{k,l \in \mathbb{Z}^2} e^{- \pi \alpha \, Q_1(k,l)}.
\end{equation}

\begin{center}
	\begin{figure}[ht]
        \begin{framed}
    		\begin{tabular}{c|c}
    			$Q_1(k,l)$  &  $(k,l)$\\
    			\hline
    			$1 \times 2/(3\sqrt{3})$ & $(0,-1), (-1, 0), (0,0)$\\
    			$4 \times 2/(3\sqrt{3})$ & $(-1,-1), (1,-1), (-1, 1)$ \\
    			$7 \times 2/(3\sqrt{3})$& $(0,-2), (1,-2), (-2, 0), (1,0), (-2, 1), (0,1)$\\
    		\end{tabular}
    		\caption{The first few values of $Q_1(k,l)$ together with the points where it is assumed.}\label{fig_Q1}
        \end{framed}
	\end{figure}
\end{center}

\begin{figure}[ht]
    \begin{framed}
    	\hfill
    	\includegraphics[width=.4\textwidth]{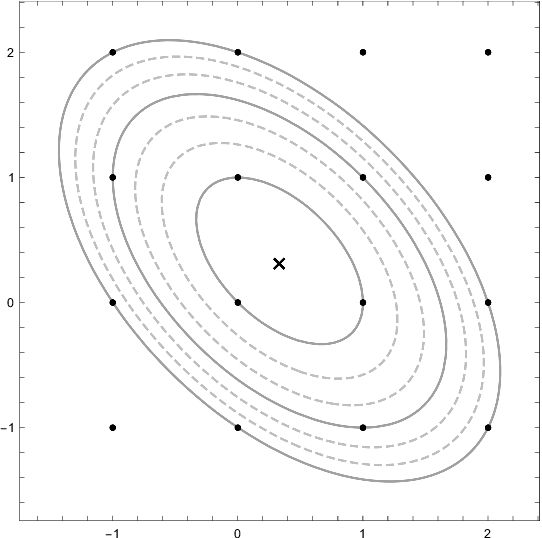}
    	\hfill
        \hfill
    	\includegraphics[width=.4\textwidth]{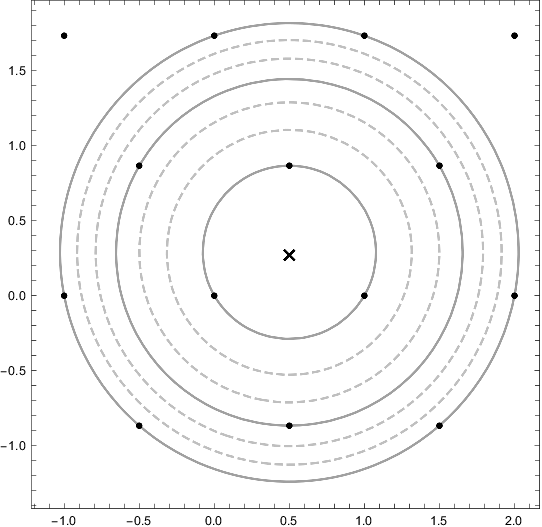}
    	\hspace*{\fill}
    	\caption{\textit{Left}: Collecting $\Z^2$-lattice points using the quadratic form $Q_1(k,l)$. Ignoring the factor $\frac{2}{3 \sqrt{3}}$, there are lattice points where $Q_1 \in \{1,4,7\}$ (solid ellipses), but $Q_1 \notin \{2,3,5,6\}$ (dashed ellipses). \textit{Right}: After a linear change of coordinates, we see that $Q_1$ collects lattice points of the hexagonal lattice at a given distance from the circumcenter.}\label{fig_ellipse_c}
    \end{framed}
\end{figure}

\begin{lemma}[Growth of $Q_1$] \label{lem:6}
	We have, for all $(k,l)\in \Z^2$,
	\begin{equation}
		Q_1(k,l) \geq \frac{2}{3\sqrt{3}}.
	\quad \text{ as well as } \quad
		Q_1(k,l) \geq \frac{2}{3 \sqrt{3}} (k^2 + l^2).
	\end{equation}
	\end{lemma}
\begin{proof}
	We write 
	\begin{equation}
		k = r \cos{\theta} \qquad \mbox{and} \qquad l = r \sin{\theta}.
	\end{equation}
	Then, after some elementary algebra, we have
	\begin{equation}
		Q_1(k,l) = \frac{2}{3\sqrt{3}} + \frac{6 r^2}{3 \sqrt{3}} + \frac{2 r (r \sin (t) \cos (t)+\sin (t)+\cos (t))}{\sqrt{3}}.
	\end{equation}
	Using $\sin(t) \cos(t) \geq -1/2$ and $\cos(t) + \sin(t) \geq -\sqrt{2}$, we see that for $r \geq 8.3 = \sqrt{68.89}$,
	\begin{align*}
		Q_1(k,l) & \geq \frac{2}{3\sqrt{3}} + \frac{3 r^2}{3 \sqrt{3}} + \frac{2 r (\sin (t)+\cos (t))}{\sqrt{3}} \\
		& \geq  \frac{2}{3\sqrt{3}} + \frac{3 r^2}{3 \sqrt{3}} - \frac{\sqrt{8} r }{\sqrt{3}} \geq \frac{2r^2}{3\sqrt{3}}.
	\end{align*}
	It thus suffices to check all lattice points $(k,l) \in \mathbb{Z}^2$ for which $k^2 + l^2 \leq 69$ and
	we see that the inequality holds for these.
\end{proof}

\subsubsection{Proof of Proposition \ref{prop1}: the range $1 \leq \alpha \leq 6$.}

We will reduce the special case $1 \leq \alpha \leq 6$ to an elementary argument using a finite amount of computation. It is relatively easy to verify (Lemma \ref{lem:smalla}) that the function is concave for $\sqrt{3}/2 \leq y \leq 1$ and $1 \leq \alpha \leq 6$: combined with Lemma \ref{lem:crit1}, this implies that the largest value in that interval is assumed for $y = \sqrt{3}/2$. It then remains to show (Lemma \ref{lem:onlyfor1}) that there cannot be any maximum outside this interval.

\begin{lemma} \label{lem:smalla}
	For $1 \leq \alpha \leq 6$ and in the region $\sqrt{3}/2 \leq y \leq 1$, the function $f_{\alpha}(y)$ has a unique maximum in $y=\sqrt{3}/2$. 
\end{lemma}
\begin{proof}
	We know (Lemma \ref{lem:crit1}) that there is a critical point in $y = \sqrt{3}/2$. It thus suffices to show that the second derivative is negative. This ends in explicit computations in the region 
	\begin{equation}
		(\alpha, y) \in (1,6) \times (\sqrt{3}/2, 1) = \Omega.
	\end{equation}
	The function $f_{\alpha}$ can be differentiated in closed form and the derivatives inherit the exponential decay (see Lemma \ref{lem:66}). The rapid exponential decay allows to observe that
	\begin{equation}
		\sup_{(\alpha, y) \in \Omega} \quad \frac{d^2 }{d y^2} f_{\alpha}\left( \frac{1}{2}, y \right) \leq -0.01.
	\end{equation}
	More precisely, we see that in the closure of that region the maximum of the second derivative is attained for $y=1$ and $\alpha = 6$ and evaluates roughly to $-0.0115$. We will not use this fact (which would require more extensive numerics) and use it for illustrative purposes: since we do not use the precise result $-0.0115$ and are content with a much weaker estimate (even $-0.001$ would suffice), the continuity and exponential decay of the involved functions reduces to the verification to a finite number of computations (see Lemma \ref{lem:66}).
\end{proof}

This establishes that there cannot be any maximum close to $y = \sqrt{3}/2$, it remains to rule out potential maxima satisfying $y \geq 1$. The next Lemma is quite crucial as it shows that functions of the form $\exp\left( - \pi \alpha \, \phi_{k,l}(y) \right)$ are uniformly small when $k$ and $l$ are large because $\phi_{k,l}(y)$ is uniformly large when $y \geq \sqrt{3}/2$. 
\begin{lemma} \label{growth1}
	We have
	$$\min_{y \geq \sqrt{3}/2} \phi_{k,l}(y) \geq \frac{1}{2} \sqrt{Q_1(k,l)}.$$
\end{lemma}
\begin{proof}
	We first establish the statement by hand for $|k|, |l| \leq 20$. In that region, the minimal ratio is at least $0.62$ and attained only for $(k,l)$ one of $(0,0), (-1,0), (0,-1)$ (these are the three dominant terms). We aim to prove a lower bound on
	$$ \phi_{k,l}(y) = \frac{(2 k+l+1)^2}{4 y}+y \left(l+\frac{1}{2} -\frac{1}{8 y^2}\right)^2.$$
	We note that for $y \geq \sqrt{3}/2$, we have
	$$ \frac{1}{3} \leq \frac{1}{2} -\frac{1}{8 y^2} \leq \frac{1}{2}.$$
	Therefore, we have, for $\sqrt{3}/2 \leq y$,
	\begin{equation}
        \left(l+\frac{1}{2} -\frac{1}{8 y^2}\right)^2 \geq \begin{cases}
    	l^2 \quad &\mbox{if}~l \geq 0 \\
    	\left(|l| - \frac{1}{2}\right)^2 \quad &\mbox{if}~l < 0 \end{cases} \geq \left(|l| - \frac{1}{2}\right)^2
    \end{equation}
	and thus
	$$ \phi_{k,l}(y) \geq  \frac{(2 k+l+1)^2}{4 y} +  y \left(|l| - \frac{1}{2}\right)^2.$$
	This leads us to consider general functions of the type
	$$ h(y) = \frac{A}{y} + B y,$$
	where $A,B > 0$. It is easy to see that $h$ assumes its global minimum in $y = (A/B)^{1/2}$ (with value $2(AB)^{1/2}$) and is then monotonically increasing. Since we only care about $y \geq \sqrt{3}/2$, this leads to two different cases: if $A/B \leq 3/4$, then the minimum is assumed for a value of $y \leq \sqrt{3}/2$ and, by convexity of the function $h$, the minimal value in the region of interest is thus assumed in $y = \sqrt{3}/2$. Conversely, if $A/B \geq 3/4$, then the minimal value is $2 (AB)^{1/2}$. The argument then follows from comparing quadratic forms (as above). We quickly note the precise results that result from this approach: the approach does not imply the result for small values of $(k,l)$ (here: $(0,-1)$ and $(-1,1)$) where the bound is too lossy; however, these cases have been excluded by hand before. More computations would imply that the constant $1/2$ could be increased to $4/5$ (except for the three dominant terms) but this will not be necessary.
\end{proof}

\begin{lemma} \label{lem:onlyfor1}
	For $1 \leq \alpha \leq 6$ and $y \geq 1$, we have
	\begin{equation}
		f_{\alpha}\left(\tfrac{1}{2}, y\right) \leq f_{\alpha}\left(\tfrac{1}{2}, \tfrac{\sqrt{3}}{2} \right).
	\end{equation}
\end{lemma}
\begin{proof}
	We start by noting that, as above, there is some room to maneuver since
	\begin{equation}
		f_{\alpha}\left(\tfrac{1}{2}, 1\right) \leq 0.99 \cdot f_{\alpha}\left(\tfrac{1}{2}, \tfrac{\sqrt{3}}{2} \right).
	\end{equation}
	The precise constant is $0.9875\dots$, which is only assumed for $\alpha=1$, then the ratio is steadily decreasing. For $\alpha = 6$, the ratio is approximately $0.8977$.
	Moreover, we claim that, for $1 \leq \alpha \leq 6$ and $y \geq 1$,
	\begin{equation}
		\sum_{\substack{k,l \in \mathbb{Z}\\|k|, |l| \leq 5}} e^{- \pi \alpha \, \phi_{k,l}(y)} - \sum_{\substack{k,l \in \mathbb{Z} \\ |k|, |l| \leq 2}} e^{- \pi \alpha \, \phi_{k,l}(y)} \leq \tfrac{1}{500} \cdot  f_{\alpha}\left(\tfrac{1}{2}, \tfrac{\sqrt{3}}{2} \right).
	\end{equation}
	This inequality is again far from sharp: the constant could be replaced by $1/1500$ for $\alpha = 1$ and for much smaller terms when $\alpha$ is even slightly bigger (the
	terms that do not cancel have faster exponential decay in $\alpha$). Using Lemma \ref{growth1}, we see that
	the remaining terms are actually quite small: we have
	\begin{align*}
		\sum_{\substack{(k,l) \in \mathbb{Z}^2 \\ \max \left\{ |k|, |l| \right\} > 5}} \hspace*{-9pt} e^{- \pi \alpha \, \phi_{k,l}(y)}
		\leq \hspace*{-9pt} \sum_{\substack{(k,l) \in \mathbb{Z}^2 \\ \max \left\{ |k|, |l| \right\} > 5}} \hspace*{-9pt} e^{- \pi \alpha \, \left(\min_{y \geq \sqrt{3}/2}\phi_{k,l}(y)\right)}
		\leq \hspace*{-9pt} \sum_{\substack{(k,l) \in \mathbb{Z}^2 \\ \max \left\{ |k|, |l| \right\} > 5}} \hspace*{-9pt} e^{- \pi \alpha \, \frac{1}{2} \sqrt{Q_1(k,l)}}.
	\end{align*}
	A short computation shows that if $ \max \left\{ |k|, |l| \right\} > 5$, then
	\begin{equation}
		Q_1(k,l) \geq  \tfrac{146}{3 \sqrt{3}} \sim 28.09
	\end{equation}
	and thus
	\begin{equation}
		\frac{1}{2} \sqrt{Q_1(k,l)} \geq 2.65.
	\end{equation}
	We note that these terms therefore undergo exponential decay many orders faster than the leading order terms which are of size
	\begin{equation}
		e^{-\frac{2\pi \alpha}{3 \sqrt{3}}} \sim  e^{- 0.38 \, \pi \alpha}.
	\end{equation}
	Indeed, we have, for example, for $\alpha \geq \alpha_0 \sim 0.95$ that
	\begin{equation}
		\sum_{\substack{(k,l) \in \mathbb{Z}^2 \\ \max \left\{ |k|, |l| \right\} > 5}} e^{- \pi \alpha \cdot  \frac{1}{2} \sqrt{Q_1(k,l)}} < \tfrac{1}{200} e^{-\frac{2\pi \alpha}{3 \sqrt{3}}}.
	\end{equation}
 \end{proof}

\subsubsection{Proof of Proposition \ref{prop1}: a restricted sum.} 
The purpose of this section is to show that close to the conjectured maximum $y \sim \sqrt{3}/2$, the second derivative of the first sum is negative in a way that can be quantified. We first illustrate this for a dominant term: when $\alpha \geq 6$, the function will be approximated by three leading terms: these are the terms corresponding to $(k,l)$ being $(0,0), (0,-1)$ or $(-1,0)$ (recall Figure \ref{fig_Q1}). These are also the values for which
\begin{equation}
	Q_1(k,l) = \frac{2}{3\sqrt{3}} \qquad \text{assumes its minimal value.}
\end{equation}
For these three tuples $(k,l)$, we have
\begin{equation}
	\phi_{k,l}(y) = \frac{\left(4 y^2+1\right)^2}{64 y^3},
\end{equation}
so three times the same function. We refer to these three functions as `the dominant terms'.

\begin{lemma} \label{lem:4}
Let $\alpha \geq 6$. The dominant terms are concave for $\sqrt{3}/2 \leq y \leq \sqrt{3}/2 + 1/(3\sqrt{\alpha})$. Moreover, in that interval, 
	\begin{equation}
		\frac{d^2}{dy^2} e^{-\pi \alpha \, \phi_{0,0}(y)} \leq -0.84 \, \alpha \, e^{-\frac{2\pi \alpha}{3 \sqrt{3}}}.
	\end{equation}
\end{lemma}

\begin{remark}
	This inequality has a little bit of wiggle room: a more careful analysis shows that the constant $-0.84$ could be replaced by roughly $-0.88$ (with equality if and only if $\alpha = 6$ and $y = \sqrt{3}/2 + 1/(3\sqrt{\alpha})$. Moreover, the constant could be replaced by even smaller constants as $\alpha$ increases (see also Lemma \ref{lem:44}).
\end{remark}

\begin{proof}
	Since all three functions are identical, it suffices to consider $(k, l) = (0,0)$. A short computation shows that 
	\begin{equation}
		\frac{d^2}{dy^2} e^{-\pi \alpha \, \phi_{0,0}(y)} =\frac{\pi  \alpha \, e^{-\frac{\pi  \alpha \left(4 y^2+1\right)^2}{64 y^3}} \left(\pi \alpha \left(-16 y^4+8 y^2+3\right)^2-256 y^3 \left(4
	   y^2+3\right)\right)}{4096 y^8}.
	\end{equation}
	We note that in the range
	\begin{equation}
		\frac{\sqrt{3}}{2} \leq y \leq \frac{\sqrt{3}}{2} + \frac{1}{3\sqrt{\alpha}} \leq \frac{\sqrt{3}}{2} + \frac{1}{3\sqrt{6}},
	\end{equation}
	we have the inequalities
	\begin{align*}
		& & \pi \alpha  \left(-16 y^4+8 y^2+3\right)^2 &\leq 1450 \, \pi \alpha  \left( y - \frac{\sqrt{3}}{2} \right)^2 \\
		& \text{ and } &
		- 256 y^3 \left(4 y^2+3\right) &\leq - 990 - 4500 \left( y - \frac{\sqrt{3}}{2} \right) \ll 0
	\end{align*}
	Using these two inequalities, we can bound the second derivative from above
	\begin{equation}
		\frac{\partial^2}{\partial y^2}  e^{-\pi \alpha \phi_{0,0}(y)} \leq  \frac{\pi  \alpha  e^{-\frac{\pi  \alpha  \left(4 y^2+1\right)^2}{64 y^3}}}{4096 y^8}  \left(1450 \, \pi \alpha  \left( y - \frac{\sqrt{3}}{2} \right)^2 - 990 - 4500 \left( y - \frac{\sqrt{3}}{2} \right) \right).
	\end{equation}
	A quick computation shows that in the range $\sqrt{3}/2 \leq y \leq \sqrt{3}/2 + 1/(3\sqrt{\alpha})$, we have (uniformly for $\alpha \geq 6$) that
	\begin{equation}
		1450 \, \pi \alpha \left( y - \frac{\sqrt{3}}{2} \right)^2 - 990 - 4500 \left( y - \frac{\sqrt{3}}{2} \right) \leq -480.
	\end{equation}
	The function $\left(4 y^2+1\right)^2/(64 y^3)$ is monotonically increasing for $y \geq \sqrt{3}/2$ 	from which we can conclude that 
	\begin{equation}
		\frac{\pi  \alpha  e^{-\frac{\pi  \alpha  \left(4 y^2+1\right)^2}{64 y^3}}}{4096 y^8} \qquad \mbox{assumes its minimum at}~y = \frac{\sqrt{3}}{2} + \frac{1}{3\sqrt{\alpha}}.
	\end{equation}
	An explicit computation shows that the exponential term satisfies
	\begin{equation}
		\exp \left(-\frac{\pi  \left(4 \left(\frac{1}{3 \sqrt{a}}+\frac{\sqrt{3}}{2}\right)^2+1\right)^2 \alpha}{64 \left(\frac{1}{3 \sqrt{a}}+\frac{\sqrt{3}}{2}\right)^3}\right) \geq e^{-\frac{2 \pi }{27 \sqrt{3}}} e^{\frac{2\pi \alpha}{3 \sqrt{3}}} \geq 0.87 e^{-\frac{2\pi \alpha}{3 \sqrt{3}}}.
	\end{equation}
	Combining all these estimates shows that
	\begin{equation}
		\frac{\partial^2}{\partial y^2}  e^{-\pi \alpha \phi_{0,0}(y)} \leq  0.87 \, \frac{\pi  \alpha  e^{-\frac{2\pi \alpha}{3 \sqrt{3}}}}{4096 y^8}  \left(1450 \pi \alpha  \left( y - \frac{\sqrt{3}}{2} \right)^2 - 990 - 4500 \left( y - \frac{\sqrt{3}}{2} \right) \right).
	\end{equation}
	Some quick computations show that, in the region of interest, the quadratic polynomial assumes its maximum value at $y = \sqrt{3}/2$ (that maximum being $-990$) as long as $\alpha < 72900/(841 \pi^2) \sim 8.78$  and assumes its maximum at  $y = \sqrt{3}/2 + 1/(3\sqrt{\alpha})$ for larger values of $\alpha$. A standard computation of the arising quantities for $6 \leq \alpha \leq 20$ shows the desired inequality for these cases.
	As $\alpha$ becomes large, this allows us to prove the asymptotic estimate
	\begin{equation}
		\frac{\partial^2}{\partial y^2}  e^{-\pi \alpha \phi_{0,0}(y)} \leq  0.87 \, \frac{\pi  \alpha  e^{-\frac{2\pi \alpha}{3 \sqrt{3}}}}{1296} \cdot (-480) \leq - \alpha  e^{-\frac{2\pi \alpha}{3 \sqrt{3}}}.
	\end{equation}
\end{proof}

\begin{lemma} \label{lem:44}
	Let $\alpha \geq 6$. If
	$\sqrt{3}/2 \leq y \leq \sqrt{3}/2 + 1/(3\sqrt{\alpha}),$
	then
	\begin{equation}
		\frac{d^2}{d y^2}  \sum_{-1 \leq k,l \leq 1} e^{ - \pi \alpha \cdot \phi_{k,l}(y)} \leq  - 2.5 \, \alpha \, e^{- \frac{2 \pi \alpha}{3\sqrt{3}}}.
	\end{equation}
\end{lemma}

\begin{remark}
	This result is almost a consequence of Lemma \ref{lem:4}. The sum contains the three of the dominant terms, so Lemma \ref{lem:4} already contributes the right hand side with a constant of $-0.84 \cdot 3 = -2.52$. The inequality now states that all the other terms do not contribute too much to this constant. Indeed, a quick numerical test shows that the other terms contribute $\sim 10^{-6}$ for $\alpha = 6$ and then even less as $\alpha$ increases (due to a different exponential growth).
\end{remark}

\begin{proof}
An explicit computation shows that
\begin{equation}
	\sum_{-1 \leq k,l \leq 1} e^{ - \pi \alpha \cdot \phi_{k,l}(y)}  = e^{-\frac{\pi  a \left(144 y^4+232 y^2+1\right)}{64 y^3}}  \cdot X,
\end{equation}
where
\begin{equation}
	X = \left(e^{\frac{3 \pi \alpha}{y}}+e^{\frac{4 \pi \alpha}{y}}+e^{\pi \alpha \left(2 y+\frac{3}{2 y}\right)}+2 e^{\pi \alpha \left(2 y+\frac{5}{2 y}\right)}+3 e^{\pi \alpha \left(2 y+\frac{7}{2 y}\right)}+1\right).
\end{equation}
   The result follows from lengthy (but standard) estimates as in Lemma \ref{lem:4}.
\end{proof}

\subsubsection{Proof of Proposition \ref{prop1}: gradient estimates}
The function $f_{\alpha}$ can be written as (this representation slightly deviates from the one above but is exactly the same)
\begin{equation}
	f_{\alpha}(y) = \sum_{k,l \in \mathbb{Z}} e^{- \pi \alpha \, \phi_{k,l}(y)},
\end{equation}
where
\begin{equation}
	\phi_{k,l}(y) = \frac{(2 k+l+1)^2}{4 y}+y \left(l-\frac{1}{8 y^2}+\frac{1}{2}\right)^2.
\end{equation}
The goal of this section is to deduce bounds on $\exp\left( - \pi \alpha \, \phi_{k,l}(y)\right)$, as well as its first and second derivative. By summation, this will then trivially lead to bounds on the first and second derivatives of $f_{\alpha}$. For functions of the form $\exp(-\pi \alpha \, \phi_{k,l}(y))$, we trivially have representations of the derivative as
\begin{align}
	\frac{d}{dy} e^{-\phi_{k,l}(y)} &= - \pi \alpha \, e^{-\phi_{k,l}(y)} \phi'(y) \\
	\frac{d^2}{dy^2} e^{-\phi_{k,l}(y)} &= \left(\pi^2 \alpha^2 \phi_{k,l}'(y)^2 -  \phi_{k,l}''(y) \right) e^{-\phi_{k,l}(y)} .
\end{align}
We recall that the quadratic form introduced above;
\begin{equation}
	Q_1(k,l) = \phi_{k,l}\left(\tfrac{\sqrt{3}}{2}\right) = \tfrac{2}{3 \sqrt{3}} + \tfrac{2}{\sqrt{3}} (k+l) + \tfrac{2}{\sqrt{3}}(k^2 + kl + l^2).
\end{equation}
We will deduce bounds on the first and second derivatives in terms of $Q_1(k,l)$.
Before discussing Lemma \ref{lem:66}, we quickly note how to establish inequalities like
\begin{equation}
	\forall~k,l \in \mathbb{Z} \qquad    \frac{1}{16} + \frac{\left|8 k^2+8 k (l+1)+2 l (l+1)+1\right|}{6} + \frac{(2 l + 1)^2}{4} \leq 2  \, Q_1(k,l)
\end{equation}
since we will encounter several more such inequalities throughout the paper. Both the left-hand side and the right-hand side are quadratic polynomials in $k,l$. Quadratic polynomials are dominated by their quadratic form. This means, that asymptotically, the growth is given by the quadratic terms corresponding to $2 \times 2$ quadratic forms which are easy to compare, the lower order terms lead to perturbations for small $k,l$ (not unlike the proof of Lemma \ref{lem:6}). In this case, a simple numerical analysis actually suggests the stronger inequality
\begin{equation}
	\forall~k,l \in \mathbb{Z} \qquad    \frac{1}{16} + \frac{\left|8 k^2+8 k (l+1)+2 l (l+1)+1\right|}{6} + \frac{(2 l + 1)^2}{4} \leq \frac{135 \sqrt{3}}{128} \, Q_1(k,l),
\end{equation}
where $135 \sqrt{3}/128 \sim 1.82$ is actually slightly smaller than 2, but this will not be necessary.

\begin{lemma} \label{lem:66}
We have
\begin{equation}\label{eq_phi'}
	\max_{y \geq \sqrt{3}/2} |\phi_{k,l}'(y)| \leq 2 \, Q_1(k,l),
\end{equation}
\begin{equation}\label{eq_phi''}
	\max_{y \geq \sqrt{3}/2} |\phi_{k,l}''(y)| \leq  3 \, Q_1(k,l).
\end{equation}

\end{lemma}
\begin{proof}
	Collecting powers of $y$, we have
	\begin{equation}
		\phi_{k,l}'(y) = \frac{\left(16 y^4 (2 l+1)^2 - 8 y^2 \left(8 k^2+8 k (l+1)+2 l (l+1)+1\right)-3\right)}{64 y^4}.
	\end{equation}
	As $y \geq \sqrt{3}/2$, we see that
	\begin{align}
		|\phi_{k,l}'(y)| &\leq  \frac{(2 l + 1)^2}{4} + \left|\frac{\left(8 k^2+8 k (l+1)+2 l (l+1)+1\right)}{8 y^2}\right| + \frac{1}{12} \\
		&\leq \frac{(2 l + 1)^2}{4} + \frac{\left|8 k^2+8 k (l+1)+2 l (l+1)+1\right|}{6} + \frac{1}{12}.
	 \end{align}
	 This should be compared to
	 \begin{equation}
	 	Q_1(k,l) = \phi_{k,l}(\sqrt{3}/2) =  \frac{2}{3 \sqrt{3}} + \frac{2}{\sqrt{3}} (k+l) + \frac{2}{\sqrt{3}}(k^2 + kl + l^2)
	 \end{equation}
	 and a short computation shows that
	 \begin{equation}
	 	\max_{y \geq \sqrt{3}/2} |\phi_{k,l}'(y)| \leq 2 \, Q_1(k,l).
	 \end{equation}
	 We do the same thing for the second derivative. Collecting powers of $y$, we obtain
	 \begin{equation}
	 	\phi_{k,l}''(y) = \frac{4 y^2 \left(8 k^2+8 k (l+1)+2 l (l+1)+1\right)+3}{16 y^5}.
	 \end{equation}
	 We note that the algebraic expression $ \left(8 k^2+8 k (l+1)+2 l (l+1)+1\right)$ is comprised of a positive semi-definite quadratic form with lower order perturbations: the quantity can be negative but only linearly in $k,l$ (while it grows quadratically in $k,l$). As a fairly standard computation shows, we have
	\begin{align}
	| \phi_{k,l}''(y)| &= \left| \frac{4 y^2 \left(8 k^2+8 k (l+1)+2 l (l+1)+1\right)+3}{16 y^5} \right| \leq 3 \, Q_1(k,l).
	\end{align}
	Again, as in many parts of the proof, stronger bounds are conceivable: an asymptotic analysis shows that as $|k|, |l|$ become large, the best constant converges to $8/3$ (which will not be needed for the argument).
\end{proof}
 
Using Lemma \ref{growth1}, we can immediately deduce Lemma \ref{growth11}.
 
\begin{lemma} \label{growth11}
	We have, for all $k,l \in \mathbb{Z}$,
	\begin{align} 
		\max_{y \geq \sqrt{3}/2}  \left| e^{- \pi \alpha \, \phi_{k,l}(y)}\right|    &\leq e^{-  \frac{\pi \alpha }{2}  \sqrt{Q_1(k,l)} }.
	\end{align}
	Moreover, for $\alpha \geq 5$, we have
	\begin{equation}
		\max_{y \geq \sqrt{3}/2}  \left| \frac{d^2}{dy^2} e^{- \pi \alpha \, \phi_{k,l}(y)}\right|  \leq 5 \pi^2 \alpha^2 Q_1(k,l) ^2 e^{-\frac{\pi \alpha }{2} \sqrt{Q_1(k,l)}}.
	\end{equation}
 \end{lemma}
\begin{proof}
	This follows quickly from the previous inequalities. Note that we have
	\begin{align}
		\left| \frac{d}{dy} e^{- \pi \alpha \, \phi_{k,l}(y)} \right| &= \pi \alpha  \left| \phi_{k,l}'(y) \right| e^{- \pi \alpha \, \phi_{k,l}(y)}\\
		&\leq 2 \pi \alpha \, Q_1(k,l) \, e^{-\frac{\pi \alpha }{2}  \sqrt{Q_1(k,l)}}. 
	\end{align}
	Taking second derivatives, we have
	\begin{align}
		\frac{d^2}{dy^2} e^{- \pi \alpha \, \phi_{k,l}(y)} = \left( \pi^2 \alpha^2 \phi_{k,l}'(y)^2 -  \pi \alpha \, \phi_{k,l}''(y) \right) e^{- \pi \alpha \, \phi_{k,l}(y)}
	\end{align}
	and applying the bounds from \eqref{eq_phi'} and \eqref{eq_phi''} leads to the desired result.
\end{proof}

\subsubsection{Proof of Proposition \ref{prop1}: bounding second derivatives}

\begin{lemma} \label{lem:conc}
	Let
	$$ f_{\alpha}(y) = \sum_{k,l \in \mathbb{Z}} e^{- \pi \alpha \, \phi_{k,l}(y)}.$$
	Then, for $\alpha \geq 6$ and for $\sqrt{3}/2 \leq y \leq \sqrt{3}/2 + 1/(3\sqrt{\alpha})$, we have
	$$ \frac{d^2}{dy^2} f_{\alpha}(y) < 0.$$
\end{lemma}
\begin{proof}
	Lemma \ref{lem:44} shows that, under the above assumptions,
	\begin{equation}
		\frac{d^2}{d y^2} \sum_{-1 \leq k,l \leq 1} e^{- \pi \alpha \, \phi_{k,l}(y)} \leq - 2.5 \, \alpha \, e^{- \frac{2 \pi \alpha}{3\sqrt{3}}}.
	\end{equation}
	It therefore suffices to ensure that
	\begin{equation}
		X = \frac{d^2}{d y^2} \hspace*{-9pt} \sum_{\substack{(k,l) \in \Z^2 \\ \max\left\{|k|,|l|\right\} > 1}} \hspace*{-9pt} e^{- \pi \alpha \, \phi_{k,l}(y)} \leq  2.5\, \alpha \, e^{- \frac{2 \pi \alpha}{3\sqrt{3}}}.
	\end{equation}
	Using the triangle inequality and Lemma \ref{growth11}, we have
	\begin{align}
		 X \leq \hspace*{-9pt} \sum_{\substack{(k,l) \in \Z^2 \\ \max\left\{|k|,|l|\right\} > 1}} \hspace*{-6pt} \max_{y \geq \sqrt{3}/2}\left| \frac{d^2}{d y^2} e^{- \pi \alpha \, \phi_{k,l}(y)} \right| 
		 \leq 5 \pi^2 \alpha^2 \hspace*{-9pt} \sum_{\substack{(k,l)\in\Z^2 \\ \max\left\{|k|,|l|\right\} > 1}} \hspace*{-6pt} Q_1(k,l)^2 e^{-\frac{\pi \alpha}{2}  \sqrt{Q_1(k,l)}}.
	 \end{align}
	 This bound, as it turns out, is actually sufficient: we have
	 \begin{equation}
	 	\min_{k,l \in \Z} Q_1(k,l) = \frac{2}{3\sqrt{3}} \sim 0.38\dots
	 \end{equation}
	 while
	\begin{equation}
		\min_{\substack{(k,l)\in\Z^2 \\ \max\left\{|k|,|l|\right\} > 1}} \hspace*{-9pt} Q_1(k,l) = \frac{14}{3\sqrt{3}} \sim 2.69\dots
	\end{equation}
	is by a factor 7 larger (compare figure \ref{fig_Q1}, \ref{fig_ellipse_c}). So, it is clear that for all $\alpha \geq \alpha_0$ sufficiently large
	\begin{equation}
		5 \pi^2 \alpha^2 \hspace*{-9pt} \sum_{\substack{(k,l)\in\Z^2 \\ \max\left\{|k|,|l|\right\} > 1}} Q_1(k,l)^2 e^{-\frac{ \pi \alpha}{2} \sqrt{Q_1(k,l)}} \leq 2.5 \, \alpha \, e^{- \frac{2 \pi \alpha}{3\sqrt{3}}}
	\end{equation}
	has to hold. We may pass to truncated geometric series to obtain closed expressions which can be compared. Numerically, we see that the desired inequality holds for $\alpha_0 \sim 5.96$.
\end{proof}

\subsubsection{Proof of Proposition \ref{prop1}: the maximum is close to $\sqrt{3}/2$}

\begin{lemma} \label{lem:prop1final}
	Let $\alpha \geq 6$ and
	\begin{equation}
		f_{\alpha}(y) = \sum_{k,l \in \mathbb{Z}} e^{- \pi \alpha \, \phi_{k,l}(y)} = \sum_{k,l \in \Z} f_{k,l}(y).
	\end{equation}
	The function $f_{\alpha}(y)$ assumes its maximal value on $[\sqrt{3}/2,\infty)$ for some $y \leq \sqrt{3}/2 + 1/(3\sqrt{\alpha})$.
\end{lemma}
\begin{proof}
	The combined value of the three dominant terms at $y = \sqrt{3}/2$ is $3 \exp\left(-2 \pi \alpha/(3 \sqrt{3})\right).$ A simple algebraic computation shows that for one of the three dominant terms (meaning that $(k,l)$ is either $(0,0)$ or $(-1,0)$ or $(0,-1)$)	 $f_{k,l}$ is monotonically decreasing and
	\begin{equation}
		f_{k,l}\left(\tfrac{\sqrt{3}}{2} + \tfrac{1}{c \sqrt{\alpha}}\right) \leq e^{-\frac{2 \pi  \left(9 c^2+10 \sqrt{3} c+9\right)}{9 c \left(\sqrt{3} c+2\right)^3}} f_{k,l}\left(\tfrac{\sqrt{3}}{2}\right).
	\end{equation}
	 Outside the $[\sqrt{3}/2, \sqrt{3}/2 + 1/(3\sqrt{\alpha})]$ interval, we can use this to argue that the contribution of the three dominant is at most, using $\alpha \geq 6$,
	 \begin{equation}
	 	3 \, e^{-\frac{2 \pi  \sqrt{\alpha } \left(27 \alpha +10 \sqrt{3} \sqrt{\alpha }+3\right)}{9 \left(3 \sqrt{3} \sqrt{\alpha }+2\right)^3}} e^{-\frac{2 \pi \alpha}{3 \sqrt{3}}} \leq 2.7 \, e^{-\frac{2 \pi \alpha}{3 \sqrt{3}}}.
	 \end{equation}
	 So there is substantial loss in the leading terms. The question is thus whether it is possible for the remaining terms to contribute anything of size at least $0.3$ times the contribution of the leading terms in $y=\sqrt{3}/2$. As it turns out, this does not happen: the inequality
	 \begin{equation}
	 	\sum_{\substack{k,l \in \mathbb{Z} \\ Q_1(k,l) \geq \frac{8}{3\sqrt{3}}}} e^{- \frac{\pi \alpha}{2}  \sqrt{Q_1(k,l)}} <  0.3 \, e^{- \frac{2 \pi \alpha}{3\sqrt{3}}}
	\end{equation}
	is true for $\alpha \geq \alpha_0$ where $\alpha_0 \sim 0.7$.
\end{proof}

\subsubsection{Proof of Proposition \ref{prop1}}
Lemma \ref{lem:smalla} implies that we only need to consider $\alpha \geq 6$. Lemma \ref{lem:prop1final} telling us that the maximum of $f_{\alpha}(y)$ for $\alpha \geq 6$ occurs at a point
\begin{equation}
	\tfrac{\sqrt{3}}{2} \leq y_{\max} \leq \tfrac{\sqrt{3}}{2} + \tfrac{1}{3\sqrt{\alpha}}.
\end{equation}
Lemma \ref{lem:conc} implies that $f_{\alpha}$ is concave in that region and Lemma \ref{lem:crit1} implies there is a critical point in $y = \sqrt{3}/2$ which therefore has to be the unique maximum.

\subsection{Proof of Proposition \ref{prop2}}
It remains to prove Proposition \ref{prop2}: for all $\alpha \geq 1$ and all $y \geq \sqrt{3}/2$, the function $ g_{\alpha}(y)$ assumes its maximum in $y = \sqrt{3}/2$. We will write $g_{\alpha}(y)$ as
\begin{equation}
	g_{\alpha}(y) = \sum_{k,l \in \mathbb{Z}} e^{- \pi \alpha \phi_{k,l}(y)} \psi_{k,l}(y)
\end{equation}
where the functions $\phi_{k,l}$ and $\psi_{k,l}$ are now (re-)defined via
\begin{align}
	\phi_{k,l}(y) &= \frac{1}{y} \left(k^2 +  k l + \left(\tfrac14 + y^2\right) l^2\right) \\
	\psi_{k,l}(y) &=  \cos\left(2\pi  \left(k \left( \tfrac{1}{2} - \tfrac{1}{8y^2}\right) - l \left( \tfrac{1}{4} + \tfrac{1}{16 y^2}\right) \right) \right).
\end{align}
Note that $\phi_{k,l}$ has the same name as the analogous function for the first sum but is a different term -- since it plays the same role and since the two proofs are similar but independent, we decided to keep the notation as is.
\begin{lemma} \label{lem:crit2}
	For all $\alpha \geq 1$, the function $g_{\alpha}(y)$ has a critical point in $y = \sqrt{3}/2$.
\end{lemma}
\begin{proof} This is a simple consequence of Lemma \ref{lem_hexagon_critical}. \end{proof}

\subsubsection{Proof of Proposition \ref{prop2}: introducing $Q_2$}
Arguing as above, we will ultimately show that $\exp(-\pi \alpha \, \phi_{k,l}(y)) \, \psi_{k,l}(y)$ can be controlled in terms of the value of $\phi_{k,l}(y)$ at $y = \sqrt{3}/2$. This value is given by the quadratic form
\begin{equation}
	\phi_{k,l}(\tfrac{\sqrt{3}}{2}) = \tfrac{2}{\sqrt{3}}(k^2 + kl +l^2).
\end{equation}
We will work with a slightly simplified form, which is only different up to a universal factor;
\begin{equation}
	Q_2(k,l) = k^2 +  k l + l^2 \geq \frac{k^2+l^2}{2}.
\end{equation}
This simplifies the way several inequalities are written -- there is no fundamental difference to the role that $Q_1$ played in the proof of Proposition \ref{prop1}, the difference is in the simplicity of some of the arising terms.

\begin{center}
	\begin{figure}[ht]
		\begin{tabular}{c|c}
			$Q_2(k,l)$  &  $(k,l)$\\
			\hline
			$1$ & (0, -1), (1, -1,), (-1, 0), (1, 0), (-1, 1), (0, 1)\\
			3 & $(1, -2), (-1, -1), (2, -1), (-2, 1), (1, 1), (-1, 2)$ \\
			4 & $(0, -2), (2, -2), (-2, 0), (2, 0), (-2, 2), (0, 2)$\\
		\end{tabular}
		\caption{The  values $(k,l)$ for which $0 < Q_2(k,l) \leq 6$.}
	\end{figure}
\end{center}
\vspace*{-10pt}
\begin{figure}[ht]
    \begin{framed}
    	\hfill
    	\includegraphics[width=.4\textwidth]{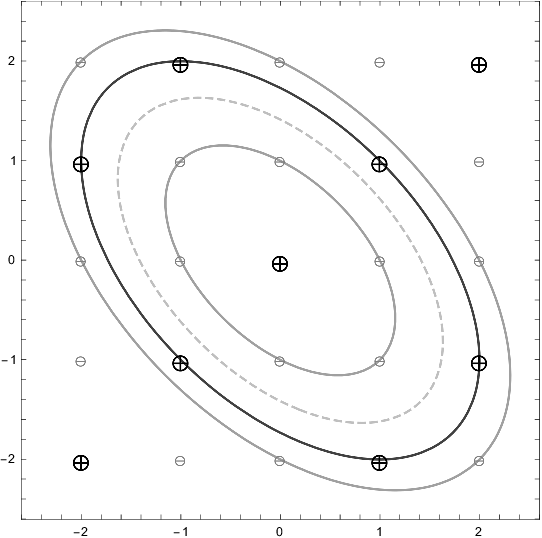}
    	\hfill
    	\includegraphics[width=.4\textwidth]{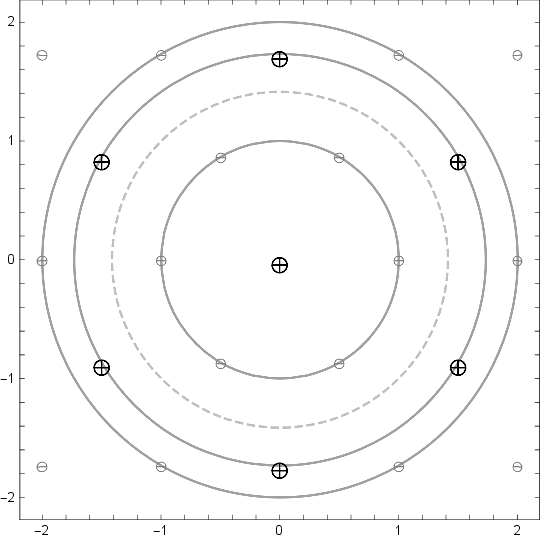}
    	\hspace*{\fill}
    	\caption{\textit{Left}: The quadratic form $Q_2$ collects lattice points with charges according to $\psi_{k,l}\left(\sqrt{3}/2\right) \in \{+1,-1/2\}$. The origin has a positive charge $+1$. There are 6 lattice points with negative charges of value $-1/2$ where $Q_2 \in \{1,4\}$ and 6 lattice points with positive charges $+1$ where $Q_2 \in \{3\}$. We note that $Q_2 \notin \{ 2 \}$ (dashed ellipse). \textit{Right}: After a linear change of coordinates, we see that we collect charges in the hexagonal lattice.}
    \end{framed}
\end{figure}

The analog of Lemma \ref{growth1}, showing that the minimal value of $\phi_{k,l}(y)$ grows with $k$ and $l$, is no longer true: we need to assume that $l \neq 0$. This condition is necessary: when $l = 0$, then 
\begin{equation}
	e^{- \pi \alpha \, \phi_{k,0}(y)} \, \psi_{k,0}(y) =e^{-\frac{\pi \alpha k^2}{y}} \cos\left(2 \pi k \left(\tfrac{1}{2} - \tfrac{1}{8y^2}\right) \right)
\end{equation}
which takes values $\pm 1$ as $y \rightarrow \infty$. This will be compensated by a completely different way of treating the $(k,0) \in \mathbb{Z}^2$ terms (in Lemma \ref{lem:mon}) and by proving substitute estimates restricted to $\sqrt{3}/2 \leq y \leq 1$ (in Lemma \ref{lem:25}).

\subsubsection{Proof of Proposition \ref{prop2}: the case $1 \leq \alpha \leq 5$}

\begin{lemma} \label{lem:a5}
If $1 \leq \alpha \leq 5$, then $g_{\alpha}(y)$ assumes its maximum in $y= \sqrt{3}/2$.
\end{lemma}

\begin{proof}
	We proceed as in the case of the first sum. Using Lemma \ref{lem:crit2}, we already know that there is a critical point in $y= \sqrt{3}/2$. The argument is based on a case distinction.
    
    \textit{Case 1.}
    We first deal with the case $1 \leq \alpha \leq 5$ and $y$ being close to the critical point in the sense of $\sqrt{3}/2 \leq y \leq 1.5$.  This case has the advantage of being a finite region in parameter space. We split the series (suppressing $\alpha$ in the notation) as follows
    \begin{equation}
        g_{\alpha}(y) = \sum_{\substack{(k,l) \in \Z^2 \\ Q_2(k,l) \leq 1}} e^{- \pi \alpha \phi_{k,l}(y)} \psi_{k,l}(y) + \sum_{\substack{(k,l) \in \Z^2 \\ Q_2(k,l) \geq 3}} e^{- \pi \alpha \phi_{k,l}(y)} \psi_{k,l}(y) = g_{\text{main}}(y) + g_{\text{rem}}(y).
    \end{equation}

    We set $u_{k,l}(y) = e^{-\pi \alpha \phi_{k,l}(y)} \psi_{k,l}(y)$, and termwise differentiation yields the explicit second derivative:
    \begin{equation}\label{eq:u_diff}
        u_{k,l}''(y) = e^{-\pi \alpha \phi_{k,l}(y)} \left( \psi_{k,l}'' - 2\pi \alpha \phi_{k,l}'\psi_{k,l}' - \pi \alpha \phi_{k,l}''\psi_{k,l} + \bigl(\pi \alpha \phi_{k,l}'\bigr)^2 \psi_{k,l} \right).
    \end{equation}
    To establish that $g_\alpha''(y) < 0$ globally, the remainder must not disrupt the sign of $g_{\text{main}}''(y)$. We further partition the $\alpha$-interval, governed by the damping strength: $\alpha \in  [1, 2]$ and $\alpha \in  (2, 5]$.

    \textit{Case 1a: $\alpha \in [1,2]$}. A direct numerical evaluation of the finite main sum shows it is actually $g_{\text{main}}''(y) \leq -1.5 \cdot 10^{-3}$ in this case. We split the remainder at a modest cutoff $R_1 = 6$:
    \begin{equation}
        |g_{\text{rem}}''(y)| \leq \sum_{3 \leq Q_2(k,l) < 6} |u_{k,l}''(y)| + \sum_{Q_2(k,l) \geq 6} |u_{k,l}''(y)|.
    \end{equation}
    Evaluating the finite sum ($3 \leq Q_2(k,l) \leq 6$) yields a bound of $1.2 \cdot 10^{-4}$. The infinite tail is bounded via a coarse integral comparison (analog to the case below) by $4.5 \cdot 10^{-6}$. Thus,
    \begin{equation}
        \sup_{(\alpha,y)\in[1,2]\times[\sqrt{3}/2,1.5]} |g_{\text{rem}}''(y)| \ll 1.5 \cdot 10^{-3},
    \end{equation}
    preserving the strict negativity of the second derivative.
    
    \textit{Case 1b: $\alpha \in (2,5]$.} The main part is now bounded by
    \begin{equation}
        g''_{\text{main}}(y) < -1.1142 \cdot 10^{-5}.
    \end{equation}
    We choose a cutoff threshold $R_2 = 10$:
    \begin{equation}
        |g_{\text{rem}}''(y)| \leq \sum_{3 \leq Q_2(k,l) < 10} |u_{k,l}''(y)| + \sum_{Q_2(k,l) \geq 10} |u_{k,l}''(y)| = \sum_{\text{near}} + \sum_{\text{far}}.
    \end{equation}
    Evaluating the finite sum using interval arithmetic yields $\sum \limits_{\text{near}} \leq 2.5 \cdot 10^{-7}$.

    We exploit the uniform ellipticity of the function $\phi_{k,l}(y) = \frac{(k + l/2)^2}{y} + l^2 y$. For $y \in [\sqrt{3}/2, 1.5]$, the coefficients are uniformly bounded from below, yielding the explicit relation to the quadratic form $Q_2(k,l) = k^2 + kl + l^2$:
    \begin{equation}\label{eq:ellipticity}
        \phi_{k,l}(y) \geq c_0 \left(k^2 + kl + l^2\right) = c_0 Q_2(k,l),
    \end{equation}
    where optimizing over the interval endpoints yields the sharp constant $c_0 = \frac{2}{3}$.

    Furthermore, the algebraic prefactors in \eqref{eq:u_diff} are safely enveloped by a polynomial growth condition. Maximizing over $y \in [\sqrt{3}/2, 1.5]$ and setting $\alpha = 5$ for the polynomials yields the uniform bound (compare Lemma \ref{lem:ell3})
    \begin{equation}
        |u_{k,l}''(y)| \leq C_1 \, Q_2(k,l)^2 e^{-\pi \alpha \phi_{k,l}(y)} \leq C_1 \, Q_2(k,l)^2 e^{-\pi \alpha c_0 Q_2(k,l)},
    \end{equation}
    where a direct worst-case coefficient collection over the interval bounds yields the explicit constant $C_1 = 157.8$. For the tail region $Q_2(k,l) \geq 10$, and since $\alpha > 2$, we dominate the series by an integral over the domain $\Omega = \{(x,y) \in \R^2 : x^2 + xy + y^2 \geq 10\}$.
    \begin{equation}
        \sum_{\text{far}} \leq \iint_{\Omega} C_1 (x^2 + xy + y^2)^2 e^{-2\pi c_0 (x^2 + xy + y^2)} \,dx \,dy.
    \end{equation}
    Applying a linear transformation to diagonalize the quadratic form, followed by a change to polar coordinates, simplifies the area integral directly to:
    \begin{equation}
        \sum_{\text{far}} \leq \frac{4\pi C_1}{\sqrt{3}} \int_{\sqrt{10}}^{\infty} r^5 e^{-2\pi c_0 r^2} \, dr = \frac{2\pi C_1}{\sqrt{3}(2\pi c_0)^3} \left( (20\pi c_0)^2 + 2(20 \pi c_0) + 2 \right) e^{-20\pi c_0}.
    \end{equation}
    Substituting $c_0 = \frac{2}{3}$ and $C_1 = 157.8$ provides the explicit analytical bound:
    \begin{equation}
        \sum_{\text{far}} \leq \frac{157.8 \cdot 2\pi}{\sqrt{3} (4\pi/3)^3} \left[ \left(\frac{40\pi}{3}\right)^2 + \frac{80\pi}{3} + 2 \right] e^{-40\pi/3} < 10^{-14}.
    \end{equation}
    Combining the estimates, we obtain
    \begin{equation}
        |g_{\text{rem}}''(y)| \leq 2.5 \cdot 10^{-7} + 10^{-14} < 3 \cdot 10^{-7}.
    \end{equation}
    Consequently, the second derivative is strictly negative in this case.

    \textit{Case 2a.}  The case $y \geq 1.5$ and $1.5 \leq \alpha \leq 5$ will follow from Lemma \ref{lem:twodecay} further below. Lemma \ref{lem:twodecay} shows that the maximum is confined to a small neighborhood of $\sqrt{3}/2$, so it cannot occur in $y\geq 1.5$ when $\alpha \geq 1.5$.
    
    \textit{Case 2b.} It remains to analyze $1 \leq \alpha \leq 1.5$, $y \geq 1.5$. This domain is straightforward to handle since $y \geq 1.5$ enforces a strict separation from the critical point. In this regime, the first seven terms of the lattice sum, corresponding to the origin $(k,l) = (0,0)$ and the six pairs satisfying $k^2 + kl + l^2 = 1$, provide a very good approximation of the series. 

    As $y$ increases, the term $l^2 y$ in $\phi_{k,l}(y)$ grows linearly, which causes all terms with $l \neq 0$ to decay exponentially fast. A direct interval evaluation of these dominant terms confirms a substantially uniform gap. The computations for the first seven terms ($Q_2(k,l)\leq 1$) yield
    \begin{equation}
        g_{\text{main}}\left( \frac{\sqrt{3}}{2}\right) - g_{\text{main}}(1.5) \geq \frac{1}{20} \quad \text{for all } 1 \leq \alpha \leq 1.5.
    \end{equation}
    The remaining unbounded tail, where $y > 1.5$, can then be handled using the integral estimates from the previous case or standard geometric series arguments.
\end{proof}

\begin{lemma} \label{growth2}
	Suppose $l \neq 0$. Then
	\begin{equation}
		\min_{y \geq \sqrt{3}/2} \quad \frac{1}{y} \left(k^2 +  k l + \left(\frac{1}{4} + y^2 \right) l^2\right) \geq \sqrt{Q_2(k,l)}
	\end{equation}
\end{lemma}
\begin{proof}
	Let $l \neq 0$. The only interesting case is when $k$ and $l$ have opposite signs; if they have the same sign, we can flip one of the signs to obtain a function with a smaller minimum. Therefore, we now assume that $k$ and $l$ have opposite signs and that $l \neq 0$. The function is positive in the region of interest since $k^2 + k l + l^2 \geq 1$. First, observe that the function is convex
	$$ \frac{d^2}{dy^2} ~ \frac{1}{y} \left(k^2 +  k l + \left(\frac{1}{4} + y^2 \right) l^2\right)  = \frac{(2k+l)^2}{2y^3} \geq 0.$$
	This shows that minimum is either at $y = \sqrt{3}/2$ or at a point where the derivative vanishes. The derivative vanishes at
	$$ y_0 = \left|\frac{1}{2} + \frac{k}{l}\right|.$$
	We recall that $k$ and $l$ have opposite signs, therefore the term $k/l$ is negative. There is now a simple case distinction: if $k/l \leq -\sqrt{3}/2 - 1/2$, then $y_0 \geq \sqrt{3}/2$ and the minimum is assumed at the point where the derivative vanishes. Otherwise, the derivative vanishes earlier, and the minimum is assumed in $y = \sqrt{3}/2$. This second case is trivial, since then
	$$   \frac{1}{y} \left(k^2 +  k l + \left(\frac{1}{4} + y^2 \right) l^2\right)  \geq \frac{2}{\sqrt{3}} Q_2(k,l).$$
	If, on the other hand, $k/l \leq -\sqrt{3}/2 - 1/2$, then a short computation shows that
	$$ \min_{y \geq \sqrt{3}/2} \quad \frac{1}{y} \left(k^2 +  k l + \left(\frac{1}{4} + y^2 \right) l^2\right) = -l(2k+l).$$
	This can only be small when $2k+ l \sim 0$, but this cannot happen because  $k/l \leq -\sqrt{3}/2 - 1/2$ which implies
	$  -l(2k+l) \geq \sqrt{3} \, l^2$. This then proves the result unless
	$$\left(  \sqrt{3} \, l^2 \right)^2 = 3 \, l^4 \leq k^2 +kl + l^2 = Q_2(k,l).$$
	If this is the case, then certainly also $2 \, l^4 \leq k^2 + kl $ which implies 
	$$ |k| \geq \frac{1}{2} \left(\sqrt{8 \, l^4+l^2}-l\right) \geq l^2.$$
	In that regime, we have
	$$ \min_{y \geq \sqrt{3}/2} \quad \frac{1}{y} \left(k^2 +  k l + \left(\frac{1}{4} + y^2 \right) l^2\right) = -l(2k+l) \sim 2 |k| |l|$$
	while our lower bound is clearly smaller
	$$ \sqrt{Q_2(k,l)} = \sqrt{k^2 + kl + l^2} \sim |k|.$$
\end{proof}

\begin{lemma}[Bounds on the function] \label{lem:ell1}
We have, for $l \neq 0$,
$$ \left| e^{- \pi \alpha \, \phi_{k,l}(y)} \, \psi_{k,l}(y)\right| \leq e^{-\pi \alpha \sqrt{Q_2(k,l)}} .$$
\end{lemma}
\begin{proof}
	This follows immediately from Lemma \ref{growth2} and $|\psi_{k,l}(y)| \leq 1$.
\end{proof}

\begin{lemma}[Second derivative bounds] \label{lem:ell3}
	We have, for $l \neq 0$, $y \geq \sqrt{3}/2$, $\alpha \geq 5$,
	\begin{align}
		 \left|  \frac{d^2}{dy^2} \left( e^{- \pi \alpha \phi_{k,l}(y)} \, \psi_{k,l}(y) \right)\right| &\leq  3 \pi^2 \alpha^2  Q_2(k,l)^2 \cdot e^{-\pi \alpha \sqrt{Q_2(k,l)}}.
	\end{align}
\end{lemma}
\begin{proof}
	The product rule yields
	\begin{align}
		\frac{d}{dy} \left(e^{- \pi \alpha \, \phi_{k,l}(y)} \, \psi_{k,l}(y) \right) = \left(- \pi \alpha \, \phi_{k,l}'(y)  \, \psi_{k,l}(y) + \psi_{k,l}'(y)\right) e^{- \pi \alpha \, \phi_{k,l}(y)}.
	\end{align}
	For the first term, we need an upper bound on 
	$$ |\phi_{k,l}'(y)| = \left| l^2 - \frac{(2k+l)^2}{4y^2} \right| \leq l^2 + \frac{(2k + l)^2}{3} \leq \frac{4}{3} Q_2(k,l).$$
	For the second term, we need an upper bound on
	$$ | \psi_{k,l}'(y)| \leq 2\pi \left(\frac{k}{4y^3} + \frac{l}{8 y^3} \right) \leq 2.5 |k| + 1.25 |l| \leq 4\sqrt{Q_2(k,l)}.$$
	Combining both, we get
	$$ \left|  \frac{d}{dy} \left( e^{- \pi \alpha \, \phi_{k,l}(y)} \, \psi_{k,l}(y) \right)\right| \leq \left( \frac{4}{3} \pi \alpha \, Q_2(k,l) + 4\sqrt{Q_2(k,l)} \right)e^{-\pi \alpha \sqrt{Q_2(k,l)}} .$$
	Since $\alpha \geq 5$, we have
	$$ 4\sqrt{Q_2(k,l)} \leq 4 Q_2(k,l) \leq  \left( \frac{4}{5 \pi} \right) \pi \alpha Q_2(k,l)$$
	and the result follows. The product rule further yields
	\begin{align}
		\frac{d^2}{dy^2} \left(e^{- \pi \alpha \, \phi_{k,l}(y)} \, \psi_{k,l}(y)\right) = \;
		& \Big(
        \psi_{k,l}''(y)
		- \pi \alpha \, \phi_{k,l}''(y) \, \psi_{k,l}(y)\\
        & \quad
		+ \pi^2 \alpha^2 \phi_{k,l}'(y)^2 \, \psi_{k,l}(y)
		- 2\pi \alpha \, \phi_{k,l}'(y) \, \psi'_{k,l}(y)
		\Big) e^{- \pi \alpha \, \phi_{k,l}(y)}.
	\end{align}
	We need two new ingredients: the first is
	$$ |\phi_{k,l}''(y)| = \frac{(2k+l)^2}{2y^3} \leq 4 \cdot Q_2(k,l).$$
	The second is
	\begin{align*}
		| \psi_{k,l}''(y)| & \leq 4\pi^2 \left(\frac{k}{4y^3} + \frac{l}{8 y^3} \right)^2  + 2\pi \left( \frac{3 |k|}{4y^4} + \frac{3|l|}{8 y^4}\right) \\
		& \leq 16 \, Q_2(k,l) + 2 \pi \left( \frac{4 |k|}{3} + \frac{2|l|}{3 } \right) \\
		& \leq 16 \, Q_2(k,l) + 13 \, Q_2(k,l) = 29 \, Q_2(k,l).
	\end{align*}
	It remains to combine all the estimates. We have
	\begin{align*}
		\left| \pi \alpha \, \phi_{k,l}''(y) e^{- \pi \alpha \, \phi_{k,l}(y)} \, \psi_{k,l}(y) \right|
        & \leq 4 \pi \alpha \, Q_2(k,l) e^{-\pi \alpha \sqrt{Q_2(k,l)}} \\
		\left| \pi^2 \alpha^2 \phi_{k,l}'(y)^2 e^{- \pi \alpha \, \phi_{k,l}(y)} \,\psi_{k,l}(y) \right|
        & \leq  \tfrac{16}{9} \pi^2 \alpha^2  Q_2(k,l)^2 e^{-\pi \alpha \sqrt{Q_2(k,l)}}\\
		\left| -2 \pi \alpha \, \phi_{k,l}'(y)e^{- \pi \alpha \, \phi_{k,l}(y)} \, \psi'_{k,l}(y) \right|
        & \leq \tfrac{32}{3}  \pi \alpha \,Q_2(k,l)^{3/2} e^{-\pi \alpha \sqrt{Q_2(k,l)}}\\
	    \left| e^{- \pi \alpha \, \phi_{k,l}(y)} \, \psi_{k,l}''(y) \right|
        & \leq 29 \, Q_2(k,l) e^{-\pi \alpha \sqrt{Q_2(k,l)}}\\
	    \left| - \pi \alpha \, \phi_{k,l}'(y) e^{- \pi \alpha \, \phi_{k,l}(y)} \, \psi_{k,l}(y) \right|
        & \leq \tfrac{4}{3} \pi \alpha \, Q_2(k,l) e^{-\pi \alpha \sqrt{Q_2(k,l)}}.
	 \end{align*}
	 Using $Q_2(k,l) \geq 1$ and $\alpha \geq 5$, by adding all these terms and bounding them from above, we end up getting the desired bound.
\end{proof}

\subsubsection{The case $l=0$}
We will now deal with the terms  $(k,l) \in \mathbb{Z}^2$ for which $l = 0$. These are the terms for which Lemma \ref{lem:ell1} and Lemma \ref{lem:ell3} do not apply. In contrast to other arguments in this paper, we will not deal with these terms on an individual basis but instead investigate their sum
\begin{equation}
	\sum_{k \in \mathbb{Z}} e^{- \pi \alpha \phi_{k,0}(y)} \, \psi_{k,0}(y) =  \sum_{k \in \mathbb{Z}} e^{- \frac{\pi\alpha k^2}{y}} \, \cos \left( 2\pi  k \left(  \tfrac{1}{2} - \tfrac{1}{8y^2} \right) \right).
\end{equation}
We recall the standard formula for the heat kernel on the torus $\mathbb{T} \cong [0,1]$.
\begin{equation}
	u(t,x) = 1 + 2\sum_{k=1}^{\infty} e^{- 4 \pi^2 k^2 t} \cos\left(2 \pi k x\right).
\end{equation}
This is the same expression as above for 
$$ 4 \pi^2 t = \frac{\alpha \pi }{y} \qquad \mbox{or} \qquad t(y) = \frac{\alpha}{4 \pi y}$$
evaluated at 
$$ x(y) = \frac{1}{2} - \frac{1}{8y^2}.$$
This naturally leads us to consider some basic properties of the heat kernel.
\begin{lemma}\label{lem:heat}
	The heat kernel $u(t,x)$ on $\mathbb{T} \cong [0,1]$ centered in $x=0$ has two inflection points. The inflection point in $(0,1/2)$ is smaller than $0.3$ for all $t>0$. In particular, the function $u(t,x)$ is convex for $0.3 \leq x \leq 0.7$.
\end{lemma}
\begin{proof}
	For small values of $t$, we have the Gaussian short-time asymptotic
	\begin{equation}
		u(t,x) \sim \tfrac{1}{\sqrt{4 \pi t}} \, e^{-\frac{x^2}{4t}},
	\end{equation}	
	placing the inflection point asymptotically at $x_0 \sim \sqrt{2t} \ll 0.3$. As time becomes large, we have dominance of the leading order term and expect
	\begin{equation}
		u(t,x) \sim 1 + 2e^{-4 \pi^2 t} \cos{(2 \pi x)}
	\end{equation}
	having the inflection point at $0.25 < 0.3$. For intermediate times $t \sim 1$, some basic estimates show that the critical point actually wanders continuously from 0 to $1/4$. Since we do not require sharp bounds, the case can be dealt with basic asymptotics.
\end{proof}

\begin{lemma} \label{lem:mon}
	For all $\alpha >0$ and all $y \geq \sqrt{3}/2$, the function
	\begin{equation}
		1 + 2\sum_{k=1}^{\infty} e^{- \frac{\pi \alpha}{y} k^2  } \cos\left( 2k \pi  \left(\tfrac{1}{2}-\tfrac{1}{8 y^2}\right)\right)
	\end{equation}
	is monotonically decreasing in $y$.
\end{lemma}
\begin{proof}
	We write the expression as $u(t(y), x(y))$ with
	\begin{equation}
		t(y) = \frac{\alpha}{4 \pi y} \qquad \mbox{and} \qquad x(y) = \frac{1}{2} - \frac{1}{8y^2}.
	\end{equation}
  	Note that, since $y \geq \sqrt{3}/2$, we have $1/3 \leq x(y) \leq 1/2$ and, in particular, Lemma \ref{lem:heat} is always applicable. We note that, using the chain rule,
   $$ \frac{d}{dy} u(t(y), x(y)) = \frac{d}{dt} u(t(y), x(y)) \frac{dt(y)}{dy} +  \frac{d}{dx} u(t(y), x(y)) \frac{dx(y)}{dy}.$$
Since $u(t,x)$ is the heat kernel, it satisfies the heat equation
$$ \frac{d}{dt} u(t,x) = \frac{d^2}{dx^2} u(t,x).$$
We have some information about the second derivatives of the heat kernel in the region of interest thanks to Lemma \ref{lem:heat}
and can conclude that this term is positive. Moreover, we have
$$  \frac{dt(y)}{dy} = - \frac{\alpha}{4 \pi y^2} < 0 \qquad \mbox{and} \qquad \frac{dx(y)}{dy} = \frac{1}{4y^3} > 0.$$
Therefore,
   $$ \frac{d}{dy} u(t(y), x(y)) =  \underbrace{\frac{d^2}{dx^2} u(t(y), x(y))}_{>0} \underbrace{\frac{dt(y)}{dy}}_{<0} +  \underbrace{\frac{d}{dx} u(t(y), x(y))}_{<0} \underbrace{\frac{dx(y)}{dy}}_{>0} < 0.$$
\end{proof}

\subsubsection{Proof of Proposition \ref{prop2}: the maxima are in $[\sqrt{3}/2, \sqrt{3}/2 + 1/(4\sqrt{\alpha})]$}
The purpose of this subsection is to prove that when $\alpha \geq 5$, then the global maximum has to be in the region $[\sqrt{3}/2, \sqrt{3}/2 + 1/(4\sqrt{\alpha})]$. The argument comes in two parts: in the first part we sum over $(k,l) \in \mathbb{Z}^2$ with $l = 0$ to recover the heat kernel. We will show that the heat kernel is smaller in $y \geq \sqrt{3}/2 + 1/(4\sqrt{\alpha})$ than it is in $y = \sqrt{3}/2$ and we quantify the difference. The second step of the argument is to show that all the other terms cannot compensate for that difference which implies the statement. We abbreviate
\begin{equation}	
	G_{\alpha}(y) = \sum_{k \in \mathbb{Z}} e^{- \frac{\pi\alpha k^2}{y}} \cos \left( 2\pi  k \left(  \tfrac{1}{2} - \tfrac{1}{8y^2} \right)  \right)
\end{equation}
\begin{lemma} \label{lem:heat_est}
	We have, for $\alpha \geq 1$, that
	\begin{equation}
		\max_{y \geq \frac{\sqrt{3}}{2} + \frac{1}{4\sqrt{\alpha}}} G_{\alpha}(y) = G_{\alpha}\left(\tfrac{\sqrt{3}}{2} + \tfrac{1}{4\sqrt{\alpha}} \right) \leq  G_{\alpha}\left(\tfrac{\sqrt{3}}{2} \right) - \tfrac{2 \sqrt{a}}{3} e^{-\frac{2 \pi \alpha}{\sqrt{3}}}.
	\end{equation}
\end{lemma}
\begin{proof}
The first part of the statement,
\begin{equation}
	\max_{y \geq \frac{\sqrt{3}}{2} + \frac{1}{4\sqrt{\alpha}}} G_{\alpha}(y) = G_{\alpha}\left(\tfrac{\sqrt{3}}{2} + \tfrac{1}{4\sqrt{\alpha}} \right)
\end{equation}
follows immediately from Lemma \ref{lem:mon}. We use Lemma \ref{lem:mon} once more: instead of evaluating the heat kernel at $(t(\sqrt{3}/2), x(\sqrt{3}/2))$ and comparing it to
\begin{equation}
	\text{the heat kernel in} \quad \left(t\left(\tfrac{\sqrt{3}}{2} + \tfrac{1}{4\sqrt{\alpha}}\right), x\left(\tfrac{\sqrt{3}}{2} + \tfrac{1}{4\sqrt{\alpha}}\right)\right)
\end{equation}
we may as well use the monotonicity in $x$ once more to compare $(t(\sqrt{3}/2), x(\sqrt{3}/2))$ with $(t(\sqrt{3}/2 + 1/(4\sqrt{\alpha})), x(\sqrt{3}/2))$. It suffices to estimate the difference between
\begin{equation}
	\sum_{k \in \mathbb{Z}} e^{- \frac{\pi\alpha  k^2 }{\sqrt{3}/2}} \, \cos \left( \tfrac{ 2\pi  k}{3}  \right)
	\quad \text{ and } \quad
	\sum_{k \in \mathbb{Z}} e^{- \frac{\pi\alpha  k^2 }{\sqrt{3}/2 +1/(4\sqrt{\alpha})}} \, \cos \left( \tfrac{ 2\pi  k}{3}  \right).
\end{equation}
Both sums are, due to the rapid decay of its terms, essentially given by their first term as soon as $\alpha$ is sufficiently large and
\begin{equation}
	\sum_{k \in \mathbb{Z}} e^{- \pi\alpha  k^2 t} \, \cos \left( \tfrac{ 2\pi  k}{3}  \right) \sim 1 - e^{- \pi\alpha t} + \mbox{l.o.t.}
\end{equation}
This requires us to estimate $\exp(- \pi \alpha t)$ for two nearly adjacent values of $t$, these being
$$ t_1 = \frac{1}{\sqrt{3}/2} \qquad \mbox{and} \qquad t_2 = \frac{1}{\sqrt{3}/2 + 1/(4\sqrt{\alpha})}.$$
We see that the leading order difference is of the order given by
$$ e^{- \pi\alpha  t_2} - e^{- \pi\alpha  t_1}.$$
The derivative of $\exp\left(- \pi \alpha t\right)$ is $ -\pi \alpha \exp\left(- \pi \alpha t\right)$ and thus, using the mean value theorem
\begin{align*}
	e^{- \pi\alpha  t_2} - e^{- \pi\alpha  t_1} &\geq \pi \alpha \, (t_1-t_2) e^{- \pi \alpha t_1}\\
	&\geq \frac{\pi \alpha}{4 \sqrt{\alpha}} \frac{1}{\frac{\sqrt{3}}{2} \left( \frac{\sqrt{3}}{2}  + \frac{1}{4\sqrt{\alpha}}\right)} \, e^{-\frac{2 \pi \alpha}{\sqrt{3}}}\\
	&= \frac{2 \pi  \alpha}{6 \sqrt{\alpha}+\sqrt{3}} \, e^{-\frac{2 \pi \alpha}{\sqrt{3}}}\\
	&\geq \frac{2 \sqrt{\alpha}}{3} \, e^{-\frac{2 \pi \alpha}{\sqrt{3}}}.
\end{align*}
The inequality is easily verified to be effective for $1 \leq \alpha \leq 10$ after which the asymptotic analysis is more than accurate.
\end{proof}

The next lemma shows that this loss in the leading term of the heat kernel cannot be compensated by the other remaining terms.

\begin{lemma} \label{lem:twodecay}
	We have, for $\alpha \geq 1.5$, that
	\begin{equation}
		g_{\alpha}(y) = \sum_{k, l \in \mathbb{Z}} e^{- \frac{\pi\alpha}{y} \left(k^2 +  k l + \left(\frac14 + y^2\right) l^2 \right)} \cos \left( 2\pi  (k a_2(y) - l a_1(y) ) \right)
	\end{equation}
	assumes its maximum in $\sqrt{3}/2 \leq y \leq \sqrt{3}/2 + 1/(4\sqrt{\alpha})$. More precisely: outside that region, the heat kernel terms ($k \in \mathbb{Z}, l = 0$) have decayed more than can be compensated for by the remaining terms.
\end{lemma}
\begin{proof}
	We decompose the indices into three sets
	\begin{align*}
	 \mathbb{Z}^2 = & \left\{(k,0): k \in \mathbb{Z}  \right\} \\
	 \cup & \left\{(k,l): Q_2(k,l) = 1 ~\mbox{and}~l \neq 0\right\} \\
	 \cup & \left\{(k,l): Q_2(k,l) > 1 ~\mbox{and}~l \neq 0\right\}.
	 \end{align*}
	The first set gives rise to the quantity resembling a heat kernel which has been analyzed in Lemma \ref{lem:heat} and is well behaved: it decays away from $y = \sqrt{3}/2$ at a controlled rate. The second set is small and completely explicit
	$$ \left\{(k,l): Q_2(k,l) = 1 ~\mbox{and}~l \neq 0\right\} =  \left\{(0,-1), (1,-1), (-1,1), (0,1) \right\}.$$
	The third set is not small but will amount to only a small contribution. We first sum over the four terms in the second set
	$$ \sum_{\substack{(k,l) \in \mathbb{Z}^2 \\ Q_2(k,l) = 1 , l\neq0}}  \hspace*{-9pt} e^{- \pi \alpha \, \phi_{k,l}(y)} \psi_{k,l}(y) = -2 \, e^{-\frac{\pi  a}{y}} \cos \left(\tfrac{\pi }{4 y^2}\right).$$
	This term is clearly decreasing in $y$. It remains to bound the rest. We note
	\begin{align}
		e^{- \pi \alpha \, \phi_{k,l}(y)} \, \psi_{k,l}(y) &\leq  e^{- \pi \alpha \, \phi_{k,l}(y)} \leq  e^{- \pi \alpha \min_{y \geq \sqrt{3}/2} \phi_{k,l}(y)}.
	\end{align}
	Appealing to Lemma \ref{growth2}, we can bound the remaining sum from above by
	\begin{equation}
		\sum_{\substack{(k,l) \in \mathbb{Z}^2, \\ Q_2(k,l) > 1, l \neq 0}}  e^{- \pi \alpha \, \phi_{k,l}(y)} \, \psi_{k,l}(y) \leq \sum_{\substack{(k,l) \in \mathbb{Z}^2, \\ Q_2(k,l) > 1, l \neq 0}} e^{- \pi \alpha \sqrt{Q_2(k,l)}}.
	\end{equation}
	The difference in the asymptotic decay shows that 
	\begin{equation}
		\sum_{\substack{(k,l) \in \mathbb{Z}^2, \\ Q_2(k,l) > 1, l \neq 0}} e^{- \pi \alpha \sqrt{Q_2(k,l)}} \leq \tfrac{2 \sqrt{\alpha}}{3} \, e^{-\frac{2 \pi \alpha}{\sqrt{3}}}
	\end{equation}
	has to hold for $\alpha \geq \alpha_0$ sufficiently large. Basic numerics show that $\alpha_0 \sim 1.3$. 
\end{proof}

\subsubsection{Proof of Proposition \ref{prop2}: Concavity in $[\sqrt{3}/2, \sqrt{3}/2+1/(4\sqrt{\alpha})]$}
The purpose of this section is to show that for $\alpha \geq 5$, the function 
\begin{equation}
	g_{\alpha}(y) = \sum_{k,l} e^{- \pi \alpha \phi_{k,l}(y)} \psi_{k,l}(y)
\end{equation}
is concave in the region $\sqrt{3}/2 \leq y \leq \sqrt{3}/2 + 1/(4\sqrt{\alpha})$. Together with Lemma \ref{lem:twodecay}, this then establishes Proposition \ref{prop2}. We start by quickly establishing a substitute result for Lemma \ref{lem:ell3} to deal with the case $l=0$. When studying the full sum
\begin{equation}
	g_{\alpha}(y) = \sum_{k,l} e^{- \pi \alpha \phi_{k,l}(y)} \, \psi_{k,l}(y),
\end{equation}
the functions for $l=0$ simplify to
\begin{equation}
	e^{- \pi \alpha \, \phi_{k,0}(y)} \psi_{k,0}(y) = e^{- \frac{ \pi \alpha k^2}{y}} \cos\left(2\pi  k \left( \tfrac{1}{2} - \tfrac{1}{8y^2}\right)  \right)
\end{equation}

\begin{lemma} \label{lem:25}
	We have, for $\sqrt{3}/2 \leq y \leq 1$ and $k \neq 0$, that
	\begin{align} 
		\max_{ \sqrt{3}/2 \leq y \leq 1}  \left| e^{- \pi \alpha \, \phi_{k,0}(y)} \psi_{k,0}(y) \right| &\leq  e^{-\pi \alpha k^2}\\
		\max_{ \sqrt{3}/2 \leq y \leq 1} \left| \frac{d^2}{dy^2} \left(e^{- \pi \alpha \, \phi_{k,0}(y)} \psi_{k,0}(y) \right) \right|  &\leq  20 \, \alpha^2 k^4 e^{-\pi \alpha k^2}
	\end{align}
\end{lemma}
\begin{proof}
	The first inequality is simple since, using $y \leq 1$,
	\begin{equation}
		\left| e^{- \frac{ \pi \alpha k^2}{y}} \cos\left(2\pi  k \left( \tfrac{1}{2} - \tfrac{1}{8y^2}\right)  \right) \right| \leq e^{-\pi \alpha k^2}.
	\end{equation}
	The second statement follows from an explicit computation and basic estimates.
\end{proof}

We can now, by explicit computation, derive an upper bound on the second derivative of the dominant terms which dominate the sum as $\alpha \rightarrow \infty$.
\begin{lemma} \label{lem:almostdone}
	We have, for $\alpha \geq 5$ and $\sqrt{3}/2 \leq y \leq 1$, that
	\begin{equation}
		\frac{d^2}{d y^2} \left( \sum_{-1 \leq k,l \leq 1} e^{- \pi \alpha \, \phi_{k,l}(y)} \, \psi_{k,l}(y) \right) 
        \leq -0.6 \, \pi^2 \alpha^2 e^{\pi \alpha (y-2)}.
	\end{equation}
\end{lemma}

As is perhaps not so surprising (the sum being comprised of the 9 terms in the periodicity cell, all of which can be differentiated twice in closed form), the entire argument boils down to a large number of computations that are perhaps not so interesting. We focus only on the most essential part of the argument which, in particular, shows how one can deduce the relevant length scales and estimates from the computations.

\begin{proof}  
	An explicit computation shows that
	\begin{equation}
		X =    \sum_{-1 \leq k,l \leq 1} \exp\left( - \pi \alpha \cdot \phi_{k,l}(y)\right)\psi_{k,l}(y)
	\end{equation}
	simplifies to
	\begin{align}
	X =  e^{-\pi  \alpha  y-\frac{9 \pi  \alpha }{4 y}} \left(2 \sin \left(\tfrac{3 \pi }{8 y^2}\right)-4 \, e^{\frac{2 \pi  \alpha }{y}} \sin \left(\tfrac{\pi }{8 y^2}\right)\right)- 2 e^{-\frac{\pi \alpha }{y}} \cos \left(\tfrac{\pi }{4 y^2}\right)+1.
	\end{align}
	We now differentiate each term twice in $y$. We first observe that 
	$$ - \frac{\pi \alpha}{y} \qquad \mbox{and} \qquad - \pi \alpha y - \frac{9 \pi \alpha}{4y} + \frac{2\pi \alpha}{y}$$
	coincide in $y = \sqrt{3}/2$, they both result in $-2\pi \alpha/\sqrt{3}$. In contrast, the remaining term is many exponential 	orders of magnitude smaller. It thus suffices to understand
	\begin{align}
	 X_2 &=  -4 \, e^{-\pi  \alpha  y-\frac{9 \pi  \alpha }{4 y}} e^{\frac{2 \pi  \alpha }{y}} \sin \left(\tfrac{\pi }{8 y^2}\right)-2 \, e^{-\frac{\pi  \alpha }{y}} \cos \left(\tfrac{\pi }{4 y^2}\right)\\
	 &= -4 \, e^{-\pi  \alpha  y-\frac{\pi \alpha }{4 y}} \sin \left(\tfrac{\pi }{8 y^2}\right)-2 \,  e^{-\frac{\pi  \alpha }{y}} \cos \left(\tfrac{\pi }{4y^2}\right).
	 \end{align}
	Differentiating twice and simplifying, we end up with
	$$ \frac{d^2}{dy^2} X_2 = \frac{\pi  e^{-\frac{\pi  \alpha  \left(y^2+1\right)}{y}}}{4 y^6} \left( Y_1 + Y_2 +Y_3 + Y_4\right),$$
	where, for $\alpha \geq 5$ and $\sqrt{3}/2 \leq y \leq 1$, all terms are negative:
	\begin{align*}
	Y_1 &= -2 \, y e^{\frac{3 \pi  \alpha }{4 y}} \left(\pi  \alpha  \left(4 y^2-1\right)+6 y\right) \cos \left(\tfrac{\pi }{8 y^2}\right) \leq 0 \\
	Y_2 &=-2 \, e^{\pi  \alpha  y} \left(4 \pi  \alpha ^2 y^2-8 \alpha  y^3-\pi \right) \cos \left(\tfrac{\pi }{4 y^2}\right)  \leq 0\\
	Y_3 &= -e^{\frac{3 \pi  \alpha }{4 y}} \left(\pi  \alpha ^2 \left(4 y^3-y\right)^2-8 \alpha  y^3-\pi \right) \sin \left(\tfrac{\pi }{8 y^2}\right)\leq 0 \\
	Y_4 &= 4 y \, e^{\pi  \alpha  y} (3 y-2 \pi  \alpha ) \sin \left(\tfrac{\pi }{4 y^2}\right) \leq 0.
	\end{align*}
	Therefore, in that regime, we have
	\begin{equation}
		\frac{d^2}{dy^2} X_2 \leq \tfrac{\pi  e^{-\frac{\pi  \alpha  \left(y^2+1\right)}{y}}}{4 y^6} Y_2
		\quad \text{ and } \quad
		Y_2 \leq -2.4 \, \pi \alpha^2   e^{\pi  \alpha  y}.
	\end{equation}
	Simultaneously, we have, in the same regime,
	$$ \tfrac{\pi  e^{-\frac{\pi  \alpha  \left(y^2+1\right)}{y}}}{4 y^6} \geq \tfrac{\pi}{4} \, e^{-2 \pi \alpha}.$$
	Combining these two inequalities, the result follows.
\end{proof}

\begin{lemma} \label{lem:prop2almost}
	We have, for $\sqrt{3}/2 \leq y \leq \sqrt{3}/2 + 1/(4\sqrt{\alpha})$ and $\alpha \geq 5$, that
	\begin{equation}
		\frac{d^2}{d y^2}  g_{\alpha}(y) < 0.
	\end{equation}
\end{lemma}
\begin{proof}
	It remains to derive an upper bound on the second derivatives of the remaining terms;
	\begin{equation}
		X = \frac{d^2}{d y^2}  \sum_{\substack{(k,l) \in \Z^2 \\ \max\left\{|k|, |l| \right\} > 1}} \hspace*{-9pt} e^{- \pi \alpha \, \phi_{k,l}(y)} \, \psi_{k,l}(y)
	\end{equation}
	and to show that they cannot compensate for the negative term derived for the leading terms in Lemma \ref{lem:almostdone}. We
	use Lemma \ref{lem:ell3} and Lemma \ref{lem:25} to argue that
	\begin{align}
		X &\leq \sum_{\substack{|k| > 1 \\ l =0}} \frac{d^2}{dy^2}  e^{- \pi \alpha \, \phi_{k,l}(y)} \, \psi_{k,l}(y)
		+  \hspace*{-18pt} \sum_{\substack{(k,l) \in \Z^2 \\ \max\left\{|k|, |l| \right\} > 1, l  \neq 0}} \hspace*{-18pt} \frac{d^2}{dy^2} e^{- \pi \alpha \, \phi_{k,l}(y)} \, \psi_{k,l}(y) \\
		&\leq 40 \sum_{k=2}^{\infty}  \alpha^2 k^4 e^{-\pi \alpha k^2}
		+ 3 \pi^2 \alpha^2 \hspace*{-18pt} \sum_{\substack{(k,l) \in \Z^2 \\ \max\left\{|k|, |l| \right\} > 1, l  \neq 0}} \hspace*{-18pt} Q_2(k,l)^2 \, e^{-\pi \alpha \sqrt{Q_2(k,l)}}.
	\end{align}
	We note that, again, this upper bound decays asymptotically like 
	$$ 40 \sum_{k=2}^{\infty}  \alpha^2 k^4  e^{-\pi \alpha k^2} \sim 40 \alpha^2 \cdot 16 \, e^{-4 \pi \alpha}$$
    which decays faster than our upper bound. Therefore, as before, it is clear that for all $\alpha \geq \alpha_0$, which by basic numerics is $\sim 1.1$, we will be able, using this argument, to deduce that
	$$X \leq 0.6 \, \pi^2 \alpha^2 \, e^{\pi \alpha (\frac{\sqrt{3}}{2}-2)}.$$
\end{proof}

\subsubsection{Proof of Proposition \ref{prop2}}
Lemma \ref{lem:a5} settles the case $\alpha \leq 5$. Hence, we only need to consider $\alpha \geq 5$. Lemma \ref{lem:twodecay} telling us that the maximum of $g_{\alpha}(y)$ for $\alpha \geq 5$ occurs at a point
\begin{equation}
	\frac{\sqrt{3}}{2} \leq y_{\max} \leq \frac{\sqrt{3}}{2} + \frac{1}{4\sqrt{\alpha}}.
\end{equation}
Lemma \ref{lem:prop2almost} implies that $f_{\alpha}$ is concave in that region, and Lemma \ref{lem:crit2} implies there is a critical point in $y = \sqrt{3}/2$, which therefore has to be the unique maximum.

\begin{appendices}
\section{Discussion of results}\label{sec_results}
Due to the universal nature of the Gaussian function and lattices, our result has several consequences. We now discuss some of the implications of Theorem~\ref{thm_main} in more detail.

\subsection{The Strohmer-Beaver Conjecture}\label{sec_Strohmer}
This conjecture from functional analysis, posed the latest in 2003 by Strohmer and Beaver \cite{StrBea03}, concerns a sampling problem using Gabor systems. Equivalently, it asks for extremal spectra of specific self-adjoint operators $S_\L$, acting on functions from $\Lt[]$ by the rule
\begin{equation}
	S_\L f = \sum_{\l \in \L} \langle f, \varphi_\l \rangle \, \varphi_\l,
	\quad \text{ where }\quad
	\varphi_\l(t) = e^{2 \pi i \l_2 t} e^{-\pi (t-\l_1)^2}, \; \l = (\l_1, \l_2).
\end{equation}
These operators arise from structured function systems $\mathcal{G}=\{\varphi_\lambda \mid \lambda \in \Lambda \subset \mathbb{R}^2\}$, called Gabor systems after the seminal work of Gabor \cite{Gab46}. They can be used to accurately reconstruct functions in $L^2(\mathbb{R})$, given that the sampling density is sufficient \cite{Lyu92}, \cite{Sei92}, \cite{SeiWal92}. This statement is quantified in the frame inequality (see \cite{Christensen_2016, Gro01, Hei07})
\begin{equation}\label{eq_frame_Gauss}
	A_\L \norm{f}_{L^2}^2 \leq \sum_{\l \in \L} |\langle f, \varphi_\lambda \rangle |^2 \leq B_\L \norm{f}_{L^2}^2, \quad \forall f \in \Lt[].
\end{equation}
Gabor systems are of importance in applications (e.g., wireless communication) \cite{FeiStr98, FeiStr03, Wunder16} or quantum mechanics \cite{FauGosRot18, Gos11}, where they had already been studied by John von Neumann~\cite{Neu32}. Denoting the spectral bounds of the Gabor frame operator $S_\L$ by $A_\L$ and $B_\L$, we have the following two-sided operator inequality
\begin{equation}
	A_\L \mathbf{I} \leq S_\L \leq B_\L \mathbf{I},
	\quad \text{ with } \quad
	\norm{S_\L}_{L^2 \to L^2} = B_\L, \quad \norm{S_\L^{-1}}_{L^2 \to L^2} = A_\L^{-1}.
\end{equation}
For $\vol(\L)^{-1} \in 2 \N$ ($\vol(\L)$ being the co-volume of $\L$), we have (for details see \cite{Faulhuber_Note_2018, Jan96}) that
\begin{align}
	A_\L & = \min_{z \in \R^2} \; \vol(\L)^{-1} \sum_{\l^\circ \in \L^\circ} e^{-\frac{\pi}{2} |\l^\circ|^2} e^{2 \pi i \sigma(\l^\circ,z)}, \\
	B_\L & = \max_{z \in \R^2} \; \vol(\L)^{-1} \sum_{\l^\circ \in \L^\circ} e^{-\frac{\pi}{2} |\l^\circ|^2} e^{2 \pi i \sigma(\l^\circ,z)},
\end{align}
where $\sigma(. \, , .)$ denotes the standard (skew-symmetric) symplectic form and $\L^\circ$ is the symplectic dual lattice (see Appendix~\ref{sec_tfa}). These are exactly the Gaussian lattice sums under consideration (the equivalence follows from the Poisson summation formula).

Montgomery's result \cite{Mon88} implies that $B_{\Lambda}$ is minimized when $\Lambda$ is the hexagonal lattice \cite{Faulhuber_Hexagonal_2018}. A conjecture in \cite{FloAlaBer95} suggested that the extremal lattice (of fixed density) minimizing the ratio $B_\L / A_\L$ is the square lattice. In 2003, Strohmer and Beaver disproved this conjecture and asked whether the extremal lattice may be given by the hexagonal lattice \cite{StrBea03}. Theorem~\ref{thm_main} implies that $A_{\Lambda}$ is maximized when $\Lambda$ is the hexagonal lattice. The minimality of $B_\lambda/A_\Lambda$ thus follows when combining Montgomery's results with ours.

Solving the Conjecture of Strohmer and Beaver requires several auxiliary steps. Computing sharp spectral bounds is extremely challenging and has been studied in \cite{Jan96}. An alternative approach is to use the Zak transform, (re)discovered by Zak \cite{Zak67} in mathematical physics. For Gaussians, this naturally leads to the study of theta functions. The Zak transform and its vector-valued version are also widely used in the study of Gabor systems \cite{Gro01, JanssenStrohmer_Secant_2002, LyubarskiiNes_Rational_2013}. We remark that the properties of Gabor systems, let alone the spectral bounds of the corresponding operators, in higher dimensions are in general not well understood (see, e.g., \cite{GroLyu19} for recent examples). There are further connections to quantum harmonic analysis introduced in \cite{Wer84} (see also \cite{Lue11}, \cite{LueMan09}, \cite{LueSkr21}, \cite{Skrettingland_2020}), for which our result has potential applications in estimating the norms of trace-class and Hilbert-Schmidt operators. A new approach for studying Gaussian Gabor systems (in arbitrary dimension) has recently been investigated by Luef and Wang \cite{LueWan23}. They combine Kähler geometry, Hörmander's $\overline{\partial}$ method \cite{Hor73} and symplectic embedding theorems \cite{DufPol94}. Their work then connects to the complex torus and its Seshadri constant, and the constants $A_\L$ and $B_\L$ of \eqref{eq_frame_Gauss} can be estimated by the Buser-Sarnak invariant \cite{BusSar94} of the symplectic dual lattice $\L^\circ$. Lastly, we note that Gabor systems were used to construct Cohn-Elkies functions for the sphere packing problem \cite{ManMar22}.

\subsection{Optimal lattice shapes for maximizing cold spots in heat conduction}\label{sec_torus}

We consider two-dimensional flat tori $\mathbb{R}^2/\Lambda$. It is well-known that the eigenvalues of the Laplace-Beltrami operator $\Delta_\L$ on $\T_\L$ (with appropriately chosen sign) are given by $\kappa_\l = 4 \pi^2 |\l^\perp|^2$, where $\l^\perp \in \L^\perp$ is an element of the dual lattice. The eigenfunctions are $e_\l(z) = e^{2 \pi i \l^\perp \cdot z}$, $z \in \T_\L$ and the dot $\cdot$ denotes the inner product. The heat kernel on $\T_\L$ is given by (see \cite{Gri_Heat_09})
\begin{equation}
	k_\L(z_1,z_2;t) = \sum_{\l \in \L} e^{-\kappa_\l t} e_\l(z_1) \overline{e_\l(z_2)}, \quad t > 0, \, z_1, z_2 \in \T_\L.
\end{equation}
After a simplification ($z_1-z_2 = z \in \T_\L$), this can explicitly be written as
\begin{equation}
	k_\L(z;t) = \sum_{\l^\perp \in \L^\perp} e^{-4 \pi^2 t |\l^\perp|^2} e^{2 \pi i \l^\perp \cdot z}.
\end{equation}
For each flat torus $\T_\L$ of unit area we consider the heat equation
\begin{equation}
	\begin{cases}
		\Delta_\L u(z;t) - \partial_t u(z;t) = 0, \quad t > 0\\
		u(z;0) = \delta_0
	\end{cases}
\end{equation}
and note that $k_\L$ is its fundamental solution. We define the lower and upper temperatures
\begin{equation}
	A_\L(t) = \min_{z \in \T_\L} k_\L(z;t)
	\quad \text{ and } \quad
	B_\L(t) = \max_{z \in \T_\L} k_\L(z;t).
\end{equation}
By the Poisson summation formula, we have
\begin{equation}
	k_\L(z;t) = \sum_{\l^\perp \in \L^\perp} e^{-4 \pi^2 t |\l^\perp|^2} e^{2 \pi i \l^\perp \cdot z} = \frac{1}{4 \pi t} \sum_{\l \in \L} e^{-\frac{1}{4 t} |l+z|^2} = \frac{1}{4\pi t} E_\Lambda \left(z;\frac{1}{4\pi t} \right).
\end{equation}
Theorem~\ref{thm_main} implies that among all tori $\T_\L$ and for all times $t>0$, the minimal temperature (cold spot, cf.\ \cite{BacMouValZim26}) of the heat distribution is maximized by the hexagonal torus~$\T_{\L_2}$.

\subsection{The polarization problem}\label{sec_cm}
The special role of the Gaussian allows for an immediate extension of our main result to all completely monotone interaction kernels (cf.\ \cite{Coh-Via19}). Let $p:\mathbb{R}_+ \rightarrow \mathbb{R}_+$ be completely monotone with sufficiently fast decay. Among all lattices $\Lambda$ of the same co-volume as $\Lambda_2$, we have that
\begin{equation}
    \min_{z \in \mathbb{R}^2} E_{p,\Lambda}(z) = \min_{z \in \mathbb{R}^2} \sum_{\l \in \L} p(|\l - z |^2) \qquad \mbox{is maximized}
\end{equation}
if and only if $\Lambda$ is the hexagonal lattice. In particular, we solve the 2-dimensional polarization problem formulated in \cite[Sec.~8.3]{HarPetSaf22} (see also \cite{BacMouValZim26}, \cite[Chap.~14]{BorHarSaf19}). For this purpose, we recall the Bernstein-Widder theorem, which characterizes completely monotone functions \cite{bernstein, Wid41}.
\begin{theorem*}[Bernstein-Widder]\label{thm:BernsteinWidder}
	Let $p: \R_+ \to \R_+$ and $(-1)^n p^{(n)}(r) \geq 0$, $n \in \N$, i.e., $p$ is completely monotone. Then, $p$ is the Laplace transform of a non-negative Borel measure $\mu_p$
	\begin{equation}\label{eq_Bernstein_Widder}
		p(r) = \int_0^\infty e^{-r \alpha} \, d \mu_p(\alpha).
	\end{equation}
\end{theorem*}
The decay condition ensures that the terms appearing in $E_{p,\L}$ are absolutely summable and that we may interchange the order of summation and integration, reducing the problem to Gaussian potentials. If $p$ is not decaying sufficiently fast, regularization methods, such as the Ewald summation, need to be applied (see, e.g., \cite{BetKnu_Born_18}).

For $\Rd$ and Riesz potentials, the problem was posed by Saff at the AIM workshop \textsf{Discrete geometry and automorphic forms} (\url{http://aimpl.org/discreteaf/1/}). The physically important class of Riesz potentials is, for example, obtained by
\begin{equation}
	\frac{1}{r^s} = \int_0^\infty e^{- \alpha r^2} \frac{\alpha^{\frac{s}{2}-1}}{\Gamma(\frac{s}{2})} \, d \alpha,
\end{equation}
where $\Gamma$ is the usual Gamma-function (compare \cite[Sec.~1.2]{Coh-Via19}). For the Riesz potential, the lattice sums result in the shifted Epstein zeta functions, resulting from the Mellin transform of the lattice theta function (see, e.g., \cite{Jorgenson_Heat_2001} or \cite{Mon88}).

The corresponding problem on the sphere is known as the polarization problem (see \cite{Bor22}). The first polarization result for periodic configurations on the real line was obtained in \cite{FauSte24}.

\subsection{Born's problem for optimal charges on a lattice}\label{sec_Born}
Another related problem in mathematical physics is the so-called ``Born Conjecture" dating back to the work of Born \cite{Born-1921} from 1921: he asks about the \textit{minimal charge configuration} on a given lattice $\Lambda$ for completely monotone interactions. The goal is to find the $N$-periodic (in all directions) distribution of charges $\varepsilon=\{\varepsilon_x\}_{x\in \Lambda}$ where $\L \subset \Rd$, satisfying certain conditions on finite subsets of size $N \times \ldots \times N$ (see Figure \ref{fig_charges}). We call these subsets the periodicity cell $K_N$. We define
\begin{equation}
	\norm{\varepsilon}_{K_N}^2 = \sum_{y \in K_N} |\varepsilon_y|^2.
\end{equation}
The constraints on $\varepsilon$ are
\begin{equation}
	\norm{\varepsilon}_{K_N}^2 = N^d
	\quad \text{ and } \quad
	\sum_{y \in K_N} \varepsilon_y = 0 \; \text{ (neutrality assumption)}.
\end{equation}
The task is, for a given interaction potential $p$, to minimize the charge energy per point:
\begin{equation}\label{eq:Born}
	C_{p,\Lambda}(\varepsilon) = \frac{1}{N^d} \sum_{x\in \Lambda} \sum_{y\in K_N} \varepsilon_x \varepsilon_y \, p(|x-y|^2).
\end{equation}
\begin{figure}[ht]
    \begin{framed}
    	\hfill
    	\includegraphics[width=.425\textwidth]{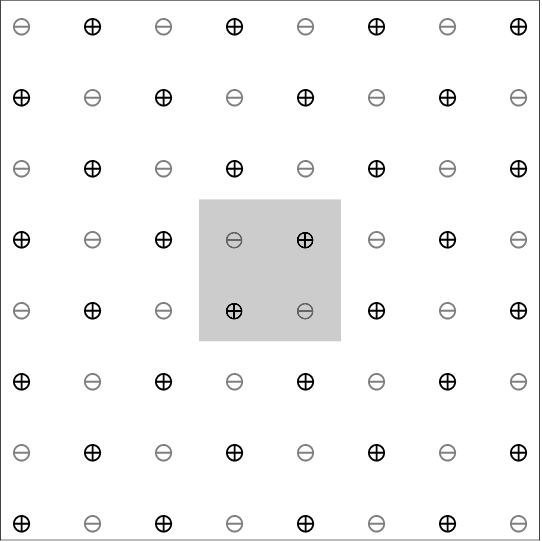}
    	\hfill
        \hfill
    	\includegraphics[width=.425\textwidth]{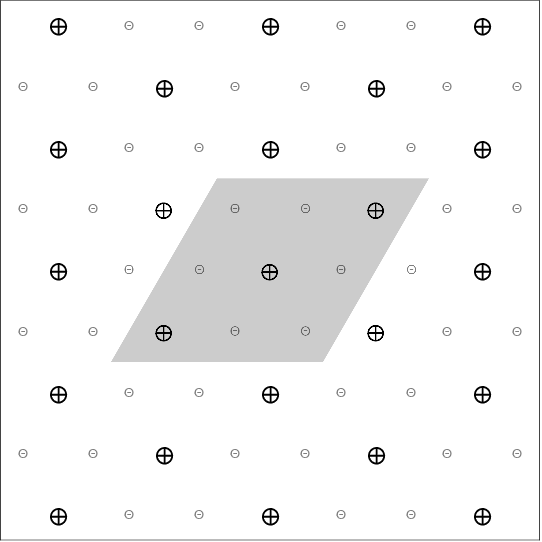}
    	\hspace*{\fill}
    	\caption{Minimal charged ionic crystals for the square lattice (left) and the hexagonal lattice (right). The periodicity cell contains 4, respectively, 9 charged ions. In the square lattice, we find alternating charges (Born's conjecture). In the hexagonal lattice, the negative charges, each with half the weight of a positive charge, form a honeycomb structure.}\label{fig_charges}
    \end{framed}
\end{figure}
\vspace{-\baselineskip}

For ionic crystals and their lattice energies, given by \eqref{eq:Born}, we can use the result from \cite{BetKnu_Born_18} and a combination of \nameref{thm:Baernstein}'s Theorem, the \nameref{thm:BernsteinWidder} Theorem, and our \nameref{thm_main} to obtain the following result. Among all completely monotone potentials $p$ and all lattices $\Lambda\subset \R^2$ of a fixed density, the minimal ionic crystal energy
\begin{equation}\label{eq:corBorn}
	\min_{\varepsilon} C_{p,\Lambda}(\varepsilon) \qquad \text{ is maximized}
\end{equation}
if and only if $\Lambda$ is the hexagonal lattice $\Lambda_2$. The minimum is taken over all $N\geq 1$ and all $N$-periodic charge distributions satisfying $\|\varepsilon\|^2_{K_N}=N^d$ and the neutrality assumption.

Lattice sums of type \eqref{eq:Born} appear in the calculation of numerous quantities in physics and chemistry \cite{BorBorTay85}, including the Ewald constant \cite{Ewald1} corresponding to the alternation of charges $\pm 1$ on the cubic lattice $\Z^3$. Further related results have been discussed in \cite{BDefects20,BFK20}.

\subsection{A problem of Landau}\label{sec_Landau}
A problem about discrete extremal geometries in function theory that caught our attention is based on the following theorem of Landau \cite{Lan29}.
\begin{theorem*}[Landau, 1929]
	Let $f \colon \mathbb{D} \to \C$ be holomorphic and $|f'(0)| = 1$. Then, there exists an absolute constant $\mathcal{L} > 0$ such that a disc $D_\mathcal{L}$ of radius $\mathcal{L}$ is contained in the image $f(\D)$.
\end{theorem*}
The open problem is to find the precise value of the Landau constant $\mathcal{L}$. By $\ell(f)$ we denote the radius of the largest disc in $f(\D)$:
\begin{equation}\label{eq_covering_l}
	\ell(f) = \sup \{ r \in \R_+ \mid D_r \subset f(\D), f \text{ as above} \}
\end{equation}
Then, the Landau constant is
\begin{equation}
	\mathcal{L} = \inf \{ \ell(f) \mid f \text{ as above} \}.
\end{equation}
We remark that there is a univalent version of the result and the problem, due to Bloch~\cite{Blo25}. For the Landau constant, it is known that
\begin{equation}
	\frac{1}{2} < \mathcal{L} \leq \mathcal{L}_{\L_2} = \frac{\Gamma\left(\frac13\right) \Gamma \left(\frac56\right)}{\Gamma \left(\frac16\right)} \approx 0.543259 \dots,
\end{equation}
where the (non-strict) lower bound is due to Ahlfors \cite{Ahl38}. This lower bound was seemingly shown to be strict by Pommerenke \cite{Pom70}, but the article contained a flaw discovered by Yanagihara \cite{Yam88}. The lower bound has then further been improved to $\frac{1}{2}+10^{-335}$ \cite{Yan95} and $\frac{1}{2} + 10^{-8}$ \cite{CheShi04}. The upper bound is due to a 1943 construction of Rademacher \cite{Rad43}. The same upper bound was previously established by Robinson (1938, unpublished). It is widely assumed that Rademacher's upper bound is sharp and, hence, gives the true value of Landau's constant.

Baernstein suggested \cite{Baernstein_ExtremalProblems_1994, Baernstein_HeatKernel_1997, BaernsteinVinson_Local_1998} studying heat kernels on flat tori $\T_\L = \R^2 / \L$ to understand Landau's problem.
For $\alpha = 1$, denote the minimal value of $E_\Lambda(z;1)$ by
\begin{equation}
	\mathcal{A}_\Lambda = \min_{z \in \R^2} E_{\L}(z;1).
\end{equation}
We know by the results in \cite{Faulhuber_PhD_2016, Faulhuber_Rama_2019} (see \cite{Baernstein_Metric_2005, Eremenko_Hyperbolic_2011} for the restricted problem) that
\begin{equation}
	\mathcal{L}_{\L_2} \mathcal{A}_{\L_2} = \frac{1}{2}
    \quad 
    \text{ and }
    \quad
    \mathcal{L}_{\mathbb{Z}^2} \mathcal{A}_{\mathbb{Z}^2} = \frac{1}{2},
\end{equation}
where $\mathcal{L}_\Lambda$ denotes the constant obtained from using the universal covering map (see Figure~\ref{fig_Landau})
\begin{equation}
    f_\Lambda \colon \mathbb{D} \to  \mathbb{C} \setminus \Lambda
\end{equation}
in Landau's theorem. Unfortunately, we are not able to provide deeper connections.
\begin{figure}[ht]
    \begin{framed}
	\hspace*{\fill}
	\includegraphics[width=.375\textwidth]{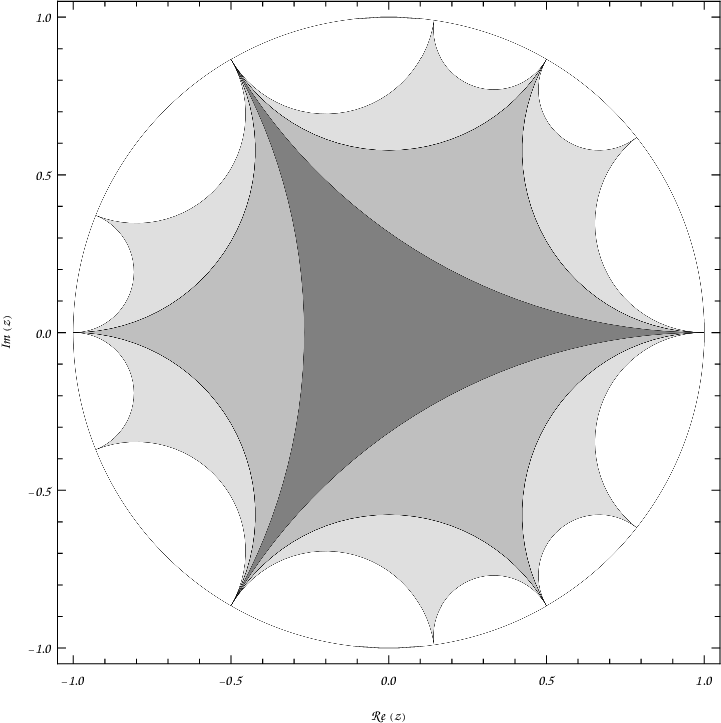}
	\hspace*{\fill} \raisebox{3.25cm}{$\; \underset{f_{\L_2}}{\longrightarrow}$} \hspace*{\fill}
	\includegraphics[width=.375\textwidth]{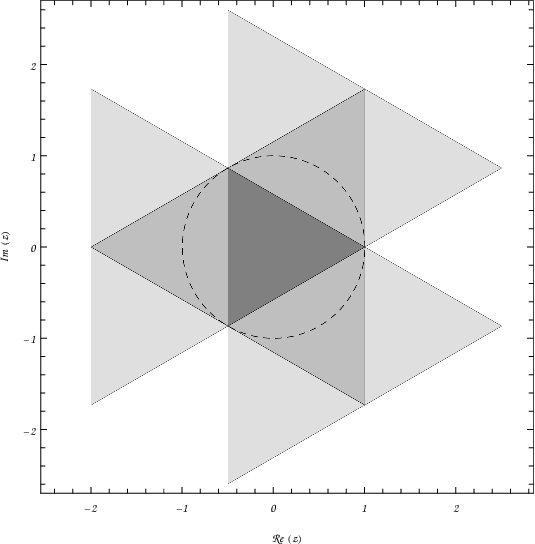}
	\hspace*{\fill}
	\caption{The map $f_{\L_2}$ sends the triangular tessellation of the hyperbolic disc to the triangular tessellation of the complex plane. The (shifted) lattice points become branching points of infinite order and are, hence, not in the image $f_{\L_2}(\D)$. The dashed circle line is the boundary of the largest (open) disc that can be placed in the image $f_{\L_2}(\D)$.}\label{fig_Landau}
    \end{framed}
\end{figure}

\section{The space of lattices}\label{sec_lattice}
Indexing 2-dimensional lattices is classically done by switching to complex numbers with a positive imaginary part. As we scale our lattices by $\alpha > 0$, we will focus on indexing lattices of co-volume 1. A lattice $\L$ in $\R^2$ is a discrete co-compact subgroup of $\R^2$ with co-volume $\vol(\R^2/\L)$. It can be generated by a non-unique matrix $M \in GL(2,\R)$;
\begin{equation}
	\L = M \Z^2 .
\end{equation}
So, $\L$ is the linear integer span of the columns of $M$. Assuming $\vol(\L) = 1$ restricts our attention to matrices with $\det(M) = \pm 1$. By a relabeling of the columns, we thus may assume that $\det(M) = 1$. As the problem under consideration is invariant under rotation, we may assume that $M$ is of the form
\begin{equation}\label{eq_lattice_M}
	M = y^{-1/2} Q
	\begin{pmatrix}
		1 & x\\
		0 & y
	\end{pmatrix},
	\quad Q \in SO(2,\R).
\end{equation}
The geometry of the lattice thus only depends on the 2 parameters $(x,y)$, with the condition $y > 0$ (as we normalize by $y^{-1/2}$). By the natural identification of a complex number $z = x + i y$ with vectors
$\begin{pmatrix}
	x\\y
\end{pmatrix} \in \R^2$, we can identify a lattice $\L$ with a complex number in the upper half plane $\mathbb{H} = \{ z \in \C \mid \Im(z) > 0\}$.

\medskip

As already mentioned, the matrix generating a lattice is not uniquely defined. This is due to the fact that for any $\mathsf{B} \in SL(2,\Z)$ we have $\mathsf{B} \Z^2 = \Z^2$. In particular, we can choose between different bases of $\Z^2$ and, hence, between different bases for any lattice $\L$. Furthermore, we note that $\L$ is an additive group, hence $\L = - \L$, which means that we only need to consider bases from $PSL(2,\Z) = SL(2,\Z) \slash \{\pm I\}$, the modular group. As explained in detail in \cite{Serre_Course_1973} (note that we consider column vectors instead of row vectors), the modular group is generated by the matrices
\begin{equation}
	J = \begin{pmatrix}
		0 & 1\\
		-1 & 0
	\end{pmatrix}
	\quad \text{ and } \quad
	T = \begin{pmatrix}
		1 & 1\\
		0 & 1
	\end{pmatrix}.
\end{equation}
The group $SL(2,\R)$ and in particular its subgroup $PSL(2,\Z)$ act on $\mathbb{H}$ by fractional linear transformations. That is, for a matrix
$
S = \begin{pmatrix}
	a & b\\
	c & d
\end{pmatrix} \in SL(2,\R)
$ we have the (right) action
\begin{equation}
	\tau \circ S = \frac{a \tau + b}{c \tau + d}.
\end{equation}
It is then sufficient to focus on lattices generated by
\begin{equation}
	\tau \in D = \{ z \in \H \mid |z| \geq 1 \text{ and } |\Re(z)| \leq \tfrac{1}{2} \}.
\end{equation}
The fundamental domain $D$ contains so to say all the canonical bases of 2-dimensional lattices, i.e., the basis is a Minkowski basis in this case. By symmetry reasons, it is enough for us to work only in the right (or equivalently in the left) half of $D$, which we denote by $D_+$ (or $D_-$).
\begin{proposition}\label{pro:change_B}
	Let $M \in SL(2,\R)$ with parameters $(x,y)$ coming from $\tau \in D$ and let $\mathsf{B} \in PSL(2, \Z)$ (the modular group). Using the base $M \mathsf{B}$ leaves the lattice invariant and the point $b_\L$ has coordinate vector $\mathsf{B}^{-1} b$. Moreover, the values of the functions $\theta_\L(b_\L;\alpha)$ and $\widehat{\theta}_\L(b_\L;\alpha)$ are invariant under the choice of basis by adjusting the coordinate vector accordingly.
\end{proposition}
\begin{proof}
	First, we note that extending the lattice parameters from $D_+$ to $D$ only needs a simple flipping argument. We replace $b_1$ by $-b_1$ if we replace $(x,y)$ by $(-x,y)$ to achieve that $\theta_\L$ and $\widehat{\theta}_\L$ stay invariant. So, we may assume that $M$ has parameters from $D$. We now provide the argument for $\theta_\L$. For this purpose, we set $v=(k,l) \in \Z^2$ and note that $\mathsf{B} \Z^2 = \Z^2$ for $\mathsf{B} \in PSL(2,\Z)$. Trivially, we have $b_\L = (M \mathsf{B}) (\mathsf{B}^{-1} b)$. Now, we compute
	\begin{align}
		\theta_\L(b_\L;\alpha) & = \sum_{v \in \Z^2} e^{-\pi \alpha |M(v+b)|^2} = \sum_{v \in \mathsf{B} \Z^2} e^{-\pi \alpha |M(v+b)|^2}\\
		& = \sum_{v \in \Z^2} e^{-\pi \alpha |M(\mathsf{B}v+b)|^2} = \sum_{v \in \Z^2} e^{-\pi \alpha |M\mathsf{B}(v+\mathsf{B}^{-1}b)|^2}.
	\end{align}
	The argument for $\widehat{\theta}_\L$ is similarly simple by direct manipulation of the terms. Alternatively, one can also apply the symplectic Poisson summation formula on $\theta_\L$ to draw the conclusion.
\end{proof}
To switch from $D_+$ to $D_-$ we use the map $\tau \mapsto -\overline{\tau}$. So we flip the lattice $\L$ to the lattice $\overline{\L}$:
\begin{equation}
	\overline{\L} = y^{-1/2} \begin{pmatrix}
		1 & -x\\
		0 & y
	\end{pmatrix} \Z^2.
\end{equation}
In this case, we need to replace $b_1$ by $-b_1$ in formula \eqref{eq_b} in order to evaluate in the correct point in \eqref{eq_theta_b} and \eqref{eq_theta_hat_b}. If we have a lattice coming from $\widetilde{\tau} \notin D$, then there exists a matrix $\mathsf{B} \in PSL(2,\Z)$ such that $\tau \circ \mathsf{B} = \widetilde{\tau}$. In this case, we performed a change of basis which we need to pass on to $b$ as in Proposition \ref{pro:change_B}. Say $M$ is coming from $\tau \in D$ and $\widetilde{\L} = M \mathsf{B} \Z^2 = M \Z^2 = \L$. Note that the change of base gets lost in $\Z^2$, however, in our explicit formulas we may arrive at an integrally equivalent quadratic from $q(k,l)$ (see \cite[Chap.~2.2]{ConSlo98}) in the exponent as the matrix $\mathsf{B} \in PSL(2,\Z)$ may explicitly appear;
\begin{equation}
	q(k,l) = (k,l) \mathsf{B}^T M^T M \mathsf{B}
	\begin{pmatrix}
		k\\l		
	\end{pmatrix}.
\end{equation}
Not adapting $b$ in the explicit formulas to the chosen base may then actually have the result that the functional $E_\L$ evaluates to a larger value than the minimum in the hexagonal case. We have to choose $\widetilde{b} = \mathsf{B}^{-1} b$ to construct the desired point in $\R^2$ from the columns of $M \mathsf{B}$. Note again that the parameters coming from $\tau$ and $\tau \circ \mathsf{B}$ always yield the same lattice. However, we need to use the explicit coordinates $\mathsf{B}^{-1} b$ to describe the desired point for Theorem \ref{thm_theta} and the \nameref{thm_main} to hold.

\begin{figure}[h!t]
    \begin{framed}
    	\includegraphics[width=.85\textwidth]{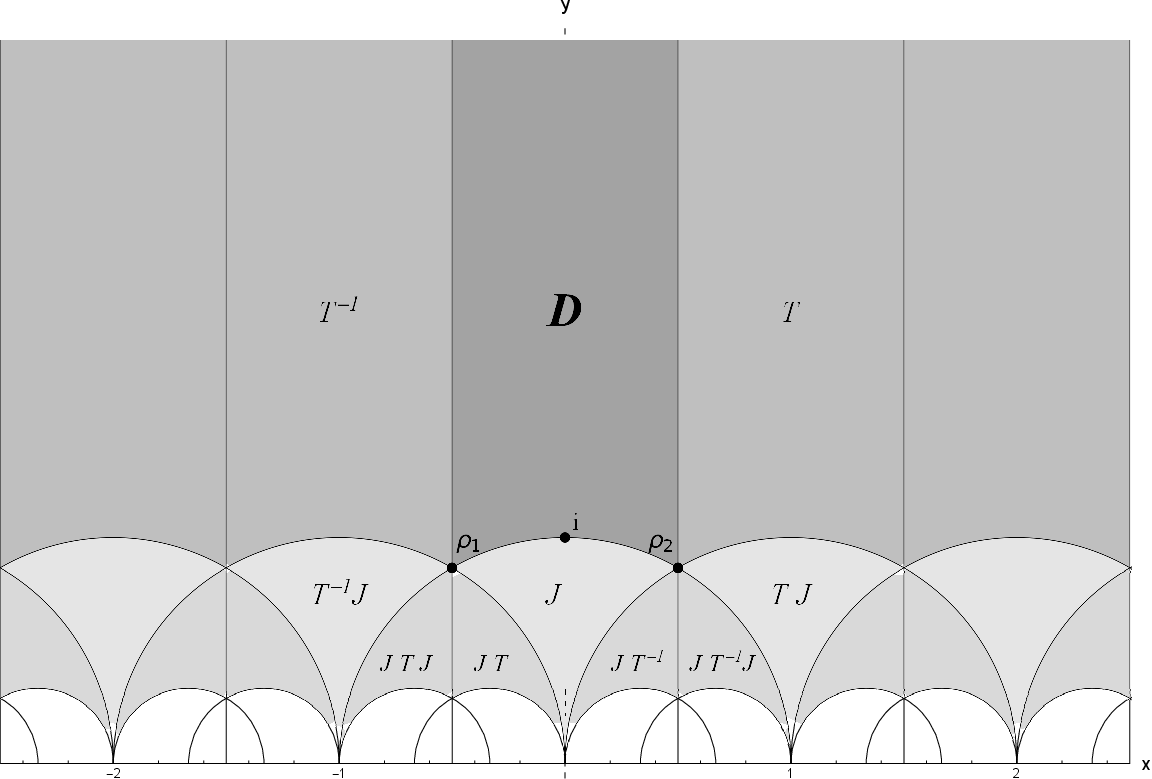}
    	\caption{Tessellation of $\H$ into hyperbolic triangles. Any point in $D$ can be mapped to a unique point in any part of the tessellation by applying the respective rules. For example, to map $\tau \in D$ to the triangle $TJ$, we apply the respective linear fractional transformation on $\tau$. For a lattice, this results in a different choice of basis. The points $\rho_1 \in D_-$ and $\rho_2 \in D_+$ yield the hexagonal lattice, whereas $i$ yields the square lattice.}
    \end{framed}
\end{figure}

\section{Phase space methods and symplectic lattices}\label{sec_tfa}
As we work in dimension 2, we can actually use phase space methods from quantum mechanics or time-frequency analysis. We provide some elementary notation from both fields. For our purposes, this is more or less for notational and technical convenience, nonetheless, these methods are of importance and we will explain the basics in the sequel. For functions on $\Rd$, the common translation and modulation operators are given by
\begin{equation}
	T_x f(t) = f(t-x)
	\quad \text{ and } \quad
	M_\omega f(t) = e^{2 \pi i \omega \cdot t} f(t),
\end{equation}
respectively. The dot $\cdot$ denotes the Euclidean inner product on $\Rd$ \footnote{We see $\Rd$ as a space of column vectors, hence $\omega \cdot t = \omega^T t$}. The operators are of course unitary on $\Lt$ and they commute up to a phase factor;
\begin{equation}\label{eq_comm_rel}
	M_\omega T_x = e^{2 \pi i \omega \cdot x} T_x M_\omega
\end{equation}
The Fourier transform $\F$ is unitary on $\Lt$ as well, by using the following normalization
\begin{equation}
	\F f (\omega) = \widehat{f}(\omega) = \int_{\Rd} f(t) e^{- 2 \pi i \omega \cdot t} \, dt.
\end{equation}
In his pioneering article \cite{Gab46}, Gabor suggested to use a mixture of the time-representation $f(t)$ and the spectral representation $\widehat{f}(\omega)$ in order to overcome the drawbacks of the separated representations. This is achieved by using a generalization of the Fourier transform, namely the short-time Fourier transform (STFT)\footnote{The notation $V_g$ comes from engineering, where the transform is often called voice transform. If $g$ is a Gaussian, it is often called Gabor transform as in \cite{Gab46} the Gaussian window was considered.};
\begin{equation}
	V_g f(x,\omega) = \int_{\Rd} f(t) \overline{g(t-x)} e^{-2 \pi i \omega \cdot t} \, dt = \langle f, M_\omega T_x g \rangle.
\end{equation}
The above formula makes sense for $f,g \in \Lt$, but also whenever the bracket $\langle . \, , . \rangle$ is defined in a distributional sense. For $\norm{g}_2 = 1$, $V_g$ is an isometry from $\Lt$ to $\Lt[2d]$. By using the facts that
\begin{equation}\label{eq_FT_tf-shift}
	\F T_x = M_{-x} \F
	\quad \text{ and } \quad
	\F M_\omega = T_\omega \F
\end{equation}
in combination with Parseval's formula for $\Lt$ (or the fact that $\F$ is unitary), we see that
\begin{align}
	V_g f(x,\omega) & = \langle f , M_\omega T_x g \rangle = \langle \F f, \F(M_\omega T_x g) \rangle = e^{-2 \pi i \omega \cdot x} \langle\widehat{f}, M_{-x} T_\omega \widehat{g} \rangle = e^{-2 \pi i \omega \cdot x} V_{\widehat{g}} \widehat{f} (\omega, -x).
\end{align}
So, the STFT gives a local average of the function $f$ and its spectrum $\widehat{f}$ at the point $(x,\omega) \in \R^{2d}$. In this context, $\R^{2d}$ is called the time-frequency plane or phase space. The joint operation of translation and modulation is called a time-frequency shift and is denoted by
\begin{equation}
	\pi(\gamma) = M_\omega T_x, \quad \gamma = (x,\omega)\in \R^{2d}.
\end{equation}
There is also a symmetric version of the time-frequency shifts, which is
\begin{equation}
	\rho(\gamma) = M_{\frac{\omega}{2}} T_x M_{\frac{\omega}{2}} = T_{\frac{x}{2}} M_\omega T_{\frac{x}{2}}.
\end{equation}
In this way we may also define the (cross-)ambiguity function as it appears in radar technology;
\begin{equation}
	A(f,g)(\gamma) = \langle \pi(- \tfrac{\gamma}{2}) f, \pi(\tfrac{\gamma}{2}) g \rangle = \langle f, \rho(\gamma) g \rangle .
\end{equation}
It is well-known that neither $\pi(\gamma)$ nor $\rho(\gamma)$ are closed under composition. This stems from the commutation relations \eqref{eq_comm_rel} of $T_x$ and $M_\omega$. A group law can be imposed by adding an auxiliary variable and viewing the operators as unitary representations of the (polarized) Heisenberg group \cite{Fol89}, \cite{Gro01}. We will only focus on $\rho(\gamma)$ for the moment, as the group law of the underlying Heisenberg group underlines the importance of the symplectic structure of phase space in this case. We compute
\begin{equation}
	\rho(\gamma) \rho(\gamma') = e^{-\pi i \sigma(\gamma, \gamma')} \rho(\gamma+\gamma'),
\end{equation}
where $\sigma(. \, , .)$ is the skew-symmetric form
\begin{equation}
	\sigma(\gamma, \gamma') = \gamma \cdot J \gamma' = x \cdot \omega' - x' \cdot \omega,
	\qquad
	J = \begin{pmatrix}
		0 & I \\
		-I & 0
	\end{pmatrix}.
\end{equation}
also called standard symplectic form. The matrix is called the standard symplectic matrix. The sign convention differs in the literature and sometimes $-J = J^T = J^{-1}$ is said to be the standard symplectic matrix. Adding an auxiliary unitary operator (a phase factor) gives
\begin{equation}
	e^{2 \pi i \tau} \rho(\gamma) e^{2 \pi i \tau'} \rho(\gamma') = e^{2 \pi i \left((\tau+\tau') - \frac{1}{2} \sigma(\gamma,\gamma')\right)} \rho(\gamma+\gamma').
\end{equation}
These operators are now closed under composition and are in fact unitary representations on $\Lt$ of the underlying Heisenberg group. The Heisenberg group $(\mathbf{H}, \circ)$ is a non-Abelian group with $\mathbf{H} = \R^{2d} \times \R$ and the composition law given by
\begin{equation}
	(x,\omega,\tau) \circ (x',\omega',\tau') = (x+x',\omega+\omega',\tau+\tau'+\tfrac{1}{2}(x' \cdot \omega - x \cdot \omega')).
\end{equation}
As a topological object, $\mathbf{H}$ is identical with $\R^{2d} \times \R$, but their algebraic structures are different. A different representation of the Heisenberg group, namely the theta representation, was popularized by Mumford \cite{Mum_Tata_I}. However, they are unitarily equivalent, as is any other representation of the Heisenberg group by the Stone -- von Neumann Theorem. The Jacobi theta functions are then invariant under the action of a (certain) discrete subgroup of $\mathbf{H}$.

\medskip

A matrix $S \in GL(2d,\R)$ is called symplectic if and only if it preserves the symplectic form;
\begin{equation}\label{eq_symp}
	\sigma(S \gamma, S \gamma') = \sigma( \gamma, \gamma')
	\qquad \Longleftrightarrow \qquad
	S^T J S = J.
\end{equation}
Symplectic matrices actually form a group under matrix multiplication, denoted by $Sp(d)$. The notation is again not coherent in the literature and one might as well find $Sp(d,\R)$, $Sp(2d)$ or $Sp(2d,\R)$. It can be shown that symplectic matrices actually have determinant 1 (from the above equation it already follows that the determinant must be $\pm1$). Also, the identity $S^T J S = J$ puts $(2d-1)d$ constraints on $S$, which is the dimension of the space of skew-symmetric matrices. This leaves $(2d+1)d$ of the $(2d)^2$ variable free. In general, $Sp(d)$ is a proper subgroup of $SL(2d,\R)$. Only if $d = 1$ we have $Sp(1) = SL(2,\R)$, because then the only constraint is that $\det(S) = 1$. Now, a lattice $\L \subset \R^{2d}$ is called symplectic if and only if
\begin{equation}
	\L = \alpha S \Z^2, \quad S \in Sp(d), \; \alpha > 0,
\end{equation}
where $\Z^2$ is equipped with the canonical (or a symplectic) basis. We note that $\Z^2$ itself is symplectic with the canonical basis as well as, e.g., $J \Z^2 = \Z^2$. However, e.g., the matrix
\begin{equation}
	P_{2,3} =
	\begin{pmatrix}
		1 & 0 & 0 & 0\\
		0 & 0 & 1 & 0\\
		0 &-1 & 0 & 0\\
		0 & 0 & 0 & 1
	\end{pmatrix},
\end{equation}
has determinant 1 and as lattices $P_{2,3} \Z^2 = \Z^2$, but $P_{2,3}$ is not symplectic. Therefore, also $SP_{2,3}$ is not symplectic for $S \in Sp(2)$.

Usually, the Poisson summation formula involves the dual lattice $\L^\perp$ of the lattice $\L$. The usual characterization is the following
\begin{equation}
	\L^\perp = \{ \l^\perp \in \R^d \mid \l \cdot \l^\perp \in \Z, \; \forall \l \in \L \}.
\end{equation}
An alternative definition is
\begin{equation}
	\L^\perp = \{ \l^\perp \in \Rd \mid e^{2 \pi i \l \cdot \l^\perp} = 1, \; \forall \l \in \L \}.
\end{equation}
We note that $\L^\perp$ is indeed again a lattice and we have
\begin{equation}
	\L = M \Z^d
	\qquad \Longleftrightarrow \qquad
	\L^\perp = M^{-T} \Z^2 .
\end{equation}
In time-frequency analysis, the dual lattice is usually replaced by the adjoint lattice $\L^\circ$, which could also be called the symplectic dual lattice. This name has already been suggested in \cite{LueWan23} It is defined in a similar manner, but by means of commuting time-frequency shifts;
\begin{equation}
	\L^\circ = \{ \l^\circ \in \R^{2d} \mid \pi(\l)\pi(\l^\circ) = \pi(\l^\circ) \pi(\l), \; \forall \l \in \L\}.
\end{equation}
We note that, unlike the dual lattice, the adjoint lattice is only characterized in even dimensions. Computing the commutator of two time-frequency shifts yields
\begin{equation}
	[\pi(\l),\pi(\l')] = \pi(\l) \pi(\l')-\pi(\l')\pi(\l) = \left(1 - e^{2 \pi i \sigma(\l,\l')}\right) \pi(\l)\pi(\l').
\end{equation}
Again, the symplectic form appears (for a different reason though) and can hence be used for the characterization of the symplectic dual or adjoint lattice;
\begin{align}
	\L^\circ = \{ \l^\circ \in \R^{2d} \mid \sigma(\l^\circ, \l) \in \Z, \; \forall \l \in \L \} = \{ \l^\circ \in \R^{2d} \mid e^{2 \pi i \sigma(\l^\circ,\l)} = 1, \; \forall \l \in \L \}.
\end{align}
So, the symmetric Euclidean inner product $\cdot$ is replaced by the skew-symmetric form $\sigma$. In terms of the defining matrix, we have
\begin{equation}
	\L^\circ = J S^{-T} \Z^{2d} = J S^{-T} J^T \Z^{2d}.
\end{equation}
We note that the second part of the equality is usually not found in the literature, but that it is important. Putting a minus on both sides of \eqref{eq_symp}, noting that $-J = J^T = J^{-1}$, we get
\begin{equation}
	S^T J^T S = J^T
	\qquad \Longleftrightarrow \qquad
	S = J S^{-T} J^T.
\end{equation}
Hence, if $\L$ is a symplectic lattice of the form $\L = \alpha S \Z^{2d}$, then the adjoint lattice $\L^\circ$ is only a scaled version of $\L$, namely $\L^\circ = \alpha^{-2} \L$. Note that $\vol(\L) = \alpha^{2d}$ and, hence,
\begin{equation}
	\L^\circ = \vol(\L)^{-1/d} \L.
\end{equation}

The simple idea of replacing the Euclidean inner product $\cdot$ by $\sigma$ also leads to a new version of the Fourier transform, namely the symplectic Fourier transform. For a function $F$ of $2d$ variables, the symplectic Fourier transform is defined as
\begin{equation}
	\F_\sigma F(z) = \int_{\R^{2d}} F(z') e^{-2 \pi i \sigma(z',z)} \, dz'
\end{equation}
It inherits its properties from the usual Fourier transform as
\begin{equation}
	\F_\sigma F (z) = \F F(Jz).
\end{equation}
In addition it is also involutive, i.e., $\F_\sigma \circ \F_\sigma = \mathbf{I}$, where $\mathbf{I}$ is the identity operator. This follows from the fact that $J^2 = - I$ and that $\F(\F F)(z) = F(-z)$. One advantage of the symplectic Fourier transform in phase space is that it has more eigenfunctions than the planar Fourier transform. To see this, we first need to introduce the (cross-) Wigner distribution of two functions $f,g \in \Lt$. It is given by
\begin{equation}
	W(f,g)(x,\omega) = \int_{\Rd} f(x+\tfrac{t}{2}) \overline{g(x-\tfrac{t}{2})} e^{2 \pi i \omega \cdot t} \, dt.
\end{equation}
It almost looks like the ambiguity function and in fact we have the following algebraic relation
\begin{equation}
	W(f,g)(x,\omega) = 2^d A(f,g^\vee)(2x,2\omega),
\end{equation}
where $g^\vee(t) = g(-t)$ is the reflection of $g$. Furthermore, we have
\begin{equation}
	\F_\sigma(W(f,g))(z) = A(f,g)(z)
\end{equation}
and vice versa. From these relations, it follows (see also \cite{Faulhuber_Note_2018}) that
\begin{align}
	\F_\sigma (D_{\sqrt{2}} \, A(f,g))(z) = \pm D_{\sqrt{2}} \, A(f,g)(z)
	\; \text{ and } \;
	\F_\sigma (D_{1/\sqrt{2}} \, W(f,g))(z) = \pm D_{1/\sqrt{2}} \, W(f,g)(z),
\end{align}
for $g^\vee = \pm g$. The operator $D_\beta$ is the (non-unitary) isotropic dilation operator given by
\begin{equation}
	D_\beta F(z) = F(\beta z), \quad \beta > 0.
\end{equation}
So, after a proper scaling the ambiguity functions (or Wigner distributions) of any $f \in \Lt$ with an even or odd $g \in \Lt$ are eigenfunctions of the symplectic Fourier transform with eigenvalue $\pm 1$, depending on the parity of $g$. Also, any $2d$-dimensional Gaussian of the form
\begin{equation}
	\Phi(z) = e^{-\pi |S z|^2}, \quad z \in \R^{2d}, \; S \in Sp(d),
\end{equation}
comes, after proper scaling, from an ambiguity function (or Wigner distribution) of a generalized $d$-dimensional Gaussian. Furthermore, $\Phi$ is an eigenfunction of the symplectic Fourier transform with eigenvalue 1. This is in general not true for the planar Fourier transform $\F$ on $\R^{2d}$ (only for $S$ orthogonal). For more information on symplectic phase space methods we refer to \cite{Fol89}, \cite{Gos11}, \cite{Gosson_Wigner_2017}, \cite{Gro01}.

\medskip

We arrive at the symplectic version of the Poisson summation formula, which we compare to the usual Poisson summation formula. The latter is
\begin{equation}
	\sum_{\l \in \L} f(\l) = \vol(\L)^{-1} \sum_{\l^\perp \in \L^\perp} \widehat{f}(\l^\perp).
\end{equation}
Sometimes, a translation by $T_{-x}$ on the left-hand side and a Modulation by $M_x$ on the right-hand side are added:
$
	\sum_{\l \in \L} f(\l+x) = \vol(\L)^{-1}  \sum_{\l^\perp \in \L^\perp} \widehat{f}(\l^\perp) e^{2 \pi i x \cdot \l^\perp}.
$
This reflects the fact that a periodization of $f$ (with suitable conditions) amounts to a function on the torus defined by a Fourier series. We also use the (slightly more general) formula
\begin{align}
	\sum_{\l \in \L} f(\l+x)e^{2 \pi i \omega \cdot \l} & = \sum_{\l \in \L} M_\omega T_{-x} f(\l) = \vol(\L)^{-1}  \sum_{\l^\perp \in \L^\perp} \F(M_\omega T_{-x} f)(\l^\perp)\\
	& = \vol(\L)^{-1} \sum_{\l^\perp \in \L^\perp} M_x T_\omega \widehat{f}(\l^\perp) e^{-2 \pi i \omega \cdot x}.
\end{align}
In the underlying work, we only use this version of the Poisson summation formula in dimension 1 and for the (scaled) integer lattice, which is self-dual and also the only 1-dimensional lattice. In $2d$ dimensions, we can modify the Poisson summation formula by using the symplectic Fourier transform. We obtain
\begin{equation}
	\sum_{\l \in \L} F(\l + z) = \vol(\L)^{-1} \sum_{\l^\circ \in \L^\circ} (\F_\sigma F)(\l^\circ) e^{2 \pi i \sigma(\l^\circ,z)}.
\end{equation}
We will use this formula only for $d=1$ again, i.e., for 2-dimensional lattices. Assuming $\vol(\L) = 1$, and since any lattice in $\R^2$ is symplectic, we obtain
\begin{equation}
	\sum_{\l \in \L} F(\l + z) = \sum_{\l \in \L} (\F_\sigma F)(\l) e^{2 \pi i \sigma(\l,z)}, \quad \L \; \ldots \text{ a lattice in } \R^2, \, \vol(\L) = 1.
\end{equation}
In particular, since any 2-dimensional Gaussian is an eigenfunction of the symplectic Fourier transform, we obtain the functional equation for our families of theta functions;
\begin{equation}
	\theta_\L(z_\L;\alpha) = \tfrac{1}{\alpha} \, \widehat{\theta}_\L(z_\L;\tfrac{1}{\alpha}), \quad \forall z_\L \in \R^2, \, \alpha > 0.
\end{equation}

\section{A conjecture on the stability of \texorpdfstring{$z_{\L_2}^-$}{z-}}\label{sec_stability}
Our family of theta functions shows -- at least numerically, but also in our proofs -- some stability against controlled perturbation, a property which is provably false for the theta functions studied by Montgomery \cite{Mon88}. We observed that the following inequality seems to hold:
\begin{equation}\label{eq_perturb}
	\theta_{\L_2}(z_{\L_2}^-;\alpha) \geq \theta_\L(\widetilde{z}_\L;\alpha), \quad \forall \alpha > 0,
	\quad \text{ and } \quad
	\widetilde{z}_\L \to z_{\L_2}^- \text{ as } \L \to \L_2.
\end{equation}
We are not able to give general conditions on how $\widetilde{z}_\L$ must approach $z_{\L_2}^-$ as $\L$ tends to the hexagonal lattice $\L_2$ in the space of lattices, but our proofs make us believe that \eqref{eq_perturb} holds at least under the following conditions;
\begin{equation}
	\widetilde{z_1} + x \, \widetilde{z_2} = \frac{1}{2},
	\quad
	\partial_x \, \widetilde{z_2} \leq 0,
	\quad \widetilde{z}(\tfrac{1}{2},y) = a(\tfrac{1}{2},y)
	\quad \text{ and } \quad
	\widetilde{z}(0,y) \in [\tfrac{1}{4},\tfrac{1}{2}].
\end{equation}
This is under the hypothesis that $(x,y) \in D_+$. We remark that the circumcenter $a_\L$ meets the above criteria, as well as our point $b_\L$. The difference between the points $a_\L$ and $b_\L$ is that $\partial_x b_2 \equiv 0$ whereas $\partial_x a_2 < 0$ for $x \in (0, \tfrac{1}{2})$. We note that $\partial_x a_2 |_{x = \frac{1}{2}} = \partial_x b_2 |_{x = \frac{1}{2}} = 0$, for all $y \geq \sqrt{3}/2$, but we have no indication that this property is needed for the perturbation result.

\medskip

There is, however, no stability or perturbation result if we seek to minimize the maximum of $\theta_\L$. By using the representation $\widehat{\theta}_\L$, it is easy to show that
\begin{equation}
	\theta_\L(z_\L;\alpha) \leq \theta_\L(0;\alpha)
	\quad \text{ and } \quad
	\widehat{\theta}_\L(z_\L;\alpha) \leq \widehat{\theta}_\L(0;\alpha), \quad \forall \alpha > 0.
\end{equation}
We observe that
\begin{align}
	\widehat{\theta}_\L (z_\L;\alpha) \leq \sum_{\l \in \L} \left| e^{-\pi \alpha |\l|^2} e^{2 \pi i \sigma(\l,Mz)} \right| = \sum_{\l \in \L} e^{-\pi \alpha |\l|^2} = \widehat{\theta}_\L(0;\alpha).
\end{align}
Of course, 0 $(= z_\L^+)$ can be replaced by any other lattice point by periodicity. By Montgomery \cite{Mon88} we have
\begin{equation}
	\theta_{\L_2}(0;\alpha) \leq \theta_\L(0;\alpha), \quad \forall \alpha > 0.
\end{equation}
However, in general	$\theta_{\L_2}(0;\alpha)$ may exceed $\theta_\L(\widetilde{z}_\L;\alpha)$ for some $\alpha > 0$ and $\widetilde{z}_\L \to 0$ as $\L \to \L_2$ (unless $\widetilde{z}_\L \equiv 0)$. This easily follows from the underlying heat equation. By substituting $\alpha \mapsto \frac{1}{4\pi t}$, we see that $\alpha \, \theta_\L(z_\L;\alpha)$ is the fundamental solution to heat equation
\begin{equation}
	\begin{cases}
		\partial_t \, h_\L(z;t) = \Delta_z h_\L(z;t), \quad z \in \R^2, \, t > 0\\
		h_\L(z;0) = \sum_{\l \in \L} \delta_\l.
	\end{cases}
\end{equation}
Hence, for $\alpha \to \infty$ we have $\theta_{\L_2}(0;\alpha) \to \delta_0$ whereas $\theta_\L(\widetilde{z}_\L;\alpha) \to 0$. This shows that any perturbation result analogous to \eqref{eq_perturb} will eventually fail for the family $\theta_\L(0;\alpha)$.
\end{appendices}

\section*{Funding information}
L.B.\ was supported by the German Research Foundation (DFG) and the Austrian Science Fund (FWF) through the joint project FR 4083/3-1 / \href{https:doi.org/10.55776/I4354}{10.55776/I4354} during his stay in Vienna. The research of M.F.\ was supported in whole or in part by the Austrian Science Fund (FWF) projects \href{https://doi.org/10.55776/TAI6}{10.55776/TAI6}, \href{https://doi.org/10.55776/P33217}{10.55776/P33217}, and \href{https://doi.org/10.55776/PAT5102224}{10.55776/PAT5102224}. S.S.\ was partially supported by the NSF (2123224) and the Alfred P.~Sloan Foundation.

For open access purposes, the authors have applied a CC BY public copyright license to any author-accepted manuscript version arising from this submission.

\section*{Acknowledgements}

The authors thank Karlheinz Gröchenig for helpful feedback on the initial draft. The authors gratefully acknowledge the anonymous referees' valuable feedback, which improved the article's exposition.

\end{document}